\documentclass[11pt,reqno]{amsart}

\usepackage{amsmath,amssymb,amsthm}
\usepackage{bm}
\usepackage{natbib}
\usepackage[dvipdfmx]{graphicx}
\usepackage{multirow}
\usepackage{here}
\usepackage{fullpage}
\usepackage{url}
\usepackage{color}
\usepackage{mathrsfs}
\usepackage{enumitem}
\usepackage{comment}

\makeatletter

\@addtoreset{equation}{section}
\makeatletter

\newtheorem{theorem}{Theorem}[section]

\newtheorem{lemma}{Lemma}[section]
\newtheorem{proposition}{Proposition}[section]
\newtheorem{assumption}{Assumption}[section]
\theoremstyle{definition}
\newtheorem{definition}{Definition}[section]
\newtheorem{remark}{Remark}[section]
\newtheorem{example}{Example}[section]

\newcommand{\Ep}{\mathrm{E}}

\renewcommand{\tilde}{\widetilde}
\renewcommand{\hat}{\widehat}

\DeclareMathOperator{\Var}{Var}

\DeclareMathOperator{\E}{E}
\DeclareMathOperator{\Prob}{P}

\DeclareMathOperator{\supp}{supp}

\DeclareMathOperator{\relu}{ReLU}

\newcommand{\argmin}{\mathop{\rm arg~min}\limits}

\def\ben#1{\begin{equation}#1\end{equation}}

\def\bm#1{\begin{multline*}#1\end{multline*}}

\def\ba#1{\begin{align*}#1\end{align*}}

\newcommand{\eps}{\varepsilon}

\newcommand{\bol}[1]{\boldsymbol{#1}}
\newcommand{\mcl}[1]{\mathcal{#1}}

\newcommand{\fun}[1]{\texttt{#1}}
\newcommand{\pck}[1]{{\bfseries#1}}

\newcommand{\bra}[1]{\left(#1\right)}

\allowdisplaybreaks

\begin{document}

\title[]{Adaptive deep learning for nonlinear time series models}
\thanks{D. Kurisu is partially supported by JSPS KAKENHI Grant Numbers 20K13468. and 23K12456.  
Y. Koike is partially supported by JST CREST Grant Number JPMJCR2115 and JSPS KAKENHI Grant Number 19K13668. 
The authors would like to thank Emmanuel Gobet, Stefan Mittnik, Taiji Suzuki, and Yoshikazu Terada for their helpful comments and suggestions. 
} 

\author[D. Kurisu]{Daisuke Kurisu}
\author[R. Fukami]{Riku Fukami}
\author[Y. Koike]{Yuta Koike}

\date{First version: January 17, 2022. This version: \today}

\address[D. Kurisu]{Center for Spatial Information Science, The University of Tokyo, 
5-1-5, Kashiwanoha, Kashiwa-shi, Chiba 277-8568, Japan.
}
\email{daisukekurisu@csis.u-tokyo.ac.jp}

\address[R. Fukami]{Graduate School of Mathematical Science, The University of Tokyo\\
3-8-1 Komaba, Meguro-ku, Tokyo 153-8914, Japan.
}
\email{rick.h.azuma@gmail.com}

\address[Y. Koike]{Graduate School of Mathematical Science, The University of Tokyo\\
3-8-1 Komaba, Meguro-ku, Tokyo 153-8914, Japan.
}
\email{kyuta@ms.u-tokyo.ac.jp}

\begin{abstract}
In this paper, we develop a general theory for adaptive nonparametric estimation of the mean function of a nonstationary and nonlinear time series model using deep neural networks (DNNs). We first consider two types of DNN estimators, non-penalized and sparse-penalized DNN estimators, and establish their generalization error bounds for general nonstationary time series. We then derive minimax lower bounds for estimating mean functions belonging to a wide class of nonlinear autoregressive (AR) models that include nonlinear generalized additive AR, single index, and threshold AR models. Building upon the results, we show that the sparse-penalized DNN estimator is adaptive and attains the minimax optimal rates up to a poly-logarithmic factor for many nonlinear AR models. Through numerical simulations, we demonstrate the usefulness of the DNN methods for estimating nonlinear AR models with intrinsic low-dimensional structures and discontinuous or rough mean functions, which is consistent with our theory.
\end{abstract}


\maketitle

\section{Introduction}

Motivated by the great success of deep neural networks (DNNs) in several applications such as pattern recognition and natural language processing, there has been an increasing interest in revealing the reason why DNNs work well from the statistical point of view. In the past few years, many researchers have contributed to understand theoretical advantages of DNN estimates for nonparametric regression models. See, for example, \cite{BaKo19}, \cite{ImFu19}, \cite{Sc19, Sc20}, \cite{Su19}, \cite{HaSu20}, \cite{NaIm20}, \cite{KoLa21}, \cite{SuNi21}, \cite{TsSu21},  and references therein.

In contrast to the recent progress of DNNs, theoretical results on statistical properties of DNN methods for stochastic processes are scarce. As exceptional studies, we refer to \cite{PhRi20}, \cite{KoKr20}, \cite{Og21}, and \cite{OgKo21}. \cite{PhRi20} investigates forecasting ability of feed-forward DNNs for stationary processes and derive bounds for an expected forecast error. \cite{KoKr20} consider a time series prediction problem and investigate the convergence rate of a deep recurrent neural network estimate. \cite{Og21} considers DNN estimation for the diffusion matrices and studies their estimation errors as misspecified parametric models. \cite{OgKo21} investigate nonparametric drift estimation of a multivariate diffusion process. Notably, there seem no theoretical results on the statistical properties of feed-forward DNN estimators for nonparametric estimation of the mean functions of nonlinear and possibly nonstationary time series models.

The goal of this paper is to develop a general theory for adaptive nonparametric estimation of the mean function of a nonlinear time series using DNNs. The contributions of this paper are as follows.

First, we provide bounds of (i) generalization error (Lemma C.1) and (ii) expected empirical error (Lemma C.2) of general estimators of the mean function of a nonlinear and nonstationary $\beta$-mixing time series. We note that Lemma C.1 allows the $\beta$-mixing coefficients to decay both polynomially and exponentially fast and are of independent theoretical interest since they can be useful to investigate asymptotic properties of nonparametric estimators including DNNs. Building upon the results, we establish a generalization error bound of non-penalized DNN estimators (Theorem \ref{thm: pred-error-bound}) with a $C$-Lipschitz activation function (e.g., rectified linear unit (ReLU), LeakyReLU, sigmoid, and softplus). 

Second, we consider a sparse-penalized DNN estimator which is defined as a minimizer of an empirical risk with a sparse penalty and develop its asymptotic properties. In particular, we establish a generalization error bound of the sparse penalized DNN estimator (Theorem \ref{thm: pred-error-bound2}) with a $C$-Lipschitz activation function when the observations are $\beta$-mixing and can be nonstationary. While the result is shown under the condition that the $\beta$-mixing coefficients decay exponentially fast, it is straightforward to extend it to the case that the $\beta$-mixing coefficients decay polynomially
fast. We note that our conditions on the penalty function cover several examples such as the clipped $L_1$ penalty \citep{Zh10}, the SCAD penalty \citep{FaLi01}, the minimax concave penalty \citep{Zh10mcp}, and the seamless $L_0$ penalty \citep{DHL13}, and the generalization error bound enables us to estimate mean functions of nonlinear time series models adaptively. Our work can be viewed as extensions of the results in \cite{Sc20} and \cite{OhKi22} for independent observations to nonstationary time series. From the technical point of view, our analysis is related to the strategy in \cite{Sc20}. Due to the existence of temporal dependence, the extensions are nontrivial and we achieve this by developing a new strategy to obtain generalization error bounds for dependent data using a blocking technique for $\beta$-mixing processes and exponential inequalities for self-normalized martingale difference sequences. It shall be noted that our approach is also quite different from that of \cite{OhKi22} since their approach strongly depends on the independence of observations and our generalization error bounds of the sparse penalized DNN estimator improve the power of the logs in their bounds. More detailed differences are discussed in Section \ref{subsec:pdnn} and the supplementary material. Our approach to deriving generalization error bounds paves a way to new techniques for studying statistical properties of machine learning methods for more richer classes of models for dependent data including time series and spatial data.

Third, 
we establish that the sparse-penalized DNN estimators achieve minimax rates of convergence up to a poly-logarithmic factor over a wide class of nonlinear AR($d$) processes including generalized additive AR models and functional coefficient AR models introduced in \cite{ChTs93} that allow discontinuous mean functions. When the mean function belongs to a class of suitably smooth functions (e.g., H\"older space), one can use other nonparametric estimators for adaptively estimating the mean function (see \cite{Ho99} for example). Similar assumptions on the smoothness of the mean functions have been made in most papers that investigate nonparametric estimation of the mean functions of nonlinear and stationary time series models (\cite{Ro83}, \cite{Tr93}, \cite{Tr94}, \cite{Ma96a, Ma96b}, \cite{Ho99}, \cite{FaYa03}, \cite{ZhWu08}, \cite{Ha08}, \cite{LiWu10}). \cite{Vo12} and \cite{ZhWu15} investigate nonparametric regression of locally stationary (i.e., nonstationary) models. The authors derive uniform convergence rates of kernel estimators of general smooth mean functions. \cite{Vo12} also considers nonparametric estimation of additive mean functions. The generalization error bounds (Theorems \ref{thm: pred-error-bound} and \ref{thm: pred-error-bound2}) in this paper can be applied to nonstationary time series models with more general mean functions that include those considered in \cite{Vo12} and \cite{ZhWu15}. However, the methods in those papers cannot be applied for estimating nonlinear time series models with possibly discontinuous mean functions. Our results show that the sparse-penalized DNN estimation is a unified method for adaptively estimating both smooth and discontinuous mean functions of time series regression models. Further, we shall note that the sparse-penalized DNN estimators attain the parametric rate of convergence up to a logarithmic factor when the mean functions belong to an $\ell^0$-bounded affine class that include (multi-regime) threshold AR processes (Theorems \ref{thm:minimax-l0} and \ref{thm:rate-l0}).

In addition to the theoretical results, we also conduct simulation studies to investigate the finite sample performance of the DNN estimators. We find that the DNN methods work well for the models with (i) intrinsic low-dimensional structures and (ii) discontinuous or rough mean functions. These results are consistent with our main results.

To summarize, this paper contributes to the literature on nonparametric estimation of nonlinear and nonstationary time series by establishing (i) the theoretical validity of non-penalized and sparse-penalized DNN estimators for the adaptive nonparametric estimation of mean functions of nonlinear time series models and (ii) show the optimality of the sparse-penalized DNN estimator for a wide class of nonlinear AR processes.

The rest of the paper is organized as follows. In Section \ref{sec:setting}, we introduce nonparametric regression models considered in this paper and demonstrate that they cover a range of nonlinear time series models. In Section \ref{sec:main}, we provide generalization error bounds of (i) the non-penalized and (ii) the sparse-penalized DNN estimators. In Section \ref{sec:ar-model}, we present the minimax optimality of the sparse-penalized DNN estimators and show that the estimators achieve the minimax optimal convergence rate up to a logarithmic factor over (i) composition structured functions and (ii) $\ell^0$-bounded affine classes. In Section \ref{sec:simulation}, we provide simulation results. Section \ref{sec: conclusion} concludes and discusses possible extensions. Proofs for Section \ref{sec:main} are given in Appendix A. The supplementary material includes a discussion of our main results (Section B), auxiliary lemmas (Section C), proofs for Section 4 (Section D), and technical tools (Section E).

\subsection{Notations}
For any $a,b \in \mathbb{R}$, we write $a \vee b=\max\{a,b\}$ and $a \wedge b=\min\{a,b\}$. For $x \in \mathbb{R}$, $\lfloor x \rfloor$ denotes the largest integer $\leq x$. Given a function $f$ defined on a subset of $\mathbb R^d$ containing $[0,1]^d$, we denote by $f|_{[0,1]^d}$ the restriction of $f$ to $[0,1]^d$. When $f$ is real-valued, we write $\|f\|_{\infty}:=\sup_{x \in [0,1]^{d}}|f(x)|$ for the supremum on the compact set $[0,1]^{d}$. Also, let $\supp(f)$ denote the support of the function $f$. For a vector or matrix $W$, we write $|W|$ for the Frobenius norm (i.e. the Euclidean norm for a vector), $|W|_{\infty}$ for the maximum-entry norm and $|W|_0$ for the number of non-zero entries. For any positive sequences $a_n, b_n$, we write $a_n \lesssim b_n$ if there is a positive constant $C>0$ independent of $n$ such that $a_n \leq Cb_n$ for all $n$, $a_n \asymp b_n$ if $a_n \lesssim b_n$ and $b_n \lesssim a_n$.

\section{Settings}\label{sec:setting}

Let $(\Omega, \mathcal{G}, \{\mathcal{G}_{t}\}_{t\geq 0},\Prob)$ be a filtered probability space. Consider the following nonparametric time series regression model: 
\begin{align}\label{NLAR-model}
Y_t &= m(X_t) + \eta(X_t)v_t,\ t=1,\dots, T, 
\end{align}
where $T \geq 3$, $(Y_t, X_t) \in \mathbb{R} \times \mathbb{R}^{d}$, and $\{X_t, v_t\}_{t=1}^{T}$ is a sequence of random vectors adapted to the filtration $\{\mathcal{G}_t\}_{t=1}^{T}$. We assume $C_\eta := \sup_{x \in [0,1]^{d}}|\eta(x)|<\infty$. In this paper we investigate nonparametric estimation of the mean function $m$ on the compact set $[0,1]^{d}$, that is, $f_0 := m\mathbf{1}_{[0,1]^{d}}$. The model (\ref{NLAR-model}) covers a range of nonlinear time series models.

\begin{example}[Nonlinear AR($p$)-ARCH($q$) model]\label{Ex:ar-arch}
Consider a nonlinear AR model:
\begin{align*}
Y_t &= \tilde m(Y_{t-1},\dots, Y_{t-p}) + (\gamma_0 + \gamma_1Y_{t-1}^2+\dots+\gamma_qY_{t-q}^2)^{1/2}v_t, 
\end{align*}
where $\gamma_0>0$, $\gamma_i \geq 0$, $i=1,\dots,q$ with $1 \leq p,q \leq d$. This example corresponds to the model (\ref{NLAR-model}) with $X_t = (Y_{t-1},\dots, Y_{t-d})'$, $m(x_1,\dots,x_d)=\tilde m(x_1,\dots, m_p)$ and $\eta(x_1,\dots,x_d) =  (\gamma_0 + \gamma_1 x_1^2 + \dots + \gamma_q x_q^2)^{1/2}$. 
\end{example}

\begin{example}[Multivariate nonlinear time series]\label{Ex:multi-ar}
Consider the case that we observe multivariate time series $\{\mathbf Y_t = (Y_{1,t},\dots, Y_{p,t})'\}_{t = 1}^{T}$ and $\{\mathbf X_t = (X_{1,t},\dots, X_{d,t})'\}_{t = 1}^{T}$ such that 
\begin{align}\label{multi-ar-reg}
Y_{j,t} &= m_{j}(\mathbf X_t) + \eta_j(\mathbf X_t)v_{j,t},\ j=1,\dots, p.
\end{align}
The model (\ref{multi-ar-reg}) corresponds to (i) multivariate nonlinear AR model when $\mathbf X_t = (\mathbf Y'_{t-1}, \dots, \mathbf Y'_{t-q})'$ for some $q \geq 1$ and (ii) multivariate nonlinear time series regression with exogenous variables when $\eta_j(\cdot) = 1$ and $\{\mathbf X_t\}_{t=1}^{T}$ is uncorrelated with $\{\mathbf v_t = (v_{1,t},\dots,v_{p,t})'\}_{t=1}^{T}$. If one is interested in estimating the mean function $\mathbf m = (m_1,\dots,m_p)': \mathbb{R}^{d} \to \mathbb{R}^{p}$, then it is enough to estimate each component $m_j$. In this case, the problem of estimating $m_j$ is reduced to that of estimating the mean function $m$ of the model (\ref{NLAR-model}). 
\end{example}

\begin{example}[Time-varying nonlinear models]\label{Ex:tv-model}
Consider a nonlinear time-varying model:
\begin{align}\label{tv-model}
Y_t &= m\left({t \over T}, Y_{t-1},\dots, Y_{t-p}\right) + \eta\left({t \over T}, Y_{t-1},\dots, Y_{t-q}\right)v_t,
\end{align}
where $1 \leq p,q \leq d-1$. This example corresponds to the model (\ref{NLAR-model}) with $X_t = (t/T,Y_{t-1},\dots, Y_{t-d+1})'$ as well as $m$ and $\eta$ regarded as functions on $\mathbb R^d$ in the canonical way. If the random variables $v_t$ are i.i.d., then the model (\ref{tv-model}) corresponds to that considered in \cite{Vo12}. Moreover, the model (\ref{tv-model}) covers, for instance, time-varying AR($p$)-ARCH($q$) models when $m(u,x_1,\dots,x_{p}) = m_0(u)+\sum_{j=1}^{p}m_j(u)x_j$ and $\eta(u,x_1,\dots,x_{q})=(\eta_0(u)+\sum_{j=1}^{q}\eta_j(u)x_j^2)^{1/2}$ with some functions $m_j:[0,1] \to \mathbb{R}$, $\eta_j:[0,1] \to [0,\infty)$. By the same approach as in Example \ref{Ex:multi-ar}, the model (\ref{NLAR-model}) can be applied to multivariate time-varying nonlinear models.
\end{example}


\section{Main results}\label{sec:main}

In this section, we provide generalization error bounds of (i) the non-penalized and (ii) the sparse-penalized DNN estimators. First, we define the $\beta$-mixing coefficients of the possibly nonstationary process $\{X_t\}_{t=1}^{T}$. Recall that the process $\{X_t\}_{t=1}^{T}$ is defined on a filtered probability space $(\Omega, \mathcal{G}, \{\mathcal{G}_{t}\}_{t\geq 0},\Prob)$. Let $\mathcal{A}$ and $\mathcal{B}$ be subfields of $\mathcal{G}$. Define $\beta(\mathcal{A},\mathcal{B})=\sup{1 \over 2}\sum_{i=1}^{I}\sum_{j=1}^{J}|\Prob(A_i \cap B_j)-\Prob(A_i)\Prob(B_j)|$ where the supremum is taken over all pairs of (finite) partitions $\{A_1,\dots,A_I\}$ and $\{B_1,\dots,B_J\}$ of $\Omega$ such that $A_i \in \mathcal{A}$ and $B_j \in \mathcal{B}$. The $\beta$-mixing coefficients of the process $\{X_t\}_{t=1}^{T}$ is defined as $\beta_X(t):=\beta_{X,T}(t)=\sup_{1\leq r \leq T-t}\beta(\sigma(X_s:1\leq s \leq r),\sigma(X_s:r+t\leq s \leq T))$, where $\sigma(Z)$ is the $\sigma$-field generated by $Z$ (cf. \cite{Vo12}).
We assume the following conditions.

\begin{assumption}\label{Ass: model}
\begin{itemize}
\item[(i)] The random variables $v_t$ are conditionally centered and sub-Gaussian, that is, $\Ep[v_t\mid\mathcal G_{t-1}]=0$ and $\Ep[\exp(v_t^2/K_t^2)|\mathcal{G}_{t-1}] \leq 2$ for some constant $K_t>0$. Moreover, $\Ep[v_t^2|\mathcal{G}_{t-1}]=1$. Define $K = \max_{1 \leq t \leq T}K_t$. 
\item[(ii)] The process $X=\{X_t\}_{t = 1}^{T}$ is exponentially $\beta$-mixing, i.e. the $\beta$-mixing coefficient $\beta_X(t)$ of $X$ satisfies $\beta_X(t) \leq C_{1,\beta}\exp(-C_{2,\beta}t)$ with some positive constants $C_{1,\beta}$ and $C_{2,\beta}$ for all $t \geq 1$. 
\item[(iii)] The process $X$ is predictable, that is, $X_t$ is measurable with respect to $\mathcal{G}_{t-1}$. 
\end{itemize}
\end{assumption}

Condition (i) is used to apply exponential inequalities for self-normalized processes presented in \cite{deKlLa04}. Since $\Ep[\Ep[\exp(v_t^2/K_t^2)|\mathcal{G}_{t-1}]]=\Ep[\exp(v_t^2/K_t^2)]$, Condition (i) also implies that each $v_t$ is sub-Gaussian. Condition (ii) is satisfied for a wide class of nonlinear time series. 
Note that the process $X = \{X_t\}_{t = 1}^{T}$ can be nonstationary. When $X_t = (Y_{t-1},\dots,Y_{t-d})'$, \cite{ChCh00} provide a set of sufficient conditions for the process $X$ to be strictly stationary and exponentially $\beta$-mixing (Theorem 1 in \cite{ChCh00}): 
\begin{itemize}
\item[(i)] $\{v_t\}$ is a sequence of i.i.d. random variables and has an everywhere positive and continuous density function, $E[v_t]=0$, and $v_t$ is independent of $X_{t-s}$ for all $s \geq 1$. 
\item[(ii)] The function $m$ is bounded on every bounded set, that is, for every $\Gamma \geq 0$, $\sup_{|x| \leq \Gamma}|m(x)| < \infty$. 
\item[(iii)] The function $\eta$ satisfies, for every $\Gamma \geq 0$, $0<\eta_1 \leq \inf_{|x|\leq \Gamma}\eta(x)\leq \sup_{|x|\leq \Gamma}\eta(x) <\infty$,
where $\eta_1$ is a constant. 
\item[(iv)] There exist constants $c_{m,i} \geq 0$, $c_{\eta,i} \geq 0$ ($i=0,\dots,d$) and $M > 0$ such that $|m(x)| \leq c_{m,0} + \sum_{i=1}^{d}c_{m,i}|x_i|$ and $\eta(x) \leq c_{\eta,0} + \sum_{i=1}^{d}c_{\eta,i}|x_i|$ for $|x| \geq M$, and $\sum_{i=1}^{d}(c_{m,i} + c_{\eta,i}E[|v_1|])<1$.
\end{itemize} 
We also refer to \cite{Tj90}, \cite{BaLe95}, \cite{LuJi01}, \cite{ClPu04} and \cite{Vo12} for other sufficient conditions for the process $X$ being strictly or locally stationary and exponentially $\beta$-mixing. 

\begin{remark}
Although the (exponential) $\beta$-mixing condition does not directly imply any kind of stationarity (as far as we know), for Markov processes it is implied by asymptotic stationarity (cf.~Proposition 3 of \cite{Li05}), and this sufficient condition is practically used to check the $\beta$-mixing condition. 
To our knowledge, locally stationary processes are the only known (non-artificial) examples satisfying the $\beta$-mixing condition without asymptotic stationarity. 
In particular, this condition seems to exclude explosive/unit-root cases, although we do not know any formal result in this direction.
\end{remark}

\subsection{Deep neural networks}

To estimate the mean function $m$ of the model (\ref{NLAR-model}), we fit a deep neural network (DNN) with a nonlinear activation function $\sigma : \mathbb{R} \to \mathbb{R}$. The network architecture $(L,\mathbf{p})$ consists of a positive integer $L$ called the \textit{number of hidden layers} or \textit{depth} and a \textit{width vector} $\mathbf{p}=(p_0,\dots,p_{L+1}) \in \mathbb{N}^{L+2}$. A DNN with network architecture $(L, \mathbf{p})$ is then any function of the form
\begin{align}\label{DNN-func}
f:\mathbb{R}^{p_0} \to \mathbb{R}^{p_{L+1}},\ x \mapsto f(x) = A_{L+1} \circ \sigma_{L} \circ A_{L} \circ \sigma_{L-1}  \circ \dots \circ \sigma_{1} \circ A_{1}(x),
\end{align}
where $A_\ell :\mathbb{R}^{p_{\ell-1}} \to \mathbb{R}^{p_{\ell}}$ is an affine linear map defined by $A_{\ell}(x) := W_{\ell}x + \mathbf{b}_\ell$ for given $p_{\ell-1} \times p_{\ell}$ weight matrix $W_\ell$ and a shift vector $\mathbf{b}_\ell \in \mathbb{R}^{p_\ell}$, and $\sigma_{\ell}: \mathbb{R}^{p_\ell} \to \mathbb{R}^{p_\ell}$ is an element-wise nonlinear activation map defined as $\sigma_{\ell}(z):=(\sigma(z_1),\dots,\sigma(z_{p_{\ell}}))'$. We assume that the activation function $\sigma$ is $C$-Lipschitz for some $C>0$, that is, there exists $C>0$ such that $|\sigma(x_1) - \sigma(x_2)| \leq C|x_1 - x_2|$ for any $x_1,x_2 \in \mathbb{R}$. Examples of $C$-Lipschitz activation functions include the rectified linear unit (ReLU) activation function $x \mapsto \max\{x,0\}$ and the sigmoid activation function $x \mapsto 1/(1+e^{-x})$. For a neural network of the form (\ref{DNN-func}), we define $\theta(f) := (\text{vec}(W_1)',\mathbf{b}'_1,\dots,\text{vec}(W_{L+1})',\mathbf{b}'_{L+1})'$ where $\text{vec}(W)$ transforms the matrix $W$ into the corresponding vector by concatenating the column vectors. 

We let $\mathcal{F}_{\sigma,p_0,p_{L+1}}$ be the class of DNNs which take $p_0$-dimensional input to produce $p_{L+1}$-dimensional output and use the activation function $\sigma:\mathbb{R}\to \mathbb{R}$. Since we are interested in real-valued function on $\mathbb{R}^d$, we always assume that $p_0=d$ and $p_{L+1}=1$ in the following. 

For a given DNN $f$, we let $\text{depth}(f)$ denote the depth and $\text{width}(f)$ denote the width of $f$ (i.e. $\text{width}(f)=\max_{1 \leq \ell \leq L}p_{\ell}$). For positive constants $L, N, B$, and $F$, we set $\mathcal{F}_{\sigma}(L,N,B):=\{f\in \mathcal{F}_{\sigma,d,1}: \text{depth}(f) \leq L, \text{width}(f) \leq N,|\theta(f)|_{\infty}\leq B\}$ and
\ben{\label{DNN-func-class}
\mathcal{F}_{\sigma}(L,N,B,F) :=\left\{f\mathbf{1}_{[0,1]^{d}}:f\in\mathcal{F}_{\sigma}(L,N,B),\|f\|_{\infty} \leq F\right\}.
}
Moreover, we define a class of sparsity constrained DNNs with sparsity level $S>0$ by 
\begin{align}\label{DNN-func-class-sparse}
\mathcal{F}_{\sigma}(L,N,B,F,S) &:= \left\{f \in \mathcal{F}_{\sigma}(L,N,B,F): |\theta(f)|_0 \leq S \right\}. 
\end{align}

\subsection{Non-penalized DNN estimator}\label{subsec:NPDNN}

Let $\hat{f}_{T}$ be an estimator which is a real-valued random function on $\mathbb{R}^{p}$ such that the map $(\omega, x) \mapsto \hat{f}_{T}(\omega,x)$ is measurable with respect to the product of the $\sigma$-field generated by $\{Y_t,X_t\}_{t=1}^{T}$ and the Borel $\sigma$-field of $\mathbb{R}^{d}$. In this section, we provide finite sample properties of a DNN estimator $\hat{f}_T \in \mathcal{F}_{\sigma}(L,N,B,F,S)$ of $f_0$. 

In particular, we provide bounds for the generalization error
\begin{align*}
R(\hat{f}_T,f_0)= \Ep\left[{1 \over T}\sum_{t=1}^{T}(\hat{f}_{T}(X_t^{\ast}) - f_0(X_t^{\ast}))^2\right],
\end{align*}
where $\{X_t^{\ast}\}_{t =1}^{T}$ is an independent copy of $X$. 
\begin{remark}
    When $\{X_t\}_{t\geq1}$ has the common marginal distribution $\Pi$, i.e.~$X_t\sim\Pi$ for all $t\geq1$, we can rewrite $R(\hat{f}_T, f_0)$ as
\[
R(\hat{f}_T, f_0)=\E\left[\int_{\mathbb R^d}|\hat f_T(x)-f_0(x)|^2\Pi(dx)\right],
\]
so it can be interpreted as the $L^2$-risk of the estimator $\hat f_T$ with respect to $\Pi$. 
In this case, we can also relate it to the so-called path-dependent generalization error using the $\beta$-mixing coefficient of $\{X_t,Y_t\}_{t\geq1}$ as follows. First, the path-dependent generalization error is introduced in \cite{KuMo17} and defined as
\[
\widetilde R_{s}(\hat{f}_T, f_0):=\E\left[\bra{\hat f_T(X_{T+s})-f_0(X_{T+s})}^2\mid \{Y_t,X_t\}_{t=1}^T\right],
\]
where $s\geq1$ is a fixed integer. This quantity can be interpreted as an $s$-ahead forecast error of $\hat f_T$ given observation data. 
Now, by Lemma E.1, we can construct a random vector $X_{T+s}^*$ in $\mathbb R^d$ independent of $\{Y_t,X_t\}_{t=1}^T$ such that $X_{T+s}^*\overset{d}{=}X_{T+s}(\sim\Pi)$ and $P(X_{T+s}^*\neq X_{T+s})=\beta_Z(s)$, where $\beta_Z$ is the $\beta$-mixing coefficient of the process $Z_t=(X_t,Y_t)$, $t\geq1$. Then we have
\begin{align*}
\widetilde R_{s}(\hat{f}_T, f_0)
&\leq \E\left[\bra{\hat f_T(X^*_{T+s})-f_0(X^*_{T+s})}^2\mid \{Y_t,X_t\}_{t=1}^T\right]+4F^2\beta_Z(s)\\
&=R(\hat f_T,f_0)+4F^2\beta_Z(s),
\end{align*}
provided that $\|\hat f_T\|_\infty\vee\|f_0\|_\infty\leq F$ for some constant $F$. In particular, if the process $Z_t$ is $\beta$-mixing in the sense that $\beta_Z(s)\to0$ as $s\to\infty$, then a bound for $R(\hat f_T,f_0)$ gives sufficiently far ahead forecast error bounds. 
  
\end{remark}

Let $\mathcal{F}$ be a pointwise measurable class of real-valued functions on $\mathbb{R}^{d}$ (cf. Example 2.3.4 in \cite{vaWe96}). 
Define $\Psi_T^{\mathcal{F}}(\hat{f}_T) := \Ep\left[Q_T(\hat{f}_T) - \inf_{\bar{f} \in \mathcal{F}}Q_T(\bar{f})\right]$ where $Q_T(f)$ is the empirical risk of $f$ defined by $Q_T(f) := {1 \over T}\sum_{t=1}^{T}(Y_t - f(X_t))^2$. The function $\Psi_T^{\mathcal{F}}(\hat{f}_T)$ measures a gap between $\hat{f}_T$ and an exact minimizer of $Q_T(f)$ subject to $f \in \mathcal{F}$. Define
\begin{align*}
\hat{f}_{T,np} \in \argmin_{f \in \mathcal{F}_{\sigma}(L,N,B,F,S)}Q_T(f) 
\end{align*}
and we call $\hat{f}_{T,np}$ as the non-penalized DNN estimator.

The next result gives a generalization error bound of $\hat{f}_{T,np}$.

\begin{theorem}\label{thm: pred-error-bound}
Suppose that Assumption \ref{Ass: model} is satisfied. Consider the nonparametric time series regression model (\ref{NLAR-model}) with unknown regression function $m$ satisfying $\|f_0\|_{\infty} \leq F$ where $f_0 = m\mathbf{1}_{[0,1]^d}$ for some $F \geq 1$. Let $\hat{f}_T$ be any estimator taking values in the class $\mathcal{F}=\mathcal{F}_{\sigma}(L,N,B,F,S)$ with $B \geq 1$. Then for any $\rho >1$, there exists a constant $C_\rho$, only depending on $(C_\eta, C_{1,\beta}, C_{2,\beta}, K, \rho)$, such that
\begin{align*}
R(\hat{f}_T,f_0) &\leq \rho\left(\Psi_T^{\mathcal{F}}(\hat{f}_T) + \inf_{f \in \mathcal{F}}R(f, f_0)\right) + C_\rho F^2{S(L+1)\log \left((L+1)(N+1)BT\right)(\log T) \over T}.
\end{align*}
\end{theorem}

Theorem \ref{thm: pred-error-bound} is an extension of Theorem 2 in \cite{Sc20} to possibly nonstationary $\beta$-mixing sequence and the process $\{v_t\}$ can be non Gaussian and dependent. The result follows from Lemmas C.1 and C.2 in the supplementary material. Note that these Lemmas are of independent interest since they are general results so that the estimator $\hat{f}_{T}$ do not need to take values in $\mathcal{F}_{\sigma}(L,N,B,F,S)$. Hence the results would be useful to investigate generalization error bounds of other nonparametric estimators.

Let $f_0=m\mathbf{1}_{[0,1]^d}$ belong to composition structured functions $\mcl F_0 = \mcl G\big(q, \mathbf d, \mathbf t, \bol\beta, A\big)$ for example (see Section \ref{subsec:comp} for the definition). By choosing $\sigma(x)=\max\{x,0\}$ and the parameters of $\mathcal{F}_{\sigma}(L_T,N_T,B_T,F,S_T)$ as $L_T\asymp \log T$, $N_T\asymp T$, $B_T \geq 1$, $F > \|f_0\|_\infty$, and $S_T\asymp T^{\kappa/(\kappa+1)}\log T$ with $\kappa$ depending on $\mathbf t$ and $\bol\beta$, one can show that the non-penalized DNN estimator achieves the minimax convergence rate over $\mathcal{F}_0$ up to a logarithmic factor. However, the sparsity level $S_T$ depends on the characteristics $\mathbf t$ and $\bol\beta$ of $f_0$. Therefore, the non-penalized DNN estimator is not adaptive since we do not know the characteristics in practice. In the next subsection, we provide a generalization error bound of sparse-penalized DNN estimators which plays an important role to show that the sparse-penalized DNN estimators can estimate $f_0$ adaptively.

\begin{remark}[Generalization error bound under $\beta$-mixing coefficients with polynomial decay]
We can also give a generalization error bound of the non-penalized DNN estimator $\hat{f}_{T,np}$ under $\beta$-mixing coefficients with polynomial decay. Instead of Assuption \ref{Ass: model} (ii), assume that $\beta_{X}(t) \leq C_\beta t^{-\alpha}$ for some $\alpha>0$. Then under the same assumptions in Theorem \ref{thm: pred-error-bound}, we have
\begin{align}
R(\widehat{f}_T,f_0) &\leq \rho\left(\Psi^{\mathcal{F}}(\widehat{f}_T) + \inf_{f \in \mathcal{F}}R(f, f_0)\right) +  \overline{C}F^2\left({(\log T)(\log \mathcal{N}_T) \over T} + \left({\log \mathcal{N}_T \over T}\right)^{{\alpha \over \alpha +1}}\right), \label{beta-poly-bound}
\end{align}
where $\mathcal{N}_T$ depends on the $T^{-1}$-covering number of $\mathcal{F}_\sigma(L,N,B,F,S)$ with respect to $\|\cdot\|_\infty$ and we can see that $\log \mathcal{N}_T \leq c_1S(L+1)\log \left((L+1)(N+1)BT\right)$ (see Appendix A and Section E of the supplementary material for the detailed definition the covering number and the bound of $\log \mathcal{N}_T$), respectively. Here, $c_1$ is a universal constant and $\overline{C}$ is a constant depending only on $(C_\eta, C_\beta, K, \rho)$. A similar generalization error bound can be derived for the sparse-penalized DNN, which is defined in the next subsection (see Remarks A.1 and A.2 in Appendix A for details). We speculate the bound \eqref{beta-poly-bound} would be suboptimal (at least) when $\{X_t\}_{t=1}^T$ and $\{v_t\}_{t=1}^T$ are independent. 
On this point, we refer to \cite{KuLo11} that investigates the convergence rate of the Nadaraya--Watson estimator for nonparametric time series regression models with serially correlated covariates and errors. See Section B of the supplementary material for a more detailed discussion on this point.
\end{remark}

\subsection{Sparse-penalized DNN estimator}\label{subsec:pdnn}

Define $\bar{Q}_T(f)$ as a penalized version of the empirical risk $\bar{Q}_T(f) := {1 \over T}\sum_{t=1}^{T}(Y_t - f(X_t))^2 + J_T(f)$ where $J_T(f)$ is the a sparse penalty of the form $J_T(f)=\sum_{j=1}^p\pi_{\lambda_T,\tau_T}(\theta_j(f))$ with a function $\pi_{\lambda_T,\tau_T}:\mathbb R\to[0,\infty)$ having two tuning parameters $\lambda_T>0$ and $\tau_T>0$. 
Here, $p$ is the length of the vector $\theta(f)$ and $\theta_j(f)$ denotes the $j$-th component of $\theta(f)$. 
We assume that $\pi_{\lambda_T,\tau_T}$ satisfies the following conditions:
\begin{enumerate}[label=(\roman*)]
    \item $\pi_{\lambda_T,\tau_T}(0)=0$ and $\pi_{\lambda_T,\tau_T}(\theta)$ is non-decreasing in $|\theta|$.  
    \item $\pi_{\lambda_T,\tau_T}(\theta)=\lambda_T$ if $|\theta|>\tau_T$.
\end{enumerate}
A prominent example is the clipped $L_1$ penalty of \cite{Zh10} which is given by 
\begin{equation}\label{eq:clip}
    \pi_{\lambda_T,\tau_T}(\theta)=\lambda_T\left({|\theta| \over \tau_T} \wedge 1\right).
\end{equation}
This choice is used in \cite{OhKi22}. Other possible choices are the SCAD penalty \citep{FaLi01}, the minimax concave penalty \citep{Zh10mcp} and the seamless $L_0$ penalty \citep{DHL13}.  
In this section, we provide finite sample properties of the sparse-penalized DNN estimator defined as 
\begin{align*}
\hat{f}_{T, sp} \in \argmin_{f \in \mathcal{F}_{\sigma}(L,N,B,F)}\bar{Q}_T(f). 
\end{align*}
Further, for any estimator $\hat{f}_{T} \in \mathcal{F}=\mathcal{F}_{\sigma}(L,N,B,F)$ of $f_0$, we define \\$\bar{\Psi}_T^{\mathcal{F}}(\hat{f}_T) := \Ep\left[\bar{Q}_T(\hat{f}_T) - \inf_{\bar{f} \in \mathcal{F}}\bar{Q}_T(\bar{f})\right]$.

The next result provides a generalization error bound of the sparse-penalized DNN estimator. 

\begin{theorem}\label{thm: pred-error-bound2}
Suppose that Assumption \ref{Ass: model} is satisfied. Consider the nonparametric time series regression model (\ref{NLAR-model}) with unknown regression function $m$ satisfying $\|f_0\|_{\infty} \leq F$ where $f_0 = m\mathbf{1}_{[0,1]^d}$ for some $F \geq 1$. Let $\hat{f}_T$ be any estimator taking values in the class $\mathcal{F}=\mathcal{F}_{\sigma}(L_T,N_T,B_T,F)$ where $L_T, N_T$, and $B_T$ are positive values such that $L_T \leq C_L\log^{\nu_0} T$, $N_T \leq C_N T^{\nu_1}$, $1 \leq B_T \leq C_B T^{\nu_2}$ for some positive constants $C_L, C_N, C_B, \nu_0,\nu_1$, and $\nu_2$. Moreover, we assume that the tuning parameters $\lambda_T$ and $\tau_T$ of the sparse penalty function $J_T(f)$ satisfy $\lambda_T = (F^2\iota_{\lambda}(T)\log^{2+\nu_0} T)/T$ with a strictly increasing function $\iota_{\lambda}(x)$ such that $\iota_{\lambda}(x)/\log x \to \infty$ as $x \to \infty$ and $\tau_T(L_T+1)((N_T+1)B_T)^{L_T+1}\leq C_\tau T^{-1}$ with some positive constant $C_\tau$ for any $T$. Then, 
\begin{align*}
R(\hat{f}_T,f_0) &\leq 6\left(\bar{\Psi}_T^{\mathcal{F}}(\hat{f}_T) + \inf_{f \in \mathcal{F}}\left(R(f,f_0) + J_T(f)\right)\right) + CF^2 \left({1+ \log T \over T}\right),
\end{align*}
where $C$ is a positive constant only depending on $(C_\eta, C_{1,\beta}, C_{2,\beta}, C_L, C_N, C_B, C_\tau,\nu_0,\nu_1, \nu_2, K, \iota_\lambda)$.
\end{theorem}

Theorem \ref{thm: pred-error-bound2} is an extension of Theorem 1 in \cite{OhKi22}, which considers i.i.d.~observations. Here we explain some differences between their result and ours.  First, our conditions on the penalty function cover the clipped $L_1$ penalty, which is considered in \cite{OhKi22}, as a special case. Second, Theorem \ref{thm: pred-error-bound2} can be applied to nonstationary time series since we only assume the process $X$ to be $\beta$-mixing. Third, our approach to proving Theorem \ref{thm: pred-error-bound2} is different from that of \cite{OhKi22}. Their proofs heavily depend on the theory for i.i.d.~data in \cite{GKKW02} so extending their approach to our framework seems to require substantial work. In contrast, our approach is based on other technical tools such as the blocking technique of $\beta$-mixing processes in \cite{Ri13} and exponential inequalities for self-normalized martingale difference sequences. In particular, considering continuous time embedding of a martingale difference sequence and applying the results on (super-) martingales in \cite{BJY86}, we can allow the process $\{v_t\}_{t=1}^{T}$ to be conditionally centered and circumvent additional conditions on its distribution such as conditional Gaussianity or symmetry (see also Lemma E.2 and the proof of Lemma E.3 in the supplementary material). As a result, our result improves the power of logs in their generalization error bound. Fourth, (a) the upper bound of the depth of the sparse penalized DNN estimator $L$ can grow by a power of $\log T$ and (b) we take the tuning parameter $\lambda_T$ to depend on $F^2$. Particularly, (a) enables us to estimate $f_0$ adaptively when $f_0$ belongs to an $\ell^0$-bounded affine class as well as composition structured functions (see Sections \ref{subsec:comp} and \ref{subsec:ell0-bound} for details) and (b) enables $\hat{f}_{T,sp}$ to be adaptive with respect to $\|f\|_{\infty} \leq F$. See also the comments on Proposition \ref{prop:pdnn} on the improvement of the upper bound.

\section{Minimax optimality in nonlinear AR models}\label{sec:ar-model}

In this section, we show the minimax optimality of the sparse-penalized DNN estimator $\hat{f}_{T,sp}$. In particular, we show that $\hat{f}_{T,sp}$ achieves the minimax convergence rate over (i) composition structured functions and (ii) $\ell^0$-bounded affine class. We note that these classes of functions include many nonlinear AR models such as (generalized) additive AR models, single-index models, (multi-regime) threshold AR models, and exponential AR models. 

We consider the observation $\{Y_t\}_{t=1}^T$ generated by the following nonlinear AR($d$) model: 
\ben{\label{ar-model}
\left\{\begin{array}{l}
Y_t=m(Y_{t-1},\dots,Y_{t-d})+v_t,\qquad t=1,\dots,T,\\
v_t\overset{i.i.d.}{\sim}N(0,1), \qquad (Y_0,Y_{-1},\dots,Y_{-d+1})'\sim\nu.
\end{array}\right.
}
Here, $\nu$ is a (fixed) probability measure on $\mathbb R^d$ such that $\int_{\mathbb R^d}|x|\nu(dx)<\infty$, and $m:\mathbb R^d\to\mathbb R$ is an unknown function to be estimated. 
\begin{remark}
    Note that the process $\{Y_t\}_{t=1}^T$ is possibly non-stationary because the initial distribution $\nu$ is not necessarily the stationary distribution. However, below we impose conditions to ensure ergodicity of the process, so we require the model to be asymptotically stationary in this sense. This is necessary because we need an error bound \textit{uniformly valid} over a class of mean functions to establish (near) minimax optimality. A major difficulty to obtain such a bound is establishing a uniform upper bound for $\beta$-mixing coefficients because it is unclear how most available bounds depend on model parameters. For this purpose we prove Lemma \ref{lemma:beta} using the technique developed in \cite{HaMa11}. This is the place where we need assumptions to ensure ergodicity. 
    By contrast, the setup in the previous sections allows for some asymptotically non-stationary models like locally stationary processes (cf.~Example \ref{Ex:tv-model}). 
    We conjecture that a uniform $\beta$-mixing bound could be established for locally stationary processes under appropriate uniform versions of assumptions in \cite[Theorem 3.4]{Vo12}, but this requires us to carefully examine the entire proof of this result. This is beyond the scope of the paper and we leave it to future research.
\end{remark}


Let $\mathbf c=(c_0,c_1,\dots,c_d)\in(0,\infty)^{d+1}$ satisfy $\sum_{i=1}^dc_i<1$. 
We denote by $\mcl M_0(\mathbf c)$ the set of measurable functions $m:\mathbb R^d\to\mathbb R$ satisfying $|m(x)|\leq c_0+\sum_{i=1}^dc_i|x_i|$ for all $x\in\mathbb R^d$. The following lemma shows that the process $Y=\{Y_t\}_{t=1}^T$ is exponentially $\beta$-mixing ``uniformly'' over $m\in\mcl M_0(\mathbf c)$.
\begin{lemma}\label{lemma:beta}
Consider the nonlinear AR($d$) model \eqref{ar-model} with $m\in\mathcal M_0(\mathbf c)$.  Let $\beta_Y(t)$ be the $\beta$-mixing coefficient of $Y$. 
There are positive constants $C_\beta$ and $C_\beta'$ depending only on $\mathbf c, d$ and $\nu$ such that
\begin{equation}\label{eq:minimax-beta}
\beta_Y(t)\leq C'_\beta e^{-C_\beta t}\qquad\text{for all }t\geq1.
\end{equation}
\end{lemma}

The next result gives a generalization error bound of DNN estimators for a family of functions that can be approximated with a certain degree of accuracy by DNNs.

\begin{proposition}\label{prop:pdnn}
Consider the nonlinear AR($d$) model \eqref{ar-model} with $m\in\mathcal M_0(\mathbf c)$. 
Let $F$, $\hat f_T$, $\mcl F$, $L_T$, $N_T$, $B_T$, $\lambda_T$ and $\tau_T$ as in Theorem \ref{thm: pred-error-bound2}. 
Suppose that there are constants $\kappa,r\geq0$, $C_0>0$ and $C_S>0$ such that $\inf_{f\in\mathcal{F}_{\sigma}(L_T,N_T,B_T,F,S_{T})}\|f-f_0\|_{L^2([0,1]^d)}^2\leq C_0T^{-1/(\kappa+1)}$ with $S_{T}:=C_ST^{\kappa/(\kappa+1)}\log^r T$. Then,
\[
R(\hat{f}_T,f_0)\leq 6\bar{\Psi}_T^{\mathcal{F}}(\hat{f}_T)+C'F^2\frac{\iota_\lambda(T)\log^{2+\nu_0+r} T}{T^{1/(\kappa+1)}},
\]
where $C'$ is a positive constant only depending on $(\! \mathbf c, d, \nu, C_L, C_N, C_B,C_\tau, \nu_0,\nu_1, \nu_2, K, \iota_\lambda,\kappa,r,C_0,C_S\!)$.
\end{proposition}

If $\hat{f}_T = \hat{f}_{T,sp}$, then the generalization error bound in Proposition \ref{prop:pdnn} is reduced to $R(\hat{f}_T,f_0) \leq C'F^2\iota_\lambda(T)(\log^{2+\nu_0+r} T)T^{-1/(\kappa+1)}$. 
When $\nu_0=1$ and $\iota_\lambda(T)=\log^{\nu_3} T$ with $\nu_3 \in (1,2)$, one can see that our result improves the power of logs in the generalization error bound in Theorem 2 in \cite{OhKi22}. Moreover, our result allows the generalization error bound to depend explicitly on $F$. Combining this with the results in the following sections implies that the sparse-penalized DNN estimator can be adaptive concerning the upper bound of $\|f_0\|_{\infty}$ (by taking $F \asymp \log^{\nu_4}T$ with $\nu_4>0$ for example) and hence Proposition \ref{prop:pdnn} is useful for the computation of $\hat{f}_{T,sp}$ since the upper bound $F$ is unknown in practice as well as other information about the shape of $f_0$.

\subsection{Composition structured functions}\label{subsec:comp}

In this subsection, we consider nonparametric estimation of the mean function $f_0$ when it belongs to a class of composition structured functions which is defined as follows (cf. \cite{Sc20}).

For $p, r\in\mathbb N$ with $p\geq r$, $\beta,A>0$ and $l<u$, we denote by $C_r^\beta([l,u]^p, A)$ the set of functions $f:[l,u]^p\to\mathbb R$ satisfying the following conditions:
\begin{enumerate}[label=(\roman*)]

\item $f$ depends on at most $r$ coordinates.

\item $f$ is of class $C^{\lfloor\beta\rfloor}$ and satisfies
\[
\sum_{\bol\alpha : |\bol\alpha|_1 < \beta}\|\partial^{\bol\alpha} f\|_\infty + \sum_{\bol\alpha : |\bol\alpha |_1= \lfloor \beta \rfloor } \, \sup_{x, y \in [l,u]^p:x \neq y}
  \frac{|\partial^{\bol\alpha} f(x) - \partial^{\bol\alpha} f(y)|}{|x-y|_\infty^{\beta-\lfloor \beta \rfloor}} \leq A,
\]
where we used multi-index notation, that is, $\partial^{\bol\alpha}= \partial^{\alpha_1}\cdots \partial^{\alpha_p}$ with $\bol\alpha = (\alpha_1, \ldots, \alpha_p) \in \mathbb{Z}_{\geq 0}^p$ and $|\bol\alpha|_1:=\sum_{j=1}^p\alpha_j.$

\end{enumerate}
Let $\mathbf d=(d_0,\ldots,d_{q+1})\in\mathbb N^{q+2}$ with $d_0=d$ and $d_{q+1}=1$, $\mathbf t=(t_0, \ldots, t_q)\in\mathbb N^{q+1}$ with $t_i\leq d_i$ for all $i$ and $\bol\beta=(\beta_0, \ldots, \beta_q)\in(0,\infty)^{q+1}$. 
We define $\mcl G\big(q, \mathbf d, \mathbf t, \bol\beta, A\big)$ as the class of functions $f:[0,1]^d\to\mathbb R$ of the form
\begin{equation}\label{eq.mult_composite_regression}
f= g_q \circ \cdots \circ g_0,
\end{equation}
where $g_i=(g_{ij})_j : [l_i, u_i]^{d_i}\rightarrow [l_{i+1},u_{i+1}]^{d_{i+1}}$ with $g_{ij} \in C_{t_i}^{\beta_i}\big([l_i,u_i]^{d_i}, A\big)$ for some $|l_{i+1}|, |u_{i+1}| \leq A$, $i=0,\dots,q$. 

Denote by $\mathcal M\big(\mathbf c, q, \mathbf d, \mathbf t, \bol\beta, A\big)$ the class of functions in $\mcl M_0(\mathbf c)$ whose restrictions to $[0,1]^d$ belong to $\mathcal G\big(q, \mathbf d, \mathbf t, \bol\beta, A\big)$. 
Also, define $\beta_i^* := \beta_i \prod_{\ell=i+1}^q (\beta_{\ell}\wedge 1)$, $\phi_T := \max_{i=0, \ldots, q } T^{-\frac{2\beta_i^*}{2\beta_i^*+ t_i}}$.

\begin{example}[Nonlinear additive AR model]\label{Ex:add-ar}
Consider a nonlinear AR model: 
\begin{align*}
Y_t &= m_1(Y_{t-1}) + \dots + m_d(Y_{t-d}) + v_t, 
\end{align*}
where $m_1,\dots,m_d$ are univariate measurable functions. In this case, the mean function can be written as a composition of functions $m = g_1 \circ g_0$ with $g_0(x_1,\dots, x_d) = (m_1(x_1),\dots, m_d(x_d))'$ and $g_1(x_1,\dots, x_d) = \sum_{j=1}^{d}x_j$. Suppose that $m_j|_{[0,1]} \in C_1^{\beta}([0,1],A)$ for $j = 1,\dots, d$. Note that $g_1 \in C_d^{\gamma}([-A,A]^d, (A+1)d)$ for any $\gamma>1$. Then we can see that $m|_{[0,1]^d}: [0,1]^{d} \to [-Ad,Ad]$ and 
 \[
 m|_{[0,1]^d} \in \mathcal G\big(1, (d,d,1), (1,d), (\beta, (\beta \vee 2)d), (A+1)d\big). 
 \]
Hence $\phi_T=T^{-\frac{2\beta}{2\beta+1}}$ in this case.
\end{example}

\begin{example}[Nonlinear generalized additive AR model]\label{Ex:g-add-ar}
Consider a nonlinear AR model: 
\begin{align*}
Y_t &= \phi(m_1(Y_{t-1}) + \dots + m_d(Y_{t-d})) +v_t, 
\end{align*}
where $\phi: \mathbb{R} \to \mathbb{R}$ is some unknown link function. In this case, the mean function can be written as a composition of functions $m = g_2 \circ g_1 \circ g_0$ with $g_0$ and $g_1$ as in Example \ref{Ex:add-ar} and $g_2 = \phi$.  Suppose that $\phi \in C_1^{\gamma}([-Ad,Ad],A)$ and take $m_j$ and $g_1$ as in Example \ref{Ex:add-ar}. Then we can see that $m|_{[0,1]^d}: [0,1]^{d} \to [-A,A]$ and 
 \[
 m|_{[0,1]^d} \in \mathcal G\big(2, (d,d,1,1), (1,d,1), (\beta, (\beta \vee 2)d,\gamma), (A+1)d\big). 
 \]
Hence $\phi_T=T^{-\frac{2\beta(\gamma\wedge1)}{2\beta(\gamma\wedge1)+1}}\vee T^{-\frac{2\gamma}{2\gamma+1}}$ in this case.
\end{example}

\begin{example}[Single-index model]\label{ex:simod}
Consider a nonlinear AR model: 
\begin{align*}
Y_t &= \phi_0(Z_t)+\phi_1(Z_t)Y_{t-1}+\cdots+\phi_d(Z_t)Y_{t-d} +v_t,&
Z_t&=b_0+b_1Y_{t-1}+\cdots+b_dY_{t-d}, 
\end{align*}
where, for $j=0,1,\dots,d$, $\phi_j: \mathbb{R} \to \mathbb{R}$ is an unknown function and $b_j$ is an unknown constant. 
In this case, the mean function can be written as a composition of functions $m = g_2 \circ g_1 \circ g_0$ with $g_0(x_1,\dots,x_d)=(b_0+b_1x_1+\cdots+b_dx_d,x_1,\dots,x_d)'$, $g_1\!(z,x_1,\dots,x_d)=(\!\phi_0(z),\dots,\phi_d(z),x_1,\dots,x_d\!)'$, and $g_2(w_0,w_1,\dots,w_d,x_1,\dots,x_d)=w_0+w_1x_1+\cdots+w_dx_d$. 
Suppose that $\phi_0,\dots,\phi_d \in C_1^\beta \!(\![-A,A],A\!)$ for some constants $\beta\geq1$ and $A\geq1\vee\sum_{j=0}^d|b_j|$. 
Then we have
 \[
 m|_{[0,1]^d} \in \mathcal G\big(2, (d,d+1,2d+1,1), (d,1,2d+1), (\beta d,\beta, \beta (2d+1)), (A+1)(1+d+dA)\big). 
 \]
 Hence $\phi_T=T^{-\frac{2\beta}{2\beta+1}}$ in this case.
\end{example}

Below we show the minimax lower bound for estimating $f_0 \in \mathcal{M}(\mathbf c, q, \mathbf d, \mathbf t, \bol\beta, A)$.

\begin{theorem}\label{thm:minimax}
Consider the nonlinear AR($d$) model \eqref{ar-model} with $m\in\mathcal{M}(\mathbf c, q, \mathbf d, \mathbf t, \bol\beta, A)$. 
Suppose that $c_0\geq A$ and $t_j\leq\min\{d_0,\dots,d_{j-1}\}$ for all $j$. Then, for sufficiently large $A$, 
\[
\liminf_{T\to\infty}\phi_T^{-1}\inf_{\hat f_T}\sup_{m\in\mathcal{M}(\mathbf c, q, \mathbf d, \mathbf t, \bol\beta, A)}R(\hat f_T,f_0)>0,
\]
where the infimum is taken over all estimators $\hat f_T$. 
\end{theorem}

Theorem \ref{thm:minimax} and the next result imply that the sparse-penalized DNN estimator $\hat{f}_{T,sp}$ is rate optimal since it attains the minimax lower bound up to a poly-logarithmic factor. We write $\relu$ for the ReLU activation function, i.e.~$\relu(x)=\max\{x,0\}$.

\begin{theorem}\label{thm:rate-comp}
Consider the nonlinear AR($d$) model \eqref{ar-model} with $m\in\mathcal{M}(\mathbf c, q, \mathbf d, \mathbf t, \bol\beta, A)$. 
Let $F\geq1\vee A$ be a constant, $L_T\asymp\log^{r}T$ for some $r>1$, $N_T\asymp T$, $B_T,\lambda_T$ and $\tau_T$ as in Theorem \ref{thm: pred-error-bound2} with $\nu_0=r$, and $\hat f_T$ a minimizer of $\bar Q_T(f)$ subject to $f\in \mathcal{F}_{\relu}(L_T,N_T,B_T,F)$. Then
\[
\sup_{m\in\mathcal{M}(\mathbf c, q, \mathbf d, \mathbf t, \bol\beta, A)}R(\hat f_T,f_0)=O\left(\phi_T\iota_\lambda(T)\log^{3+r}T\right)\qquad\text{as}~T\to\infty.
\]
\end{theorem}

\begin{remark}
    The lower bound given in Theorem \ref{thm:minimax} is the same as the one obtained in \cite[Theorem 3]{Sc20} for nonparametric regression models with i.i.d.~errors, while the upper bound in Theorem \ref{thm:rate-comp} has the extra $\log T$ factor compared to \cite{Sc20}. 
    We conjecture that the latter would be an artifact of the proof by the following reasons:
\begin{itemize}
    \item \cite{Ya01} showed that the minimax rates for nonparametric regression models under random design are typically unchanged even if errors have long-range dependence.
    \item \cite{Ho99} showed that the minimax rates for nonlinear AR(1) models are exactly the same as those for nonparametric regression models with i.i.d.~errors in some cases.  
\end{itemize} 
\end{remark}

\subsection{$\ell^0$-bounded affine class}\label{subsec:ell0-bound}

In this subsection, we consider nonparametric estimation of the mean function $f_0$ when it belongs to an $\ell^0$-bounded affine class $\mcl I^0_\Phi$. This class was introduced in \cite{HaSu20} and is defined as follows.

\begin{definition}\label{def:0-affine}
Given a set $\Phi$ of real-valued functions on $\mathbb R^d$ with $\|\varphi\|_{L^2([0,1]^d)}=1$ for each $\varphi\in\Phi$ along with constants $n_s\in\mathbb N$ and $C>0$, we define an \textit{$\ell^0$-bounded affine class $\mcl I^0_\Phi$} as
\bm{
\mcl I^0_\Phi(n_s,C):=\left\{\sum_{i=1}^{n_s}\theta_i\varphi_i(A_i\cdot-b_i):A_i\in\mathbb R^{d\times d},b_i\in\mathbb R^d,\theta_i\in\mathbb R,\varphi_i\in\Phi,\right.\\
\left.|\det A_i|^{-1}\vee|A_i|_\infty\vee|b_i|_\infty\vee|\theta_i|\leq C,~i=1,\dots,n_s\right\}.
} 
\end{definition}

By taking the set $\Phi$ suitably, the class of functions $\mcl I^0_\Phi$ includes many nonlinear AR models such as threshold AR (TAR) models and we can show that the sparse-penalized DNN estimator attains the convergence rate $O(T^{-1})$ up to a poly-logarithmic factor (Theorem \ref{thm:rate-l0}).

\begin{example}[Threshold AR model]\label{ex:tar}
Consider a two-regime TAR(1) model:
\[
Y_t=(a_1\mathbf{1}_{(-\infty,r]}(Y_{t-1}) + a_2\mathbf{1}_{(r,\infty)}(Y_{t-1}))Y_{t-1} + v_t,
\]
where $a_1,a_2,r$ are some constants. This model corresponds to \eqref{ar-model} with $d=1$ and \\$m(y)=(a_1\mathbf{1}_{(-\infty,r]}(y)+a_2\mathbf{1}_{(r,\infty)}(y))y$. Note that the mean function $m$ can be discontinuous and this $m$ can be rewritten as
\[
m(y)=-a_1\relu(r-y)+a_1r\mathbf{1}_{[0,\infty)}(r-y)+a_2\relu(y-r)+a_2r\mathbf{1}_{[0,\infty)}(y-r).
\]
Hence $m\in\mcl I^0_\Phi(n_s,C)$ with $\Phi=\{\sqrt 3\relu,\mathbf{1}_{[0,\infty)}\}$, $n_s\geq4$ and $C\geq\max\{|a_1|,|a_2|,|r|\}$. This argument can be extended to a multi-regime (self-exciting) TAR model of any order in an obvious manner.    
\end{example}

We will later see in Example \ref{ex:far} that functional coefficient AR models are also covered by Definition \ref{def:0-affine}.

We set $\mcl{M}_\Phi^0(\mathbf c,n_s,C):=\mcl{M}_0(\mathbf c)\cap\mcl I^0_\Phi(n_s,C)$. Below we show the minimax lower bound for estimating $f_0 \in \mcl{M}_\Phi^0(\mathbf c,n_s,C)$.

\begin{theorem}\label{thm:minimax-l0}
Consider the nonlinear AR($d$) model \eqref{ar-model} with $m\in\mcl{M}_\Phi^0(\mathbf c,n_s,C)$. 
Suppose that $C\geq1/2$ and there is a function $\varphi\in\Phi$ such that $\supp(\varphi)\subset[0,1]^d$ and $\|\varphi\|_\infty\leq c_0$.  
Then, 
\[
\liminf_{T\to\infty}T\inf_{\hat f_T}\sup_{m\in\mcl{M}_\Phi^0(\mathbf c,n_s,C)}R(\hat f_T,f_0)>0,
\]
where the infimum is taken over all estimators $\hat f_T$. 
\end{theorem}

Now we extend the argument in Example \ref{ex:tar}. For this, we introduce the function class \\$\mathrm{AP}_{\sigma,d}(C_1,C_2,D,r)$ which can be approximated by ``light'' networks.

\begin{definition}\label{defi:AP}
For $C_1,C_2,D>0$ and $r\geq0$, we denote by $\mathrm{AP}_{\sigma,d}(C_1,C_2,D,r)$ the set of functions $\varphi:\mathbb R^d\to\mathbb R$ satisfying that, for each $\eps\in(0,1/2)$, there exist parameters $L_\eps,N_\eps,B_\eps,S_\eps>0$ such that
\begin{itemize}

\item $L_\eps\vee N_\eps\vee S_\eps\leq C_1\{\log_2(1/\eps)\}^{r}$ and $B_\eps\leq C_2/\eps$ hold;

\item there exists an $f\in\mathcal{F}_{\sigma}(L_\eps,N_\eps,B_\eps)$ such that $|\theta(f)|_0\leq S_\eps$ and $\|f-\varphi\|_{L^2([-D,D]^d)}^2\leq\eps$.

\end{itemize}
\end{definition}

Depending on the value of $r$, $\mathrm{AP}_{\sigma,d}(C_1,C_2,D,r)$ contains various functions such as step functions ($0 \leq r $), polynomials ($r=1$), and very smooth functions ($r=2$).

\begin{example}[Piecewise linear functions]\label{ex:plin}
For $\sigma=\relu$, we evidently have \\$\relu\in\mathrm{AP}_{\sigma,1}(C_1,C_2,D,r)$ for any $C_1,C_2,D\geq2$ and $r\geq0$. In this case we also have $\mathbf{1}_{[0,\infty)}\in\mathrm{AP}_{\sigma,1}(C_1,C_2,D,r)$ if $C_1,C_2\geq7$. In fact, for any $\eps\in(0,1/2)$, the function \\$f_\eps(x)=\sigma\left(\sigma(x+1)-\sigma(x)-\frac{1}{\eps}\sigma(-x)\right)$, $x\in\mathbb R$
satisfies $\|f_\eps-\mathbf{1}_{[0,\infty)}\|_{L^2([-D,D])}^2\leq\eps$.
\end{example}

\begin{example}[Polynomial functions]\label{ex:poly}
Take $\sigma=\relu$ and consider a polynomial function $\varphi(x)=\sum_{i=0}^pa_ix^i$ for some constants $a_0,\dots,a_p\in\mathbb R$. Then, given $D>0$, we have $\varphi\in\mathrm{AP}_{\sigma,1}(C_1,1/2,D,1)$ for some constant $C_1>0$ depending only on $\max_{i=0,\dots,p}|a_i|$, $p$ and $D$ by Proposition III.5 in \cite{EPGB21}. 
\end{example}

\begin{example}[Very smooth functions]\label{ex:holo}
Take $\sigma=\relu$ again. 
Let $\varphi:\mathbb R\to\mathbb R$ be a $C^\infty$ function such that there are constants $A\geq1$ and $D>0$ satisfying $\sup_{x\in[-D,D]}|\varphi^{(n)}(x)|\leq n!A$ for all $n\in\mathbb Z_{\geq0}$. 
Then, by Lemma A.6 in \cite{EPGB21}, $A^{-1}\varphi\in\mathrm{AP}_{\sigma,1}(C_1,1,D,2)$ for some constants $C_1>0$ depending only on $D$. 
Hence $\varphi\in\mathrm{AP}_{\sigma,1}(C_1,A,D,2)$. 
The condition on $\varphi$ is satisfied e.g.~when there is a holomorphic function $\Psi$ on $\{z\in\mathbb C:|z|<D+1\}$ such that $|\Psi|\leq A$ and $\Psi(x)=\varphi(x)$ for all $x\in[-D,D]$. This follows from Cauchy's estimates (cf.~Theorem 10.26 in \cite{Ru87}).
\end{example}

\begin{example}[Product with an indicator function]\label{ex:prod}
Again consider the ReLU activation function $\sigma=\relu$. 
Let $\varphi\in\mathrm{AP}_{\sigma,1}(C_1,C_2,D,r)$ for some constants $C_1,C_2,D>0$ and $r\geq1$, and assume $\sup_{x\in[-D,D]}|\varphi(x)|\leq A$ for some constant $A\geq1$. 
Then $\varphi\mathbf1_{[0,\infty)}\in\mathrm{AP}_{\sigma,1}(C_3,C_3,D,r)$ for some constant $C_3$ depending only on $C_1,C_2,D,A$. 
To see this, fix $\eps\in(0,1/2)$ arbitrarily and take $L_\eps,N_\eps,B_\eps,S_\eps$ and $f$ as in Definition \ref{defi:AP}. Also, let $f_\eps$ defined as in Example \ref{ex:plin}. 
By Proposition III.3 in \cite{EPGB21}, there is an $f_1\in\mcl F_\sigma(C_4\log(1/\eps),5,1)$ with $C_4>0$ depends only on $A$ such that $\sup_{x,y\in[-A,A]}|f_1(x,y)-xy|\leq\eps$. 
Then, by Lemmas II.3--II.4 and A.7 in \cite{EPGB21}, there is an $f_2\in\mcl F_\sigma(C_5\{\log(1/\eps)\}^r,C_5\{\log(1/\eps)\}^r,C_5/\eps)$ with $C_5>0$ depending only on $C_1,C_2,A$ such that $f_2(x)=f_1(f(x),f_\eps(x))$ for all $x\in\mathbb R$ and $\|\theta(f_2)\|_\infty\leq C_5\{\log(1/\eps)\}^r$. For this $f_2$, we have
\ba{
&\|f_2-\varphi\mathbf1_{[0,\infty)}\|_{L^2([-D,D])}\\
&\leq\|f_2-ff_\eps\|_{L^2([-D,D])}+\|(f-\varphi)f_\eps\|_{L^2([-D,D])}
+\|\varphi(f_\eps-\mathbf1_{[0,\infty)})\|_{L^2([-D,D])}\\
&\leq (D+1+A)\sqrt{\eps}.
}
Applying this argument to $\eps/\sqrt{D+1+A}$ instead of $\eps$, we obtain the desired result. 
\end{example}

Theorem \ref{thm:minimax-l0} and the next result imply that the sparse-penalized DNN estimator $\hat{f}_{T,sp}$ attains the minimax optimal rate over $\mcl{M}_\Phi^0(\mathbf c,n_s,C)$ up to a poly-logarithmic factor.

\begin{theorem}\label{thm:rate-l0}
Consider the nonlinear AR($d$) model \eqref{ar-model} with $m\in\mcl{M}_\Phi^0(\mathbf c,n_s,C)$. 
Suppose that $\Phi\subset\mathrm{AP}_{\relu,d}(C_1,C_2,D,r)$ for some constants $C_1,C_2>0$, $D\geq(d+1)C$ and $r\geq0$. 
Let $F\geq1+c_0$ be a constant, $L_T\asymp\log^{r'}T$ for some $r'>r$, $N_T\asymp T$, $B_T\asymp T^{\nu}$ for some $\nu>1$, $\lambda_T$ and $\tau_T$ as in Theorem \ref{thm: pred-error-bound2} with $\nu_0=r'$, and $\hat f_T$ a minimizer of $\bar Q_T(f)$ subject to $f\in \mathcal{F}_{\relu}(L_T,N_T,B_T,F)$. Then
\[
\sup_{m\in\mcl{M}_\Phi^0(\mathbf c,n_s,C)}R(\hat f_T,f_0)=O\left(\frac{\iota_\lambda(T)\log^{2+r+r'}T}{T}\right)\qquad\text{as}~T\to\infty.
\]
\end{theorem}

\begin{example}
By Examples \ref{ex:tar} and \ref{ex:plin}, the sparse-penalized DNN estimator adaptively achieve the minimax rate of convergence up to a logarithmic factor for threshold AR models. Thanks to Examples \ref{ex:poly}--\ref{ex:prod}, this result can be extended to some threshold AR models with nonlinear coefficients. 
\end{example}

\begin{example}[Functional coefficient AR model]\label{ex:far}
Examples \ref{ex:holo} and \ref{ex:prod} also imply that Theorem \ref{thm:rate-l0} can be extended to some functional coefficient AR (FAR) models introduced in \cite{ChTs93}: 
\[
Y_t = f_1(\mathbf Y_{t-1}^*)Y_{t-1} + \dots + f_d(\mathbf Y_{t-1}^*)Y_{t-d} + v_t
\]
where $\mathbf Y_{t-1}^* = (Y_{t-1},\dots, Y_{t-d})'$ and $f_j : \mathbb{R}^{d} \to \mathbb{R}$ are measurable functions. This model include many nonlinear AR models such as (1) TAR models (when $f_j$ are step functions), (2) exponential AR (EXPAR) models proposed in \cite{HaOz81} (when $f_j$ are exponential functions), and (3) smooth transition AR (STAR) models (e.g. \cite{GrTe93} and \cite{Te94}). Note that some classes of FAR models such as EXPAR and STAR models can be written as a composition of functions so Theorem \ref{thm:rate-comp} can be applied to those examples. 
\end{example}

\section{Simulation results}\label{sec:simulation}

In this section, we conduct a simulation experiment to assess the finite sample performance of DNN estimators for the mean function of nonlinear time series. 
Following \cite{OhKi22}, we compare the following five estimators in our experiment: Kernel ridge regression (KRR), $k$-nearest neighbors (kNN), random forest (RF), non-penalized DNN estimator (NPDNN), and sparse-penalized DNN estimator (SPDNN).

For kernel ridge regression, we used a Gaussian radial basis function kernel and selected the tuning parameters by 5-fold cross-validation as in \cite{OhKi22}. We determined the search grids for selection of the tuning parameters following \cite{Ex13}. The tuning parameter $k$ for $k$-nearest neighbors was also selected by 5-fold cross-validation with the search grid $\{5,7,\dots,41,43\}$. For random forest, unlike \cite{OhKi22}, we did not tune the number of the trees but fix it to 500 following discussions in \cite[Section 15.3.4]{HTF09} as well as the analysis of \cite{PrBo18}. Instead, we tuned the number of variables randomly sampled as candidates at each split. This was done by the R function \fun{tuneRF} of the package \pck{randomForest}.

For the DNN based estimators, we set the network architecture $(L,\mathbf p)$ as $L=3$ and $p_1=p_2=p_3=128$ along with the ReLU activation function $\sigma(x)=\max\{x,0\}$. 
Supposing that data were appropriately scaled, we ignored the restriction to $[0,1]^d$ of observations when constructing (and evaluating) the DNN based estimators. 
The network weights were trained by Adam \citep{KiBa15} with learning rate $10^{-3}$ and minibatch size of 64. To avoid overfitting, we determined the number of epochs by the following early stopping rule: First, we train the network weights using the first half of observation data and evaluate its mean square error (MSE) using the second half of the data at each epoch. We stop the training when the MSE is not improved within 5 epochs. After determining the number of epochs by this rule, we trained the network weights using the full sample. 
For the sparse-penalized DNN estimator, we also need to select the penalty function $\pi_{\lambda_T,\tau_T}$ and the tuning parameters $\lambda_T$ and $\tau_T$. We use the clipped $L_1$ penalty given by \eqref{eq:clip}. We set $\tau_T=10^{-9}$. $\lambda_T$ was selected from $\{\frac{S_y\log_{10}^3T}{8T},\frac{S_y\log_{10}^3T}{4T},\frac{S_y\log_{10}^3T}{2T},\frac{S_y\log_{10}^3T}{T},\frac{2S_y\log_{10}^3T}{T}\}$ to minimize the MSE in the above early stopping rule. Here, $S_y$ is the sample variance of $\{Y_t\}_{t=1}^T$.

We consider the following non-linear AR models for data-generating processes. Throughout this section, $\{\eps_t\}_{t=1}^T$ denote i.i.d.~standard normal variables.

\begin{description}
\item[EXPAR] $Y_t=a_1(Y_{t-1})Y_{t-1}+a_2(Y_{t-1})Y_{t-2}+0.2\eps_t$ with 
\ba{
a_1(y)&=0.138+(0.316+0.982y)e^{-3.89y^2},\\
a_2(y)&=-0.437-(0.659+1.260y)e^{-3.89y^2}.
}

\item[TAR] $Y_t=b_1(Y_{t-1})Y_{t-1}+b_2(Y_{t-1})Y_{t-2}+\eps_t$ with
\ba{
b_1(y)&=0.4\cdot\mathbf1_{(-\infty,1]}(y)-0.8\cdot\mathbf1_{(1,\infty)}(y),\\
b_2(y)&=-0.6\cdot\mathbf1_{(-\infty,1]}(y)+0.2\cdot\mathbf1_{(1,\infty)}(y).
}

\item[FAR] 
\[
Y_t=-Y_{t-2}\exp(-Y_{t-2}^2/2)+\frac{1}{1+Y_{t-2}^2}\cos(1.5Y_{t-2})Y_{t-1}+0.5\eps_t.
\]

\item[AAR] 
\ba{
Y_t=4\frac{Y_{t-1}}{1+0.8Y_{t-1}^2}+\frac{\exp\{3(Y_{t-2}-2)\}}{1+\exp\{3(Y_{t-2}-2)\}}+\eps_t.
}

\item[SIM] 
\ba{
Y_t=\exp(-8Z_t^2)+0.5\sin(2\pi Z_t)Y_{t-1}+0.1\eps_t,\quad
Z_t=0.8Y_{t-1}+0.6Y_{t-2}-0.6.
}

\item[SIM$_v$] For $v\in\{0.5,1.0,5.0\}$,
\ba{
Y_t&=\{\Phi(-vZ_t)-0.5\}Y_{t-1}+\{\Phi(2vZ_t)-0.6\}Y_{t-2}+\eps_t,\\
Z_t&=Y_{t-1}+Y_{t-2}-Y_{t-3}-Y_{t-4},
} 
where $\Phi$ is the standard normal distribution function. 
\end{description}

The first four models, EXPAR, TAR, FAR and AAR, are taken from Chapter 8 of \cite{FaYa03}; see Examples 8.3.7, 8.4.7 and 8.5.6 ibidem. The models SIM and SIM$_v$ are respectively taken from \cite[Example 1]{XiLi99} and \cite[Example 3.2]{XLT07} to cover the single-index model (cf.~Example \ref{ex:simod}). 
Since the model SIM$_v$ has a parameter $v$ varying over $\{0.5,1,5\}$, we consider totally eight models. 
We generated observation data $\{Y_t\}_{t=1}^T$ with $T=400$ and burn-in period of 100 observations.

As in \cite{OhKi22}, we evaluate the performance of each estimator by the empirical $L_2$ error computed based on newly generated $10^5$ simulated data. 
Figure \ref{fig:sim} shows the boxplots of the empirical $L_2$ errors of the five estimators over 500 Monte Carlo replications for eight models. 
As the figure reveals, the performances of KRR, NPDNN and SPDNN are superior to those of KNN and RF. Moreover, except for FAR, the DNN based estimators are comparable or better than KRR: 
For models with intrinsic low-dimensional structures such as AAR, SIM and SIM$_{0.5}$, the DNN based estimators perform slightly better than KRR. 
For models with discontinuous or rough mean functions such as TAR, SIM$_{1}$ and SIM$_5$, the performances of the DNN based estimators dominate that of KRR. 
These observations are in line with theoretical results developed in this paper.


\begin{figure}
\begin{center}
\includegraphics[scale=1]{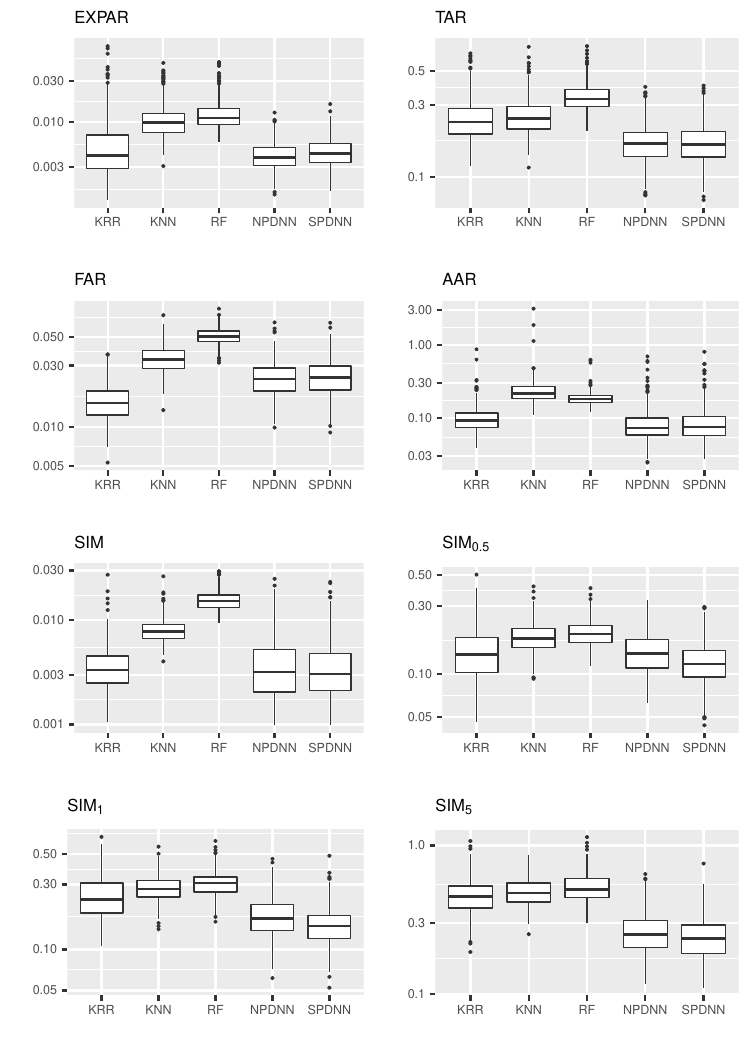}
\end{center}
\caption{Boxplots of empirical $L_2$ errors}\label{fig:sim}
\end{figure}

\section{Concluding remarks}\label{sec: conclusion}
In this paper, we have advanced statistical theory of feed-forward deep neural networks (DNN) for dependent data. For this, we investigated statistical properties of DNN estimators for nonparametric estimation of the mean function of a nonstationary and nonlinear time series. We established generalization error bounds of both non-penalized and sparse-penalized DNN estimators and showed that the sparse-penalized DNN estimators can estimate the mean functions of a wide class of the nonlinear autoregressive (AR) models adaptively and attain the minimax optimal convergence rates up to a logarithmic factor. The class of nonlinear AR models covers nonlinear generalized additive AR, single index models, and popular nonlinear AR models with discontinuous mean functions such as multi-regime threshold AR models. 

It would be possible to extend the results in Section \ref{sec:ar-model} to other function classes such as piecewise smooth functions \citep{ImFu19}, functions with low intrinsic dimensions \citep{Sc19, NaIm20}, and functions with varying smoothness \citep{Su19, SuNi21}. We leave such extensions as future research.

\newpage

\appendix

\section{Proofs for Section \ref{sec:main}}\label{Appendix: proof}

For random elements $X$ and $Y$, we write $X \stackrel{\mathcal{L}}{=} Y$ if they have the same law. Let $\mathcal{F}$ be a pointwise measurable class of real-valued functions on $\mathbb{R}^{d}$. For $\delta>0$, a finite set $\mathcal{G} \subset \mathcal{F}$ is called a \textit{$\delta$-covering} of $\mathcal{F}$ with respect to $\|\cdot \|_{\infty}$ if for any $f \in \mathcal{F}$ there exists $g \in \mathcal{G}$ such that $\|f-g\|_{\infty} \leq \delta$. The minimum cardinality of a $\delta$-covering of $\mathcal{F}$ with respect to $\|\cdot\|_{\infty}$ is called the \textit{covering number} of $\mathcal{F}$ with respect to $\|\cdot\|_{\infty}$ and denoted by $N(\delta, \mathcal{F}, \|\cdot\|_{\infty})$. Let $\sigma(Z)$ be the $\sigma$-field generated by the random element $Z$. Let $\hat{f}_{T}$ be an estimator taking values in $\mathcal{F}$ and define its expected empirical error by 
\begin{align*}
\widehat{R}(\hat{f}_{T},f_0) = \Ep\left[{1 \over T}\sum_{t=1}^{T}(\hat{f}_{T}(X_t) - f_0(X_t))^2\right]. 
\end{align*}
In what follows, we set $\beta(t) = \beta_X(t)$.

\subsection{Proof of Theorem \ref{thm: pred-error-bound}}\label{Appendix: lemmas}
First, we give an overview of the proof. To prove Theorem \ref{thm: pred-error-bound}, we will show the following results:  
\begin{lemma}[Lemma C.1]
Let $\delta>0$ and suppose that there exists an integer $\mathcal{N}_T$ such that $\mathcal{N}_T \geq N(\delta, \mathcal{F}, \|\cdot\|_{\infty}) \vee \exp(10)$. Also, let $a_T$ be a positive number such that $\mu_T := \lfloor T/(2a_T)\rfloor > 0$. In addition, suppose that there is a number $F \geq 1$ such that $\|f\|_{\infty} \leq F$ for all $f \in \mathcal{F} \cup \{f_0\}$. Then, for all $\varepsilon \in  (0, 1]$,
\begin{align*}
R(\hat{f}_{T},f_0) &\leq (1+\varepsilon)\widehat{R}(\hat{f}_{T},f_0) + {21(1+\varepsilon)^2 \over \varepsilon}F^2{\log \mathcal{N}_T \over \mu_T} + {4F^2 \over \mu_T} + 4(2+\varepsilon)F^2\beta(a_T) + 4(2+\varepsilon)F\delta.
\end{align*}
\end{lemma}
\begin{lemma}[Lemma C.2]
Let $\{(Y_t, X_t)\}_{t=1}^{T}$ be a time series satisfying (\ref{NLAR-model}), and set $f_0 := m\mathbf{1}_{[0,1]^{d}}$. Also, let $\delta > 0$ and assume $\mathcal{N}_T := N(\delta, \mathcal{F}, \|\cdot\|_{\infty})<\infty$. Suppose that there is a number $F \geq 1$ such that $\|f\|_{\infty} \leq F$ for all $f \in \mathcal{F} \cup \{f_0\}$. Suppose also that $\supp(f) \subset [0, 1]^{d}$ for all $f \in \mathcal{F}$. Then, under Assumption \ref{Ass: model}, for all $\varepsilon \in (0, 1)$ there exists a constant $C_{ \varepsilon}$ depending only on $(C_\eta, \varepsilon,K)$ such that
\begin{align*}
\widehat{R}(\hat{f}_{T},f_0) &\leq {1 \over 1-\varepsilon}\Psi_T^{\mathcal{F}}(\hat{f}_{T}) + {1 \over 1-\varepsilon}\inf_{f \in \mathcal{F}}R(f,f_0) + C_{\varepsilon}F^2\gamma_{\delta,T},
\end{align*}
where 
\[
\gamma_{\delta,T} := \delta +  {(\log T)(\log \mathcal{N}_T) \over T} + {1 \over T}. 
\]
\end{lemma}
Combining these results, we have
\begin{align}
R(\widehat{f}_T,f_0) &\leq {1+\varepsilon \over 1 - \varepsilon}\left(\Psi^{\mathcal{F}}(\widehat{f}_T) + \inf_{f \in \mathcal{F}}R(f, f_0)\right) \nonumber \\
&\quad + C_{\varepsilon}(1 + \varepsilon)F^2\left(\delta + {(\log T)(\log \mathcal{N}_T) \over T} + {1 \over T}\right) + {21(1+ \varepsilon)^2F^2 \over \varepsilon}{\log \mathcal{N}_T \over \mu_T} + {4F^2 \over \mu_T} \nonumber \\
&\quad + 4(2 +\varepsilon)F^2\beta(a_T) + 4(2 + \varepsilon)F\delta, \label{Lem:C1C2}
\end{align}
where $C_\varepsilon$ is a constant depending only on $(\varepsilon, C_\eta, K)$.
Lemma C.1 provides a bound of the generalization error using the expected empirical error and Lemma C.2 provides a bound of the expected empirical error. Note that the results do not require the estimator $\hat{f}_{T}$ to take values in $\mathcal{F}_{\sigma}(L,N,B,F,S)$ and hence would be of independent interest. Lemma C.1 can be shown as follows: Firstly, we apply the blocking technique for $\beta$-mixing sequences in \cite{Ri13} to approximate the data with independent blocks. Then, we employ an approach similar to Lemma 4(I) in \cite{Sc20} for these independent blocks. In this bound, what distinguishes it from independent data is that the second and third terms come from the blocking argument, while the fourth term depends on the $\beta$-mixing coefficient. When the data is independent, corresponding to $\mu_T=T$ and $\beta(a_T)=0$, this bound corresponds to Lemma 4(I) in \cite{Sc20}. Lemma C.2 corresponds to Lemma 4(III) in \cite{Sc20}. Although the bound in Lemma C.2 does not explicitly manifest the time-series structure, a key aspect of the proof involves establishing a new exponential inequality for self-normalized martingale differences (Lemma D.3) and using this result to evaluate $\Ep[T^{-1}\sum_{t=1}^{T}(\hat{f}_T(X_t) - f_0(X_t))\eta(X_t)v_t]$. This approach differs from \cite{Sc20} and constitutes a unique aspect of the proof. Theorem \ref{thm: pred-error-bound} follows from these lemmas and the bound on the $T^{-1}$-covering number of $\mathcal{F}_\sigma(L,N,B,F,S)$ with respect to $\|\cdot\|_\infty$.

Now we move on to the proof of Theorem \ref{thm: pred-error-bound}. Letting $\mathcal F=\mathcal{F}_{\sigma}(L,N,B,F,S)$, $\delta = T^{-1}$ and $a_T = C_{2,\beta}^{-1}\log T$ in (\ref{Lem:C1C2}), we have
\begin{align*}
R(\hat{f}_T,f_0) &\leq {1+\varepsilon \over 1-\varepsilon}\Psi_T^{\mathcal{F}}(\hat{f}_T) + {1+\varepsilon \over 1-\varepsilon}\inf_{f \in \mathcal{F}}R(f, f_0) + \left({21(1+\varepsilon)^2 \over \varepsilon}\log \mathcal{N}_T + 4\right){F^2 \over \mu_T}\\ 
&\quad + \left(2C_\varepsilon(1+\varepsilon) + 8(2+\varepsilon)\right){F^2 \over T} + C_{\varepsilon}(1+\varepsilon)F^2{(\log T)(\log \mathcal{N}_T) \over T}.
\end{align*}
We may assume $\mu_T = \lfloor {T \over 2a_T} \rfloor >0$; otherwise $T<2a_T$ and thus $C_{2,\beta}/2 < {\log T \over T}$. Moreover, using Lemma E.5, $\log N\left(T^{-1}, \mathcal{F}_{\sigma}(L,N,B,F,S), \| \cdot \|_{\infty}\right) \leq 2S(L+1)\log \left((L+1)(N+1)BT\right)$. Therefore, letting $\mathcal{N}_T := N\left(T^{-1}, \mathcal{F}_{\sigma}(L,N,B,F,S), \| \cdot \|_{\infty}\right) \vee \exp(10)$, there exists a universal constant $c_1>0$ such that $\log \mathcal{N}_T \leq c_1S(L+1)\log \left((L+1)(N+1)BT\right)$. 
Combining all the estimates above, we obtain the desired result. \qed

\begin{remark}[Derivation of (\ref{beta-poly-bound})]
When the $\beta$-mixing coefficient decays polynomially fast, that is, $\beta(t)\leq C_\beta t^{-\alpha}$ for some $\alpha>0$, then letting $\delta=T^{-1}$ and $a_T = (T/\log \mathcal{N}_T)^{{1 \over \alpha+1}}$ in (\ref{Lem:C1C2}), we obtain (\ref{beta-poly-bound}). 
\end{remark}

\subsection{Proof of Theorem \ref{thm: pred-error-bound2}}

Throughout the proof, we set $\beta(t) = \beta_X(t)$. Without loss of generality, we may assume $(1+C_\tau)T^{-1}<1$; otherwise, the desired bound holds with $C=4(1+C_\tau)$ because $R(\hat f_T,f_0)\leq4F^2$. Then we have
\ben{\label{def:delta}
\delta_T:=T^{-1}+\tau_T(L_T+1)((N_T+1)B_T)^{L_T+1}\leq(1+C_\tau)T^{-1}<1.
}
Next, for $k \geq 0$ and $s>0$, we define $\mathcal{F}_{T,k,s} := \left\{f \in \mathcal{F}: 2^{k-1}s\mathbf{1}_{\{k \neq 0\}} \leq J_T(f) < 2^k s\right\}$, $\Omega_{T,k,s} := \left\{\hat{f}_T\in\mathcal{F}_{T,k,s}\right\}$. Note that $\Omega = \bigcup_{k=0}^{\infty}\Omega_{T,k,s}$ and $\Omega_{T,k_1,s} \cap \Omega_{T, k_2,s} = \emptyset$ for $k_1 \neq k_2$. For each $k$, set $\mathcal{N}_{k}  := N(\delta_T, \mathcal{F}_{T,k,s},\|\cdot\|_{\infty})$ and let $\{f^k_1,\dots, f^k_{\mathcal{N}_k}\}$ be a $\delta_T$-covering of $\mathcal{F}_{T,k,s}$ with respect to $\|\cdot\|_{\infty}$. By construction, we can define a random variable $J_k$ taking values in $\{1,\dots, \mathcal{N}_k\}$ such that $\|\hat{f}_{T} - f^k_{J_k}\|_{\infty} \leq \delta_T$ on $\Omega_{T,k,s}$.

\underline{Step 1} (Reduction to independence) Let $a_T$ be a positive number such that $2a_T<T$, $a_T \to \infty$ and $a_T/ \iota_\lambda(T) \to 0$ as $T \to \infty$, and set $\mu_T := \lfloor T/2a_T\rfloor > 0$. For $\ell = 0,\dots, \mu_T-1$, let $I_{1,\ell} = \{2\ell a_T+1, \dots, (2\ell+1) a_T \}$, $I_{2,\ell} = \{(2\ell+1)a_T+1, \dots, 2(\ell+1)a_T\}$. For each $k$, we define
\begin{align*}
\tilde{g}^k_{\ell} &:= (\tilde{g}^k_{1,\ell},\dots, \tilde{g}^k_{\mathcal{N}_k,\ell})' = \left(\sum_{t \in I_{1,\ell}}(f^k_1(X_t) - f_0(X_t))^2, \dots, \sum_{t \in I_{1,\ell}}(f^k_{\mathcal{N}_k}(X_t) - f_0(X_t))^2\right)',\\
\tilde{g}_{\ell}^{\ast,k} &:= (\tilde{g}_{1,\ell}^{\ast,k},\dots, \tilde{g}_{\mathcal{N}_k,\ell}^{\ast,k})' = \left(\sum_{t \in I_{1,\ell}}(f^k_1(X_t^{\ast}) - f_0(X_t^{\ast}))^2, \dots, \sum_{t \in I_{1,\ell}}(f^k_{\mathcal{N}_k}(X_t^{\ast}) - f_0(X_t^{\ast}))^2\right)'.
\end{align*}
From a similar argument in Step1 of the proof of Lemma C.1, we can show that there exist two sequences of independent $\mathbb{R}^{\mathcal{N}_k}$-valued random variables $\{g_{\ell}^k\}_{\ell = 0}^{\mu_T-1}$ and $\{g_{\ell}^{\ast,k}\}_{\ell = 0}^{\mu_T-1}$  such that for all $\ell = 0,\dots,\mu_T-1$, 
\begin{align*}
g^k_\ell &\stackrel{\mathcal{L}}{=} \tilde{g}^k_\ell,\ \Prob(g^k_\ell \neq \tilde{g}^k_\ell) \leq \beta(a_T),\quad g^{\ast, k}_\ell \stackrel{\mathcal{L}}{=} \tilde{g}^{\ast, k}_\ell,\ \Prob(g^{\ast, k}_\ell \neq \tilde{g}^{\ast, k}_\ell) \leq \beta(a_T).
\end{align*}
with $0\leq g^k_{j,\ell} \leq 4F^2a_T$ and $0\leq g^{\ast,k}_{j,\ell} \leq 4F^2a_T$ a.s., where $g^k_{j,\ell}$ and $g_{j,\ell}^{\ast,k}$ are the $j$-th components of $g^k_{\ell}$ and $g_{\ell}^{\ast,k}$, respectively. Further, let $k_{T,s}$ be an integer such that $\sum_{k \geq k_{T,s}+1}\Prob(\Omega_{T,k,s}) \leq 1/\mu_T$. For $\ell = 0,\dots, \mu_T-1$, we define
\begin{align*}
\tilde{\mathfrak{g}}_\ell &:= \left((\tilde{g}^{0}_\ell)',\dots, (\tilde{g}^{k_{T,s}}_\ell)' \right)',\ \tilde{\mathfrak{g}}^{\ast}_\ell := \left((\tilde{g}^{\ast, 0}_\ell)',\dots, (\tilde{g}^{\ast, k_{T,s}}_\ell)' \right)',\\
\mathfrak{g}_\ell &:= \left((g^{0}_\ell)',\dots, (g^{k_{T,s}}_\ell)' \right)',\ \mathfrak{g}^{\ast}_\ell := \left((g^{\ast, 0}_\ell)',\dots, (g^{\ast, k_{T,s}}_\ell)' \right)'. 
\end{align*}
We can also assume that for all $\ell = 0,\dots,\mu_T-1$, 
\begin{align*}
\mathfrak{g}_\ell &\stackrel{\mathcal{L}}{=} \tilde{\mathfrak{g}}_\ell,\ \Prob(\mathfrak{g}_\ell \neq \tilde{\mathfrak{g}}_\ell) \leq \beta(a_T),\quad \mathfrak{g}^{\ast}_\ell \stackrel{\mathcal{L}}{=} \tilde{\mathfrak{g}}^{\ast}_\ell,\ \Prob(\mathfrak{g}^{\ast}_\ell \neq \tilde{\mathfrak{g}}^{\ast}_\ell) \leq \beta(a_T).
\end{align*}
Likewise, define
\begin{align*}
\tilde{h}^k_{\ell} &:= (\tilde{h}^k_{1,\ell},\dots, \tilde{h}^k_{\mathcal{N}_k,\ell}) = \left(\sum_{t \in I_{2,\ell}}(f^k_1(X_t) - f_0(X_t))^2, \dots, \sum_{t \in I_{2,\ell}}(f^k_{\mathcal{N}_k}(X_t) - f_0(X_t))^2\right)',\\
\tilde{h}_{\ell}^{\ast,k} &:= (\tilde{h}_{1,\ell}^{\ast,k},\dots, \tilde{h}_{\mathcal{N}_k,\ell}^{\ast,k}) = \left(\sum_{t \in I_{2,\ell}}(f^k_1(X_t^{\ast}) - f_0(X_t^{\ast}))^2, \dots, \sum_{t \in I_{2,\ell}}(f^k_{\mathcal{N}_k}(X_t^{\ast}) - f_0(X_t^{\ast}))^2\right)'
\end{align*}
and there exist two sequences of independent $\mathbb{R}^{\mathcal{N}_k}$-valued random variables $\{h_{\ell}^k\}_{\ell = 0}^{\mu_T-1}$ and $\{h_{\ell}^{\ast,k}\}_{\ell = 0}^{\mu_T-1}$  such that for all $\ell = 0,\dots,\mu_T-1$, 
\begin{align*}
h^k_\ell &\stackrel{\mathcal{L}}{=} \tilde{h}^k_\ell,\ \Prob(h^k_\ell \neq \tilde{h}^k_\ell) \leq \beta(a_T),\quad h^{\ast, k}_\ell \stackrel{\mathcal{L}}{=} \tilde{h}^{\ast, k}_\ell,\ \Prob(h^{\ast, k}_\ell \neq \tilde{h}^{\ast, k}_\ell) \leq \beta(a_T).
\end{align*}
For $\ell = 0,\dots, \mu_T-1$, we define
\begin{align*}
\tilde{\mathfrak{h}}_\ell &:= \left((\tilde{h}^{0}_\ell)',\dots, (\tilde{h}^{k_{T,s}}_\ell)' \right)',\ \tilde{\mathfrak{h}}^{\ast}_\ell := \left((\tilde{h}^{\ast, 0}_\ell)',\dots, (\tilde{h}^{\ast, k_{T,s}}_\ell)' \right)',\\
\mathfrak{h}_\ell &:= \left((h^{0}_\ell)',\dots, (h^{k_{T,s}}_\ell)' \right)',\ \mathfrak{h}^{\ast}_\ell := \left((h^{\ast, 0}_\ell)',\dots, (h^{\ast, k_{T,s}}_\ell)' \right)'. 
\end{align*}
We can also assume that for all $\ell = 0,\dots,\mu_T-1$, 
\begin{align*}
\mathfrak{h}_\ell &\stackrel{\mathcal{L}}{=} \tilde{\mathfrak{h}}_\ell,\ \Prob(\mathfrak{h}_\ell \neq \tilde{\mathfrak{h}}_\ell) \leq \beta(a_T),\quad \mathfrak{h}^{\ast}_\ell \stackrel{\mathcal{L}}{=} \tilde{\mathfrak{h}}^{\ast}_\ell,\ \Prob(\mathfrak{h}^{\ast}_\ell \neq \tilde{\mathfrak{h}}^{\ast}_\ell) \leq \beta(a_T).
\end{align*}

\underline{Step 2} (Bounding the generalization error) In this step, we will show
\begin{align}\label{pred-error-bound-adaptive}
R(\hat{f}_T,f_0) & \leq 3\left(\hat{R}(\hat{f}_T,f_0) + {1 \over 3}\Ep[J_T(\hat{f}_T)]\right) 
+ {32F^2a_T \over T}\left(1 + {2^{5/4}\exp\left(-{sT \over 128F^2a_T}\right) \over 1-\exp\left(-{sT \over 256F^2a_T}\right)}\right) \nonumber \\ 
&\quad \quad +24F\delta_T + {s \over 4} + {32F^2 \over \mu_T} + 48F^2\beta(a_T) 
\end{align}
for $T\geq T_0$, where $T_0>0$ is a constant depending only on ($C_L$, $C_N$, $C_B$, $\nu_0$, $\nu_1$, $\nu_2$, $\iota_\lambda$). 
Define
\begin{align*}
D &:= R(\hat{f}_T, f_0) - \hat{R}(\hat{f}_T, f_0) - {1 \over 2}\Ep[J_T(\hat{f}_T)], \ D_k := \Ep\left[\mathbf{1}_{\Omega_{T,k,s}}\left({1 \over T}\sum_{t=1}^{T}\Delta_t(\hat{f}_T)-{1 \over 2}J_T(\hat{f}_T)\right)\right],
\end{align*}
where $\Delta_t(f) = (f(X_t^{\ast}) - f_0(X_t^{\ast}))^2 - (f(X_t) - f_0(X_t))^2$. Note that $D = \sum_{k=0}^{\infty}D_k$. Since $|\Delta_t(\hat{f}_T) - \Delta_t(f_{J_k}^k)| \leq 8F\delta_T$, we have
\begin{align*}
D_k &\leq \Ep\left[\mathbf{1}_{\Omega_{T,k,s}}\left({1 \over T}\sum_{t=1}^{T}\Delta_t(f_{J_k}^k) - {1 \over 2}J_T(\hat{f}_T)\right)\right] +  \Ep\left[\mathbf{1}_{\Omega_{T,k,s}}\left({1 \over T}\sum_{t=1}^{T}\left|\Delta_t(\hat{f}_T) - \Delta_t(f_{J_k}^k)\right| \right)\right]\\
&\leq \Ep\left[\mathbf{1}_{\Omega_{T,k,s}}\left({1 \over T}\sum_{t=1}^{T}\Delta_t(f_{J_k}^k) - {1 \over 2}J_T(\hat{f}_T)\right)\right] + 8F\delta_T\Prob(\Omega_{T,k,s}). 
\end{align*}
Observe that 
\begin{align*}
&\Ep\left[\mathbf{1}_{\Omega_{T,k,s}}\left({1 \over T}\sum_{t=1}^{T}\Delta_t(f_{J_k}^k)\right)\right]\\ 
&\leq \Ep\left[\mathbf{1}_{\Omega_{T,k,s}}{1 \over T}\!\left(\sum_{\ell=0}^{\mu_T-1}\tilde{g}_{J_k,\ell}^{\ast,k} +\!\! \sum_{\ell=0}^{\mu_T-1}\tilde{h}_{J_k,\ell}^{\ast,k} - \!\! \sum_{\ell=0}^{\mu_T-1}\tilde{g}_{J_k,\ell}^{k} - \!\! \sum_{\ell=0}^{\mu_T-1}\tilde{h}_{J_k,\ell}^{k}\right)\! \right] \!+\! \Ep\left[\mathbf{1}_{\Omega_{T,k,s}}\!\left(\!{1 \over T}\sum_{t=2a_T\mu_T+1}^{T}\!\!\!\!\!\!\!\! |\Delta_t(f^k_{J_k})|\! \right)\right]\\
&\leq \Ep\left[\mathbf{1}_{\Omega_{T,k,s}}{1 \over T}\!\left(\sum_{\ell=0}^{\mu_T-1}\tilde{g}_{J_k,\ell}^{\ast,k} +\!\! \sum_{\ell=0}^{\mu_T-1}\tilde{h}_{J_k,\ell}^{\ast,k} - \!\! \sum_{\ell=0}^{\mu_T-1}\tilde{g}_{J_k,\ell}^{k} -\!\! \sum_{\ell=0}^{\mu_T-1}\tilde{h}_{J_k,\ell}^{k}\right)\! \right] + {8F^2 \over \mu_T}\Prob(\Omega_{T,k,s}). 
\end{align*}
Further, 
\begin{align}\label{g-star-block-bound}
&{1 \over T}\sum_{\ell=0}^{\mu_T-1}\sum_{k=0}^{\infty}\Ep\left[\mathbf{1}_{\Omega_{T,k,s}}\left|\tilde{g}_{J_k,\ell}^{\ast,k} - g_{J_k, \ell}^{\ast,k}\right|  \right] \leq {4F^2a_T \over T}\sum_{\ell=0}^{\mu_T-1}\sum_{k=0}^{\infty}\Ep\left[\mathbf{1}_{\Omega_{T,k,s}}\mathbf{1}_{\left\{\tilde{g}_{J_k,\ell}^{\ast,k} \neq g_{J_k, \ell}^{\ast,k}\right\}}  \right] \nonumber \\
&\leq {4F^2a_T \over T}\sum_{\ell=0}^{\mu_T-1}\sum_{k=0}^{\infty}\Ep\left[\mathbf{1}_{\Omega_{T,k,s}}\mathbf{1}_{\left\{\mathfrak{\tilde{g}}_{\ell}^{\ast} \neq \mathfrak{g}_{\ell}^{\ast}\right\}}  \right] \leq {4F^2a_T \over T}\sum_{\ell=0}^{\mu_T-1}\! \left(\sum_{k=0}^{k_{T,s}}\Ep\left[\mathbf{1}_{\Omega_{T,k,s}}\mathbf{1}_{\left\{\mathfrak{\tilde{g}}_{\ell}^{\ast} \neq \mathfrak{g}_{\ell}^{\ast}\right\}}  \right] + \sum_{k = k_{T,s}+1}^{\infty}\!\!\!\!\! \Ep\left[\mathbf{1}_{\Omega_{T,k,s}}  \right] \! \right) \nonumber \\
&\leq {4F^2a_T \over T}\sum_{\ell=0}^{\mu_T-1}\! \left(\Ep\left[\mathbf{1}_{\left\{\mathfrak{\tilde{g}}_{\ell}^{\ast} \neq \mathfrak{g}_{\ell}^{\ast}\right\}}  \right] \!+\!\! \sum_{k = k_{T,s}+1}^{\infty}\!\!\!\!\! \Ep\left[\mathbf{1}_{\Omega_{T,k,s}}  \right] \! \right) \leq {4F^2a_T \over T}\cdot \mu_T\left(\beta(a_T)+\frac{1}{\mu_T}\right) \leq 2F^2\left(\beta(a_T)+\frac{1}{\mu_T}\right).
\end{align}
Therefore, similar arguments to obtain (\ref{g-star-block-bound}) yield
\begin{align*}
&\sum_{k=0}^{\infty}\Ep\left[\mathbf{1}_{\Omega_{T,k,s}}\left({1 \over T}\sum_{t=1}^{T}\Delta_t(f_{J_k}^k) - {1 \over 2}J_T(\hat{f}_T)\right)\right]\\ 
&\leq \sum_{k=0}^{\infty}\! \Ep \! \left[\! \mathbf{1}_{\Omega_{T,k,s}}\!\! \left(\!{1 \over T}\!\! \sum_{\ell=0}^{\mu_T-1}\!\!\!(g_{J_k,\ell}^{\ast,k} \!-\! g_{J_k,\ell}^k) \!-\! {1 \over 4}J_T(\hat{f}_T)\! \right) \! \right] \\
&\quad \!+\! \sum_{k=0}^{\infty}\! \Ep \! \left[\! \mathbf{1}_{\Omega_{T,k,s}}\! \left(\!{1 \over T}\!\! \sum_{\ell=0}^{\mu_T-1}\!\!\! (h_{J_k,\ell}^{\ast,k} \!-\! h_{J_k,\ell}^k) \!-\! {1 \over 4}J_T(\hat{f}_T)\!\right)\! \right] \!+\! {16F^2 \over \mu_T} \!+\! 8F^2\beta(a_T). 
\end{align*}
Define $b_{j,k} := \mathbf{1}_{\Omega_{T,k,s}}\sum_{\ell=0}^{\mu_T-1}(g_{j,\ell}^{\ast,k} - g_{j,\ell}^k)^2$, $\bar b_k:=\Ep\left[b_{J_k,k}\right]$, $r_{j,k}:=2\sqrt{b_{j,k}+\bar b_k}$, \\$B_{j,k} := \mathbf{1}_{\Omega_{T,k,s}}\left|\sum_{\ell=0}^{\mu_T-1}(g_{j,\ell}^{\ast,k} - g_{j,\ell}^k)/ r_{j,k}\right|$, $B_k:=B_{J_k,k}$, where $B_{j,k}:=0$ if the denominator equals $0$. Then we have
\begin{align*}
\Ep\left[\mathbf{1}_{\Omega_{T,k,s}}\left({1 \over T}\sum_{\ell=0}^{\mu_T-1}(g_{J_k,\ell}^{\ast,k} - g_{J_k,\ell}^k) - {1 \over 4}J_T(\hat{f}_T)\right) \right] &\leq \Ep\left[\left({1 \over T}r_{J_k,k}B_k - \mathbf{1}_{\Omega_{T,k,s}}{1 \over 4}J_T(\hat{f}_T)\right)\right].
\end{align*}
Applying the AM-GM inequality, we have
\begin{align*}
&\Ep \!\left[\! \mathbf{1}_{\Omega_{T,k,s}} \!\! \left({1 \over T} \!\sum_{\ell=0}^{\mu_T-1}\!\! (g_{J_k,\ell}^{\ast,k} - g_{J_k,\ell}^k) \!- \! {1 \over 4}J_T(\hat{f}_T) \!\right) \right] \leq \Ep\left[{r^2_{J_k,k} \over 64F^2Ta_T}\right] + \Ep\left[{16F^2a_T \over T}B_k^2 - \mathbf{1}_{\Omega_{T,k,s}}{1 \over 4}J_T(\hat{f}_T)\right]\\
&={8 \over 64F^2Ta_T}\Ep\left[b_{J_k,k}\right]
\!+\!  \Ep \!\left[\!{16F^2a_T \over T}B_k^2 \!-\! \mathbf{1}_{\Omega_{T,k,s}}{1 \over 4}J_T(\hat{f}_T)\right] =: I_{r,k} + I_{B,k}.
\end{align*}
Now we evaluate $I_{r,k}$. Observe that
\begin{align}
&\Ep[b_{J_k,k}] \nonumber \\
&\leq 4F^2a_T\Ep\left[\mathbf{1}_{\Omega_{T,k,s}}\sum_{\ell=0}^{\mu_T-1}\left|g_{J_k,\ell}^{\ast,k} - g_{J_k,\ell}^k\right|\right]  \nonumber \\
&\leq 4F^2a_T\Ep\left[\mathbf{1}_{\Omega_{T,k,s}}\sum_{\ell=0}^{\mu_T-1}\left|\tilde{g}_{J_k,\ell}^{\ast,k} - \tilde{g}_{J_k,\ell}^k\right|\right] + 4F^2a_T\Ep\left[\mathbf{1}_{\Omega_{T,k,s}}\sum_{\ell=0}^{\mu_T-1}\left(\left|\tilde{g}_{J_k,\ell}^{\ast,k} - g_{J_k,\ell}^{\ast, k}\right|+\left|\tilde{g}_{J_k,\ell}^k - g_{J_k,\ell}^k\right|\right)\right] \nonumber \\
&\leq 4F^2a_T\left(\Ep\left[\mathbf{1}_{\Omega_{T,k,s}}\sum_{\ell=0}^{\mu_T-1}\!\!\sum_{t \in I_{1,\ell}}(\hat{f}_T(X_t^{\ast}) - f_0(X_t^{\ast}))^2\right] + \Ep\left[\mathbf{1}_{\Omega_{T,k,s}}\sum_{\ell=0}^{\mu_T-1}\sum_{t \in I_{1,\ell}}(\hat{f}_T(X_t) - f_0(X_t))^2\right]\right) \nonumber \\
&\ +\! 4F^2a_T\!\left(\!\Ep \! \left[\! \mathbf{1}_{\Omega_{T,k,s}}\!\!\!\!\sum_{\ell=0}^{\mu_T-1}\!\left|\tilde{g}_{J_k,\ell}^{\ast,k}\!-\!\!\sum_{t \in I_{1,\ell}}\!\!\!\!(\hat{f}_T(X_t^{\ast}) \!-\! f_0(X_t^{\ast})\!)^2\right|\right]  \!+\! \Ep \! \left[\! \mathbf{1}_{\Omega_{T,k,s}}\!\!\!\! \sum_{\ell=0}^{\mu_T-1}\! \left|\tilde{g}_{J_k,\ell}^k \!-\!\! \sum_{t \in I_{1,\ell}}\!\!\!\!(\hat{f}_T(X_t) \!-\! f_0(X_t)\!)^2\right|\right]\! \right) \nonumber \\
&\ +\! 4F^2a_T \cdot 8F^2a_T\sum_{\ell=0}^{\mu_T-1}\Ep\left[\mathbf{1}_{\Omega_{T,k,s}}\left(\mathbf{1}_{\left\{\tilde{g}_{J_k,\ell}^{\ast,k} \neq g_{J_k,\ell}^{\ast, k}\right\}}+\mathbf{1}_{\left\{\tilde{g}_{J_k,\ell}^k \neq g_{J_k,\ell}^k\right\}}\right)\right] \nonumber\\
&\leq 4F^2a_T\Ep\left[\mathbf{1}_{\Omega_{T,k,s}}\sum_{\ell=0}^{\mu_T-1}\!\!\sum_{t \in I_{1,\ell}}(\hat{f}_T(X_t^{\ast}) - f_0(X_t^{\ast}))^2\right]  + 4F^2a_T\Ep\left[\mathbf{1}_{\Omega_{T,k,s}}\sum_{\ell=0}^{\mu_T-1}\!\! \sum_{t \in I_{1,\ell}}(\hat{f}_T(X_t) - f_0(X_t))^2\right] \nonumber \\
&\ +\! 4F^2\!a_T \! \left(4Fa_T\mu_T\delta_T \!+\! 4Fa_T\mu_T\delta_T\right)\! \Prob(\Omega_{T,k,s}) \!+\! 32F^4\!a_T^2 \!\!\! \sum_{\ell=0}^{\mu_T-1}\!\!\! \Ep \!\left[\!\mathbf{1}_{\Omega_{T,k,s}}\!\! \left(\! \mathbf{1}_{\left\{\tilde{g}_{J_k,\ell}^{\ast,k} \neq g_{J_k,\ell}^{\ast, k}\right\}} \!\!+\! \mathbf{1}_{\left\{\tilde{g}_{J_k,\ell}^k \neq g_{J_k,\ell}^k\right\}} \! \right)\!\right]. \label{r-g-bound1}
\end{align}
For the above inequalities, we used the fact that $\left|\tilde{g}_{J_k,\ell}^{\ast,k} - g_{J_k,\ell}^{\ast, k}\right| \leq 4F^2a_T$, $\left|\tilde{g}_{J_k,\ell}^k - g_{J_k,\ell}^k\right| \leq 4F^2a_T$ a.s. and on $\Omega_{T,k,s}$, $\hat{f}_T \in \mathcal{F}_{T,k,s}$ and 
\begin{align*}
\left|(f^k_{J_k}(x) - f_0(x))^2 - (\hat{f}_T(x) - f_0(x))^2\right| &\leq \left|\hat{f}_T(x) - f^k_{J_k}(x)\right| \left|\hat{f}_T(x) + f^k_{J_k}(x) - 2f_0(x)\right| \leq 4F\delta_T. 
\end{align*}
Hence we obtain
\begin{align}\label{r-g-k-sum}
\sum_{k=0}^{\infty}I_{r,k} 
&\leq {4F^2a_T \over 8F^2Ta_T}\left(\Ep\left[\sum_{\ell=0}^{\mu_T-1}\sum_{t \in I_{1,\ell}}(\hat{f}_T(X_t^{\ast}) - f_0(X_t^{\ast}))^2\right] + \Ep\left[\sum_{\ell=0}^{\mu_T-1}\sum_{t \in I_{1,\ell}}(\hat{f}_T(X_t) - f_0(X_t))^2\right]\right) \nonumber \\
&\quad + {32F^3a_T^2\mu_T\delta_T \over 8F^2Ta_T} + {32F^4\cdot 4a_T^2\mu_T\beta(a_T) \over 8F^2Ta_T} \nonumber \\
&= {1 \over 2}\left(\Ep\left[{1 \over T}\sum_{\ell=0}^{\mu_T-1}\sum_{t \in I_{1,\ell}}(\hat{f}_T(X_t^{\ast}) - f_0(X_t^{\ast}))^2\right] + \Ep\left[{1 \over T}\sum_{\ell=0}^{\mu_T-1}\sum_{t \in I_{1,\ell}}(\hat{f}_T(X_t) - f_0(X_t))^2\right]\right) \nonumber \\
&\quad + 2F\delta_T + 8F^2\beta(a_T). 
\end{align}

Now we evaluate $I_{B,k}$. 
If $\bar b_k>0$, applying Lemmas E.2 and E.4 with $y=\sqrt{\bar b_k}$, we obtain $\Ep\left[\exp\left(2B_{j,k}^2\right)\sqrt{\bar b_k}/\sqrt{b_{j,k}+\bar b_k}\right]\leq1$.
Hence
\ba{
\Ep\left[\exp\left(2B_{J_k,k}^2\right)\sqrt{\bar b_k}/\sqrt{b_{J_k,k}+\bar b_k}\right]
\leq\Ep\left[\max_{1\leq j\leq \mathcal N_k}\exp\left(2B_{j,k}^2\right)\sqrt{\bar b_k}/\sqrt{b_{j,k}+\bar b_k}\right]\leq\mathcal N_k.
}
Therefore, by the Cauchy-Schwarz inequality, we obtain
\ba{
\E[\exp(B_k^2)]
\! \leq \! \sqrt{\Ep\! \left[\exp\left(\! 2B_{J_k,k}^2 \!\right)\!\! \sqrt{\bar b_k}/\! \sqrt{b_{J_k,k}+\bar b_k}\right]\! \Ep\!\left[\!\sqrt{b_{J_k,k} + \bar b_k}/\sqrt {\bar b_k}\right]}
\! \leq \! \sqrt{\mathcal N_k\Ep \! \left[\! \sqrt{b_{J_k,k} + \bar b_k}/\! \sqrt {\bar b_k}\right]}.
}
Since $\Ep\left[\sqrt{b_{J_k,k} + \bar b_k}/\sqrt {\bar b_k}\right]
\leq\sqrt{\Ep\left[(b_{J_k,k} + \bar b_k)/\bar b_k\right]}
\leq\sqrt 2$, we conclude
\ben{\label{Bk-bound}
\E[\exp(B_k^2)]\leq2^{1/4}\sqrt{\mathcal N_k}.
}
This inequality also holds when $\bar b_k=0$ because we always have $B_k=0$ in such a case. 
Thus, for $k\geq1$, we have
\begin{align}\label{B-k-bound1}
&\Ep\left[{16F^2a_T \over T}B_{k}^2 - \mathbf{1}_{\Omega_{T,k,s}}{1 \over 4}J_T(\hat{f}_T)\right] \leq \int_{0}^{\infty}\Prob\left({16F^2a_T \over T}B_{k}^2 - \mathbf{1}_{\Omega_{T,k,s}}{1 \over 4}J_T(\hat{f}_T) >x\right)dx \nonumber \\
&\leq 2^{1/4}\sqrt{\mathcal{N}_k}\int_{0}^{\infty}\exp\left(-{T(x + 2^{k-3}s) \over 16F^2a_T}\right)dx = 2^{1/4}\sqrt{\mathcal{N}_k} \cdot {16F^2a_T \over T}\exp\left(-{2^{k-3}sT \over 16F^2a_T}\right), 
\end{align}
where the second inequality follows from Markov's inequality and \eqref{Bk-bound}. Recall that $\delta_T$ is defined by \eqref{def:delta}. 
Then by Lemma E.6,
\begin{equation}\label{eq:n-est}
\log \mathcal{N}_k \leq {2^k s \over \lambda_T}(L_T+1)\log \left({(L_T+1)(N_T + 1)B_T \over T^{-1}}\right)
\leq C_1{2^k s \over \lambda_T}\log^{1+\nu_0} T,
\end{equation}
where $C_1$ is a positive constant depending only on ($C_L$, $C_N$, $C_B$, $\nu_0$, $\nu_1$, $\nu_2$). Since $a_T/\iota_\lambda(T)\to 0$ as $T \to\infty$, there is a constant $T_0$ depending only on $C_1$ and $\iota_\lambda$ such that $(C_1\log^{1+\nu_0} T)/\lambda_T \leq T/(128F^2a_T)$ whenever $T\geq T_0$. For such $T$, we have
\begin{align*}
2^{1/4}\sqrt{\mathcal{N}_k} \cdot {16F^2a_T \over T}\exp\left(-{2^{k-3}sT \over 16F^2a_T}\right) &\leq {2^{1/4}\cdot16F^2a_T \over T}\exp\left({2^ksT \over 256F^2a_T} - {2^ksT \over 128F^2a_T}\right)\\ 
&= {2^{1/4}\cdot16F^2a_T \over T}\exp\left(- {2^ksT \over 256F^2a_T}\right).
\end{align*}
For $k=0$, 
\begin{align}\label{B-k-bound2}
&\Ep\left[{16F^2a_T \over T}B_{0}^2 - \mathbf{1}_{\Omega_{T,0,s}}{1 \over 4}J_T(\hat{f}_T)\right] \leq {16F^2a_T \over T}\Ep[B_0^2] = {16F^2a_T \over T}\log (\exp(\Ep[B_0^2])) \nonumber \\
&\leq {16F^2a_T \over T}\log (\Ep[\exp(B_0^2)]) \leq {16F^2a_T \over T}\log (2^{1/4}\sqrt{\mathcal{N}_0}) \leq {8F^2a_T \over T}(1 + \log \mathcal{N}_0 ) \leq {8F^2a_T \over T} + {s \over 16}. 
\end{align}
For the second inequality, we used Jensen's inequality and for the last inequality, we used $\log \mathcal{N}_0 \leq s T/(128F^2a_T)$. Combining (\ref{B-k-bound1}) and (\ref{B-k-bound2}), we have
\begin{align}\label{B-g-k-sum}
\sum_{k=0}^{\infty}\! I_{B,k} 
&\leq {8F^2a_T \over T} \!+\! {s \over 16} \!+\! {2^{1/4} \!\cdot \!16F^2a_T \over T}\sum_{k=1}^{\infty}\exp\left(- {2^ksT \over 256F^2a_T}\right) \leq  {s \over 16} \!+\! {8F^2a_T \over T}\left(\!1 \!+\! {2^{5/4}\exp\left(-{sT \over 128F^2a_T}\right) \over 1- \exp\left(-{sT \over 256F^2a_T}\right)}\!\right).
\end{align}
Therefore, (\ref{r-g-k-sum}) and (\ref{B-g-k-sum}) yield
\begin{align*}
&\sum_{k=0}^{\infty}\Ep\left[\mathbf{1}_{\Omega_{T,k,s}}\left({1 \over T}\sum_{\ell=0}^{\mu_T-1}(g_{J_k,\ell}^{\ast,k} - g_{J_k,\ell}^k) - {1 \over 4}J_T(\hat{f}_T)\right) \right] \leq \sum_{k=0}^{\infty}(I_{r,k} + I_{B,k})\\
&\leq {1 \over 2}\left(\Ep\left[{1 \over T}\sum_{\ell=0}^{\mu_T-1}\sum_{t \in I_{1,\ell}}(\hat{f}_T(X_t^{\ast}) - f_0(X_t^{\ast}))^2\right] + \Ep\left[{1 \over T}\sum_{\ell=0}^{\mu_T-1}\sum_{t \in I_{1,\ell}}(\hat{f}_T(X_t) - f_0(X_t))^2\right]\right) \nonumber \\
&\quad + 2F\delta_T + 8F^2\beta(a_T) + {s \over 16} + {8F^2a_T \over T}\left(1 + {2^{5/4}\exp\left(-{sT \over 128F^2a_T}\right) \over 1- \exp\left(-{sT \over 256F^2a_T}\right)}\right).
\end{align*}
Likewise, we have
\begin{align*}
&\sum_{k=0}^{\infty}\Ep\left[\mathbf{1}_{\Omega_{T,k,s}}\left({1 \over T}\sum_{\ell=0}^{\mu_T-1}(h_{J_k,\ell}^{\ast,k} - h_{J_k,\ell}^k) - {1 \over 4}J_T(\hat{f}_T)\right) \right]\\ 
&\leq {1 \over 2}\left(\Ep\left[{1 \over T}\sum_{\ell=0}^{\mu_T-1}\sum_{t \in I_{2,\ell}}(\hat{f}_T(X_t^{\ast}) - f_0(X_t^{\ast}))^2\right] + \Ep\left[{1 \over T}\sum_{\ell=0}^{\mu_T-1}\sum_{t \in I_{2,\ell}}(\hat{f}_T(X_t) - f_0(X_t))^2\right]\right) \nonumber \\
&\quad + 2F\delta_T + 8F^2\beta(a_T) + {s \over 16} + {8F^2a_T \over T}\left(1 + {2^{5/4}\exp\left(-{sT \over 128F^2a_T}\right) \over 1- \exp\left(-{sT \over 256F^2a_T}\right)}\right).
\end{align*}
Hence, 
\begin{align*}
D &= \sum_{k=1}^{\infty}D_k \leq {1 \over 2}\left(R(\hat{f}_T,f_0) + \hat{R}(\hat{f}_T,f_0)\right) \nonumber \\
&\quad \quad \quad \quad \quad \!+\! 12F \delta_T \!+\! 16F^2\beta(a_T) \!+\! {s \over 8} \!+\! {16F^2a_T \over T}\left(1 \!+\! {2^{5/4}\exp\left(-{sT \over 128F^2a_T}\right) \over 1- \exp\left(-{sT \over 256F^2a_T}\right)}\right) \!+\! {16F^2 \over \mu_T} \!+\! 8F^2\beta(a_T). 
\end{align*}
Since $D = R(\hat{f}_T,f_0) - \hat{R}(\hat{f}_T,f_0) - {1 \over 2}\Ep[J_T(\hat{f}_T)]$, we obtain (\ref{pred-error-bound-adaptive}).  

\underline{Step 3} (Bounding the expected empirical error) In this step, we will show that for any $\bar{f} \in \mathcal{F}$, 
\begin{align}\label{emp-error-bound-adaptive}
\hat{R}(\hat{f}_T,f_0) + {1 \over 3}\Ep[J_T(\hat{f}_T)] 
&\leq 2\left(\bar{\Psi}_T(\hat{f}_T,\bar{f}) + R(\bar{f},f_0) + J_T(\bar{f})\right) + {22F^2 \over 3T(\log T)} \nonumber \\
&\quad + {96K^2C_\eta^2(\log T) \over T}\left(1 + {2^{5/4}\exp\left(-{5sT \over 1152K^2C_\eta^2(\log T)}\right) \over 1-\exp\left(-{5sT \over 2304K^2C_\eta^2(\log T)}\right)}\right) + 4C_\eta \delta_T + {5 \over 12}s
\end{align}
for $T\geq T_1$, where $\bar{\Psi}_T(\hat{f}_T,\bar{f}) = \Ep\left[\bar{Q}_T(\hat{f}_T) - \bar{Q}_T(\bar{f})\right]$ and $T_1>0$ is a constant depending only on ($C_L$, $C_N$, $C_B$, $\nu_0$, $\nu_1$, $\nu_2$, $K$, $C_\eta$, $\iota_\lambda$). 

For each $t=1,\dots,T$ and for any $\bar{f} \in \mathcal{F}$, we have
\begin{align*}
&\Ep[(\bar{f}(X_t) - Y_t)^2 + J_T(\bar{f})] - \Ep[(\hat{f}_T(X_t) - Y_t)^2 + J_T(\hat{f}_T)]\\
&= \Ep[(\bar{f}(X_t) - f_0(X_t))^2] + J_T(\bar{f}) - \Ep[(\hat{f}_T(X_t) - f_0(X_t))^2] - \Ep[J_T(\hat{f}_T)]\\ 
&\quad + 2\Ep[(\hat{f}_T(X_t) - f_0(X_t))\eta(X_t)v_t],
\end{align*}
where we used the fact $\Ep[\bar{f}(X_t)\eta(X_t)v_t] = \Ep[\bar{f}(X_t)\eta(X_t)\Ep[v_t|\mathcal{G}_{t-1}]] = 0$. Then we have
\begin{align}\label{emp-error-equation}
\hat{R}(\hat{f}_T,f_0) + {1 \over 6}\Ep[J_T(\hat{f}_T)] &= \bar{\Psi}_T(\hat{f}_T, \bar{f}) + R(\bar{f},f_0) + J_T(\bar{f}) - {5 \over 6}\Ep[J_T(\hat{f}_T)] \nonumber \\
&\quad + 2\Ep\left[{1 \over T}\sum_{t=1}^{T}(\hat{f}_T(X_t) - f_0(X_t))\eta(X_t)v_t\right].
\end{align}
Observe that 
\begin{align*}
&2\Ep\left[{1 \over T}\sum_{t=1}^{T}(\hat{f}_T(X_t) - f_0(X_t))\eta(X_t)v_t\right] = {2 \over T}\sum_{k=0}^{\infty}\Ep\left[\mathbf{1}_{\Omega_{T,k,s}}\sum_{t=1}^{T}(\hat{f}_T(X_t) - f_0(X_t))\eta(X_t)v_t\right]\\
&= {2 \over T}\!\sum_{k=0}^{\infty}\Ep\left[\mathbf{1}_{\Omega_{T,k,s}}\!\!\sum_{t=1}^{T}(f^k_{J_k}(X_t) - f_0(X_t))\eta(X_t)v_t\right] \!+\! {2 \over T}\sum_{k=0}^{\infty}\Ep\left[\mathbf{1}_{\Omega_{T,k,s}}\!\!\sum_{t=1}^{T}(\hat{f}_T(X_t) - f^k_{J_k}(X_t))\eta(X_t)v_t\right].
\end{align*}
Since $\sup_x|\hat{f}_T(x) - f^k_{J_k}(x)| \leq \delta_T$ on $\Omega_{T,k,s}$ and ${2C_\eta \delta_T\over T}\Ep\left[\sum_{t=1}^{T}|v_t|\right] \leq  {2C_\eta \delta_T\over T}\sum_{t=1}^{T}\sqrt{\Ep\left[v_t^2\right]} =  2C_\eta\delta_T$, we have
\begin{align}\label{error-resid-bound1}
2\Ep\left[{1 \over T}\sum_{t=1}^{T}(\hat{f}_T(X_t) - f_0(X_t))\eta(X_t)v_t\right] &\leq {2 \over T}\sum_{k=0}^{\infty}\Ep\left[\mathbf{1}_{\Omega_{T,k,s}}\sum_{t=1}^{T}(f^k_{J_k}(X_t) - f_0(X_t))\eta(X_t)v_t\right] + 2C_\eta\delta_T.
\end{align}
Define $\eta_{j,k} := \mathbf{1}_{\Omega_{T,k,s}}{U_{j,k} \over 2\sqrt{V_{j,k}^2 + \Ep[V_{J_k,k}^2]}}$, $\eta_k := \eta_{J_k,k}$, 
\begin{align*}
U_{j,k} &:= \sum_{t=1}^{T}(f_j^k(X_t) - f_0(X_t))\eta(X_t)v_t,\ V_{j,k} := \mathbf{1}_{\Omega_{T,k,s}}\left(\sum_{t=1}^{T}(f_j^k(X_t) - f_0(X_t))^2\eta^2(X_t)(v_t^2+1)\right)^{1/2},
\end{align*}
where $\eta_{j,k}:=0$ if the denominator equals $0$. Observe that
\begin{align}\label{error-resid-bound2}
{2 \over T}\Ep \! \left[\! \mathbf{1}_{\Omega_{T,k,s}}\!\!\sum_{t=1}^{T}\!(f^k_{J_k}(X_t) \!-\! f_0(X_t)\!)\eta(\!X_t \!)v_t \!\right] &\!\leq  {2 \over T}\Ep \! \left[\mathbf{1}_{\Omega_{T,k,s}}\!\! \left(2\sqrt{V_{J_k, k}^2 + \Ep[V_{J_k, k}^2]} \eta_k\right)\right] \nonumber \\
&\leq {1 \over 6K^2C_\eta^2 T(\log T)}\Ep\left[V_{J_k, k}^2\right] + {96K^2C_\eta^2(\log T) \over T}\Ep[\mathbf{1}_{\Omega_{T,k,s}}\eta_k^2]. 
\end{align}
Combining (\ref{error-resid-bound1}) and (\ref{error-resid-bound2}), we have
\begin{align}\label{IV-Ieta-bound}
&2\Ep\left[{1 \over T}\sum_{t=1}^{T}(\hat{f}_T(X_t) - f_0(X_t))\eta(X_t)v_t\right] - {5 \over 6}\Ep[J_T(\hat{f}_T)] \nonumber \\
&\leq \sum_{k=0}^{\infty}{1 \over 6K^2C_\eta^2 T(\log T)}\Ep\left[V_{J_k, k}^2\right]  
+ \sum_{k=0}^{\infty}\Ep\left[\mathbf{1}_{\Omega_{T,k,s}}\left({96K^2C_\eta^2(\log T) \over T}\eta_k^2 - {5 \over 6}J_T(\hat{f}_T)\right)\right] + 2C_\eta \delta_T\nonumber \\
&=: \sum_{k=0}^{\infty}I_{V,k} + \sum_{k=0}^{\infty}I_{\eta,k} + 2C_\eta \delta_T. 
\end{align}

Now we evaluate $I_{V,k}$. Observe that 
\begin{align*}
&\Ep\left[V_{J_k, k}^2 \right] \leq C_\eta^2\Ep\left[\mathbf{1}_{\Omega_{T,k,s}}\sum_{t=1}^{T}(f^k_{J_k}(X_t) - f_0(X_t))^2(v_t^2+1)\right]\\
&= C_\eta^2\Ep\left[\mathbf{1}_{\Omega_{T,k,s}}\sum_{t=1}^{T}(\hat{f}_T(X_t) - f_0(X_t))^2(v_t^2+1)\right] + C_\eta^2\Ep\left[\mathbf{1}_{\Omega_{T,k,s}}\sum_{t=1}^{T}(f^k_{J_k}(X_t) - \hat{f}_T(X_t))^2(v_t^2+1)\right]\\
&\quad + 2C_\eta^2\Ep\left[\mathbf{1}_{\Omega_{T,k,s}}\sum_{t=1}^{T}(\hat{f}_T(X_t) - f_0(X_t))(f^k_{J_k}(X_t) - \hat{f}_T(X_t))(v_t^2+1)\right]\\
&\leq C_\eta^2\Ep\left[\mathbf{1}_{\Omega_{T,k,s}}\sum_{t=1}^{T}(\hat{f}_T(X_t) - f_0(X_t))^2(v_t^2+1)\right]\\
&\quad + 2FC_\eta^2\delta_T\Ep\left[\mathbf{1}_{\Omega_{T,k,s}}\sum_{t=1}^{T}(v_t^2+1)\right] + 4FC_\eta^2 \delta_T\Ep\left[\mathbf{1}_{\Omega_{T,k,s}}\sum_{t=1}^{T}(v_t^2+1)\right]\\
&= C_\eta^2\Ep\left[\mathbf{1}_{\Omega_{T,k,s}}\sum_{t=1}^{T}(\hat{f}_T(X_t) - f_0(X_t))^2(v_t^2+1)\right] + 6FC_\eta^2 \delta_T\Ep\left[\mathbf{1}_{\Omega_{T,k,s}}\sum_{t=1}^{T}(v_t^2+1)\right].
\end{align*}
From the same arguments to obtain (C.16) in the proof of Lemma C.2, we have \\$C_\eta^2\Ep\left[\sum_{t=1}^{T}(\hat{f}_T(X_t) - f_0(X_t))^2v_t^2\right] \leq 2C_\eta^2K^2T(\log T)\widehat{R}(\hat{f}_T, f_0) + 16C_\eta^2F^2K^2$. Hence we have
\begin{align}\label{I-Vk-bound}
\sum_{k=0}^{\infty}I_{V,k} &\leq  {1 \over 6K^2C_\eta^2 T(\log T)}\left(C_\eta^2(2K^2T(\log T)+1)\widehat{R}(\hat{f}_T, f_0) + 16C_\eta^2F^2K^2 + 6FC_\eta^2T\delta_T\right) \nonumber \\
&\leq {1 \over 2}\hat{R}(\hat{f}_T,f_0) + {11F^2 \over 3(\log T)}\delta_T.
\end{align}
For the second inequality, we used the inequalities $T\delta_T\geq1$, $F\geq1$ and 
\ben{\label{K-est}
1=\max_{1\leq t\leq T}\Ep[v_t^2]\leq\max_{1\leq t\leq T}K_t^2(\Ep[\exp(v_t^2/K_t^2)]-1)\leq K.
} 

Now we evaluate $I_{\eta,k}$. 
If $\Ep[V_{J_k,k}^2]>0$, applying Lemmas E.3 and E.4 with $y=\sqrt{\Ep[V_{J_k,k}^2]}$, we obtain $\Ep\left[\exp\left(2\eta_{j,k}^2\right)\sqrt{\Ep[V_{J_k,k}^2]}/\sqrt{V_{j,k}^2+\Ep[V_{J_k,k}^2]}\right]\leq1$.
Hence
\ba{
\Ep \! \left[\exp \! \left(\!2\eta_{J_k,k}^2 \!\right)\!\!\sqrt{\Ep[V_{J_k,k}^2]}/\!\!\sqrt{V_{J_k,k}^2\!+\!\Ep[V_{J_k,k}^2]}\right]
\! \leq \! \Ep \! \left[ \! \max_{1\leq j\leq \mathcal N_k}\!\!\! \exp \! \left(\!2\eta_{j,k}^2 \!\right)\!\!\sqrt{\Ep[V_{J_k,k}^2]}/\!\!\sqrt{V_{j,k}^2\!+\!\Ep[V_{J_k,k}^2]}\right]
\!\leq \!\mathcal N_k.
}
Therefore, by the Cauchy-Schwarz inequality, we obtain \\$\E[\exp(\eta_k^2)]
\leq\sqrt{\mathcal N_k\Ep\left[\sqrt{V_{J_k,k}^2 + \Ep[V_{J_k,k}^2]}/\sqrt {\Ep[V_{J_k,k}^2]}\right]}$. 

Since $\Ep\left[\sqrt{V_{J_k,k}^2 + \Ep[V_{J_k,k}^2]}/\sqrt {\Ep[V_{J_k,k}^2]}\right]
\leq\sqrt{\Ep\left[(V_{J_k,k}^2 + \Ep[V_{J_k,k}^2])/\Ep[V_{J_k,k}^2]\right]}
\leq\sqrt 2$, we conclude
\ben{\label{eta-k-bound}
\E[\exp(\eta_k^2)]\leq2^{1/4}\sqrt{\mathcal N_k}.
}
This inequality also holds when $\Ep[V_{J_k,k}^2]=0$ because we always have $\eta_k=0$ in such a case. Thus, for $k\geq1$, we have
\begin{align*}
&\Ep\left[\mathbf{1}_{\Omega_{T,k,s}}\!\! \left({96K^2C_\eta^2(\log T) \over T}\eta_k^2 - {5 \over 6}J_T(\hat{f}_T)\right)\right] \! \leq \! \int_{0}^{\infty}\!\!\! \Prob\left(\mathbf{1}_{\Omega_{T,k,s}}\!\!\left({96K^2C_\eta^2(\log T) \over T}\eta_{k}^2 - {5 \over 6}J_T(\hat{f}_T)\right)>x\right)dx\\
&\leq 2^{1/4}\sqrt{\mathcal{N}_k} \int_{0}^{\infty}\exp\left(-{T\left(x + {5\cdot 2^{k-2}s \over 3} \right) \over 96K^2C_\eta^2(\log T)}\right)dx = {2^{1/4}\cdot96K^2C_\eta^2 (\log T) \over T}\sqrt{\mathcal{N}_k}\exp\left(-{5 \cdot 2^{k}sT \over 1152K^2C_\eta^2(\log T)}\right),
\end{align*}
where the second inequality follows from Markov's inequality and \eqref{eta-k-bound}. Since $\iota_\lambda(x)\to\infty$ as $x\to\infty$ and $F\geq1$, there is a constant $T_1$ depending only on $C_1,K,C_\eta$ and $\iota_\lambda$ such that $(C_1\log^{1+\nu_0} T)/\lambda_T \leq 5T/(1152F^2)$ whenever $T\geq T_1$. For such $T$, we have by \eqref{eq:n-est}
\begin{align*}
&{2^{1/4} \! \cdot \!96K^2C_\eta^2 \!(\log T) \over T}\!\sqrt{\mathcal{N}_k}\!\exp \! \left(-{5 \! \cdot \! 2^{k}\! sT \over 1152K^2C_\eta^2 \!(\log T)}\!\right) \!\leq \! {2^{1/4}\! \cdot \! 96K^2C_\eta^2 \!(\log T) \over T}\!\exp \! \left(-{5 \!\cdot \!  2^{k}\!sT \over 2304K^2C_\eta^2 \!(\log T)}\!\right). 
\end{align*}
For $k =0$, 
\begin{align*}
\Ep\left[\mathbf{1}_{\Omega_{T,0,s}}\left({96K^2C_\eta^2(\log T) \over T}\eta_k^2 - {5 \over 6}J_T(\hat{f}_T)\right)\right] &\leq {96K^2C_\eta^2(\log T) \over T}\Ep[\eta_0^2] \leq {48K^2C_\eta^2(\log T) \over T}(1 + \log \mathcal{N}_0)\\
&\leq {48K^2C_\eta^2(\log T) \over T} + {5s \over 24}. 
\end{align*}
Then we have
\begin{align}\label{I-eta-k-bound}
\sum_{k=0}^{\infty}I_{\eta,k} &\leq {48K^2C_\eta^2(\log T) \over T} + {5s \over 24} + {2^{1/4}\cdot96K^2C_\eta^2 (\log T) \over T}\sum_{k=1}^{\infty}\exp\left(-{5 \cdot 2^{k}sT \over 2304K^2C_\eta^2(\log T)}\right) \nonumber \\
&\leq {5s \over 24} + {48K^2C_\eta^2 (\log T) \over T}\left(1 + {2^{5/4}\exp\left(-{5sT \over 1152K^2C_\eta^2(\log T)}\right) \over 1-\exp\left(-{5sT \over 2304K^2C_\eta^2(\log T)}\right)}\right).
\end{align}

Combining (\ref{emp-error-equation}), (\ref{IV-Ieta-bound}), (\ref{I-Vk-bound}), and (\ref{I-eta-k-bound}), we obtain (\ref{emp-error-bound-adaptive}). 

\underline{Step 4} (Conclusion)

From (\ref{pred-error-bound-adaptive}) and (\ref{emp-error-bound-adaptive}) with $a_T = C_{2,\beta}^{-1}\log T$ and $s = (C_{2,\beta}^{-1}+1)(F^2 \vee K^2C_\eta^2)(\log T)/T$, we have 
\begin{align*}
R(\hat{f}_T,f_0) &\leq 6\left(\bar{\Psi}_T(\hat{f}_T,\bar{f}) + R(\bar{f},f_0) + J_T(\bar{f})\right) + {22F^2 \over \log T}\delta_T \nonumber \\
&\quad + {288K^2C_\eta^2(\log T) \over T}\left(1 + {2^{5/4}\exp\left(-{5 \over 1152}\right) \over 1-\exp\left(-{5 \over 2304}\right)}\right) + 12C_\eta \delta_T + {5(C_{2,\beta}^{-1}+1)(F^2 \vee K^2C_\eta^2)(\log T) \over 4T}\\
&\quad + {32F^2 C_{2,\beta}^{-1}(\log T) \over T}\left(1 + {2^{5/4}\exp\left(-{1 \over 128}\right) \over 1-\exp\left(-{1 \over 256}\right)}\right) + 24F \delta_T + {(C_{2,\beta}^{-1}+1)(F^2 \vee K^2C_\eta^2)(\log T) \over 4T}\\ 
&\quad + {64F^2C_{2,\beta}^{-1}(\log T) \over T} + {96F^2C_{1,\beta} \over T} 
\end{align*}
whenever $T\geq T_0\vee T_1$. 
Here, we used the fact that $e^{-2x}/(1-e^{-x})$ is decreasing for $x>0$. Therefore, we obtain the desired result. \qed

\begin{remark}[Generalization error bound for SPDNN under $\beta$-mixing coefficients with polynomial decay]
From (\ref{pred-error-bound-adaptive}) and (\ref{emp-error-bound-adaptive}), we have
\begin{align}
R(\widehat{f}_T, f_0) &\leq 6\left(\bar{\Psi}_T^{\mathcal{F}}(\widehat{f}_T) + \inf_{f \in \mathcal{F}}(R(f,f_0) + J_T(f))\right) \nonumber \\
&\quad + {22F^2 \over T(\log T)} + 12C_\eta \delta_T + 24F\delta_T + {3 \over 2}s  + {32F^2 \over \mu_T} + 48\beta(a_T) \nonumber \\
&\quad + {288K^2C_\eta^2 \log T \over T}\left(1 + {2^{5/4}\exp\left(-{5sT \over 1152K^2 C_\eta^2 \log T}\right) \over 1-\exp\left(-{5sT \over 2304K^2 C_\eta^2 \log T}\right)}\right) \nonumber \\
&\quad + {32F^2 a_T \over T}\left(1 + {2^{5/4}\exp\left(-{sT \over 128F^2 a_T}\right) \over 1-\exp\left(-{sT \over 256F^2a_T}\right)}\right). \label{SPDNN-beta-poly}
\end{align}
If the $\beta$-mixing coefficient decays polynomially fast, that is, $\beta(t)\leq C_\beta t^{-\alpha}$ for some $\alpha>0$, then, letting $s = {(K^2 C_\eta^2 \vee F^2)a_T \over T}$ and $a_T = T^{{1 \over \alpha+1}}$ in (\ref{SPDNN-beta-poly}), we obtain
\begin{align*}
R(\widehat{f}_T, f_0) &\leq 6\left(\bar{\Psi}_T^{\mathcal{F}}(\widehat{f}_T) + \inf_{f \in \mathcal{F}}(R(f,f_0) + J_T(f))\right) + \breve{C}F^2\left({1 + \log T \over T} + T^{-{\alpha \over \alpha +1}}\right),
\end{align*}
for $T > \breve{T}$ where $\breve{T}>0$ is a constant depending only on $(C_\eta, C_L, C_N, C_B, \nu_0, \nu_1, \nu_2, K, \iota_\lambda)$ and $\breve{C}$ is a constant depending only on $(C_\eta, C_\beta, C_\tau, K)$. 
\end{remark}

\section*{Supplementary Material}
The supplementary material includes discussion of our main results (Section \ref{Appendix: discuss}), proof of auxiliary lemmas (Section \ref{Appendix:aux lem}), proofs for Section \ref{sec:ar-model} (Section \ref{Appendix: minimax}), and technical tools (Section \ref{Appendix: technical tools}). In what follows, we set $\beta(t) = \beta_X(t)$.

\section{Discussion}\label{Appendix: discuss}
In this section, we provide a more detailed discussion of our theoretical results from several perspectives.

\subsection{Dependence structure}
In our paper, we assume that the process $\{X_t\}_{t=1}^{T}$ is $\beta$-mixing. On the other hand, physical dependence is also a commonly utilized dependence structure. While obtaining similar results under physical dependence may be possible, we choose to leave it for future research as the proof approach would be entirely different. In time series analysis, $\beta$-mixing is a commonly used dependence structure, and it is known that many stationary and nonstationary time series models exhibit exponential $\beta$-mixing. On this point, we refer to \cite{Mo88} for vector ARMA($p,q$) process, \cite{Bo98} for GARCH($p,q$) process, \cite{ChCh00} for nonlinear AR($p$)-ARCH($q$) process, and \cite{Vo12} for time-varying nonlinear AR($p$)-ARCH($q$) process. However, by assuming $\beta$-mixing, time series models that are $\Psi$-weakly dependent (cf. \cite{DDLJLP07}) but not $\beta$-mixing fall outside the scope of our results. We refer to \cite{Do94} and \cite{DDLJLP07} for examples of non-mixing time series models and leave the extension of our results to that case for future research.

\subsection{Theorems 3.1 and 3.2}
The main difficulties in proving Theorems 3.1 and 3.2 can be summarized as follows: In \cite{Sc20}, it is assumed that the error terms in the non-parametric regression model follow independent normal distributions. Additionally, Lemma 4 (III) in his paper, relies on a maximal inequality for self-normalized random variables to obtain an upper bound for $\hat{R}(\hat{f}_T,f_0)=\Ep[T^{-1}\sum_{t=1}^{T}(\hat{f}_T(X_t) - f_0(X_t))^2]$, but this inequality is heavily dependent on the normality assumption of the error terms. On the other hand, in our proof, we consider self-normalized martingale differences to establish an upper bound for $\hat{R}(\hat{f}_T,f_0)$, and in our model, we do not assume symmetric distributions or normality for the error terms, making it impossible to apply his approach. Specifically, in our setup, we need to apply an exponential inequality to self-normalized martingales. Existing research only provides such inequalities for martingale differences with symmetric distributions. Therefore, extending existing results to martingale differences without assuming the symmetricity of their distributions is necessary. This extension is achieved by considering the continuous-time embedding of discrete-time martingale differences (Lemma E.3). In addition, \cite{Sc20} does not handle penalized estimators, so the proof of Theorem 3.2 has an additional difficulty. In particular, since we need to optimize the objective function over the entire class $\mathcal{F}_{\sigma}(L,N,B,F)$ and this class has a too large covering number, we need to show that its ``largeness'' is appropriately controlled by the penalty term. Regarding this point, \cite{OhKi22} handle a sparse-penalized DNN estimator for nonparametric regression models with i.i.d.~errors, but their proofs heavily depend on the theory for i.i.d.~data in \cite{GKKW02} so extending their approach to our framework seems to require substantial work. 
We have thus developed a rather different approach by adapting \cite{Sc20}'s proof to penalized estimators.

\subsection{Results in Section 4}
The main difficulties in proving results in Section 4 can be summarized as follows: To prove (near) minimax optimality for estimating mean functions belonging to some function class, we need to establish both upper and lower bounds for the estimation error \textit{uniformly valid} over this function class. Development of the lower bound is basically similar to \cite{Sc20}, but we need substantial work to establish a uniform upper bound because our abstract bounds developed in Section 3 contains bound for the $\beta$-mixing coefficients: We need a uniformly valid upper bound for the $\beta$-mixing coefficients, but such a bound is rarely available in the literature. 
To resolve this issue, we prove Lemma 4.1 which provides a uniform upper bound for the $\beta$-mixing coefficients over a fairly large class of mean functions.

\section{Auxiliary Lemmas}\label{Appendix:aux lem}
Now we show two Lemmas \ref{lem: error-gap-bound} and \ref{lem: emp-error-bound}. Note that the results do not require the estimator $\hat{f}_{T}$ to take values in $\mathcal{F}_{\sigma}(L,N,B,F,S)$ and hence would be of independent interest.

\begin{lemma}\label{lem: error-gap-bound}
Let $\delta>0$ and suppose that there exists an integer $\mathcal{N}_T$ such that $\mathcal{N}_T \geq N(\delta, \mathcal{F}, \|\cdot\|_{\infty}) \vee \exp(10)$. Also, let $a_T$ be a positive number such that $\mu_T := \lfloor T/(2a_T)\rfloor > 0$. In addition, suppose that there is a number $F \geq 1$ such that $\|f\|_{\infty} \leq F$ for all $f \in \mathcal{F} \cup \{f_0\}$. Then, for all $\varepsilon \in  (0, 1]$,
\begin{align*}
R(\hat{f}_{T},f_0) &\leq (1+\varepsilon)\widehat{R}(\hat{f}_{T},f_0) + {21(1+\varepsilon)^2 \over \varepsilon}F^2{\log \mathcal{N}_T \over \mu_T} + {4F^2 \over \mu_T} + 4(2+\varepsilon)F^2\beta(a_T) + 4(2+\varepsilon)F\delta.
\end{align*}
\end{lemma}

\begin{proof}

Let $\{f_{1},\dots, f_{\mathcal{N}_T}\}$ be a $\delta$-covering of $\mathcal{F}$ with respect to $\|\cdot\|_{\infty}$ and define a random variable $J$ taking values in $\{1,\dots, \mathcal{N}_T\}$ such that $\|\hat{f}_{T} - f_{J}\|_{\infty} \leq \delta$.  

\underline{Step 1} (Reduction to independence) We rely on the coupling technique for $\beta$-mixing sequences to construct independent blocks; cf. \cite{Ri13}. For $\ell = 0,\dots, \mu_T-1$, let $I_{1,\ell} = \{2\ell a_T+1, \dots, (2\ell+1) a_T \}$, $I_{2,\ell} = \{(2\ell+1)a_T+1, \dots, 2(\ell+1)a_T\}$. Define
\begin{align*}
\tilde{g}_{\ell} &:= (\tilde{g}_{1,\ell},\dots, \tilde{g}_{\mathcal{N}_T,\ell})' = \left(\sum_{t \in I_{1,\ell}}(f_1(X_t) - f_0(X_t))^2, \dots, \sum_{t \in I_{1,\ell}}(f_{\mathcal{N}_T}(X_t) - f_0(X_t))^2\right)',\\
\tilde{g}_{\ell}^{\ast} &:= (\tilde{g}_{1,\ell}^{\ast},\dots, \tilde{g}_{\mathcal{N}_T,\ell}^{\ast})' = \left(\sum_{t \in I_{1,\ell}}(f_1(X_t^{\ast}) - f_0(X_t^{\ast}))^2, \dots, \sum_{t \in I_{1,\ell}}(f_{\mathcal{N}_T}(X_t^{\ast}) - f_0(X_t^{\ast}))^2\right)'.
\end{align*}
In the following, we extend the probability space if necessary and assume that there is a sequence $\{U_{\ell}\}_{\ell = 1}^{\infty}$ of i.i.d.~uniform random variables over $[0,1]$ independent of $X$. 

We will show that there exist two sequences of independent $\mathbb{R}^{\mathcal{N}_T}$-valued random variables $\{g_{\ell}\}_{\ell = 0}^{\mu_T-1}$ and $\{g_{\ell}^{\ast}\}_{\ell = 0}^{\mu_T-1}$  such that for all $\ell = 0,\dots,\mu_T-1$, 
\begin{align}
\left|\Ep[g_{j,\ell}] - \Ep[\tilde{g}_{j,\ell}]\right| &\leq 4F^2a_T\beta(a_T), \label{expect-block1}\\
\left|\Ep[g_{j,\ell}^{\ast}] - \Ep[\tilde{g}_{j,\ell}^{\ast}]\right| &\leq 4F^2a_T\beta(a_T). \label{expect-block2}
\end{align}
where $g_{j,\ell}$ and $g_{j,\ell}^{\ast}$ be the $j$-th component of $g_{\ell}$ and $g_{\ell}^{\ast}$, respectively. We only prove (\ref{expect-block1}) since the proof of (\ref{expect-block2}) is similar. 

First we will show that there exist a sequences $\{g_{\ell}\}_{\ell = 0}^{\mu_T-1}$ of independent random vectors in $\mathbb{R}^{\mathcal{N}_T}$ such that 
\begin{align*}
g_{\ell} &\stackrel{\mathcal{L}}{=} \tilde{g}_{\ell},\ \Prob(g_{\ell} \neq \tilde{g}_{\ell}) \leq \beta(a_T)\ \text{for $0 \leq \ell \leq \mu_T-1$}.
\end{align*}
For all $\ell_{1}, \ell_{2}$, define the $\sigma$-field $\mathcal{A}(\ell_{1},\ell_{2})$ generated by $\{X_{t}\}_{t \in I(\ell_{1},\ell_{2})}$ where $I(\ell_{1},\ell_{2}):= \bigcup_{\ell = \ell_{1}}^{\ell_{2}}I_{1,\ell}$. From the definition of $\tilde{g}_{\ell}$, we find that $\sigma(\tilde{g}_{\ell}) \subset \mathcal{A}(\ell, \ell)$ for all $\ell$. Applying Lemma \ref{lem: block}, there exists a random vector $g_{\ell}$ such that $g_{\ell} \stackrel{\mathcal{L}}{=} \tilde{g}_{\ell}$, independent of $\mathcal{A}(0,\ell-1)$, and $\Prob(g_{\ell} \neq \tilde{g}_{\ell}) \leq \beta(a_T)$. Moreover, $g_{\ell}$ is measurable with respect to the $\sigma$-field generated by $\mathcal{A}(0,\ell)$ and $U_{\ell}$. Therefore, for any $\ell$, $g_{\ell}$ is independent of $\{g_{\ell'}\}_{\ell'=1}^{\ell-1}$, since for $\ell_1, \ell_{2}$ with $\ell_1 < \ell_2$, $g_{\ell_2}$ is independent of the $\sigma$-field generated by $\mathcal{A}(0,\ell_1)$ and $U_{\ell_1}$. This implies that $\{g_{\ell}\}_{\ell = 0}^{\mu_T-1}$ is a sequence of independent random variables.

Next we will show (\ref{expect-block1}). By definition we have $0\leq\tilde{g}_{j,\ell} \leq 4F^2a_T$ for all $j$. Since $g_{\ell}$ has the same law as $\tilde{g}_{\ell}$, we also have $0\leq g_{j,\ell} \leq 4F^2a_T$ a.s. for all $j$. Consequently, we have
\begin{align*}
\left|\Ep[g_{j,\ell}] - \Ep[\tilde{g}_{j,\ell}]\right| &\leq \Ep\left[\left|g_{j,\ell} - \tilde{g}_{j,\ell}\right|\mathbf{1}_\{g_{j,\ell} \neq \tilde{g}_{j,\ell}\}\right] \leq 4F^2a_T\Prob(g_{j,\ell} \neq \tilde{g}_{j,\ell}) \leq 4F^2a_T\beta(a_T).
\end{align*}

\underline{Step 2} (Bounding the difference of the sum of independent blocks) In this step, we will show 
\begin{align}\label{indep-block-sum-bound}
\Ep\left[\sum_{\ell = 0}^{\mu_T-1}g_{J,\ell}^{\ast}\right] &\leq (1+\varepsilon)\Ep\left[\sum_{\ell = 0}^{\mu_T-1}g_{J,\ell}\right] + {21(1+\epsilon)^2 \over \varepsilon}F^2a_T\log \mathcal{N}_T.
\end{align}
for all $\varepsilon \in (0,1]$. 

For $j = 1,\dots, \mathcal{N}_T$, define
\begin{align*}
r_j &:= \left({4F^2a_T \log \mathcal{N}_T \over \mu_T} \vee {1 \over \mu_T}\sum_{\ell = 0}^{\mu_T-1}E[g_{j,\ell}^{\ast}]\right)^{1/2},\\
B& := \max_{1 \leq j\leq \mathcal{N}_T}\left|{\sum_{\ell = 0}^{\mu_T-1}(g_{j,\ell}^{\ast} - g_{j,\ell}) \over 2Fr_j}\right|.
\end{align*}
By definition, we have
\begin{align*}
\left|\Ep\left[\sum_{\ell = 0}^{\mu_T-1}g_{J,\ell}^{\ast}\right] - \Ep\left[\sum_{\ell = 0}^{\mu_T-1}g_{J,\ell}\right]\right| &\leq \Ep\left[\left|\sum_{\ell = 0}^{\mu_T-1}(g_{J,\ell}^{\ast} - g_{J,\ell})\right|\right] \leq 2F\Ep\left[r_{J}B\right].
\end{align*}
By the Cauchy-Schwarz inequality, we have 
\begin{align*}
E[r_{J}B] &\leq \Ep\left[\left({1 \over \mu_T}\sum_{\ell = 0}^{\mu_T-1}E[g_{J,\ell}^{\ast}]\right)^{1/2}B\right] + \left({4F^2a_T \log \mathcal{N}_T \over \mu_T}\right)^{1/2}\Ep\left[B\right]\\
&\leq \Ep\left[{1 \over \mu_T}\sum_{\ell = 0}^{\mu_T-1}E[g_{J,\ell}^{\ast}]\right]^{1/2}\Ep[B^2]^{1/2} + 2F\sqrt{{a_T \log \mathcal{N}_T \over \mu_T}}\Ep\left[B\right]\\
&= \Ep\left[\sum_{\ell = 0}^{\mu_T-1}E[g_{J,\ell}^{\ast}]\right]^{1/2}\left({\Ep[B^2] \over \mu_T}\right)^{1/2} + 2F\sqrt{{a_T \log \mathcal{N}_T \over \mu_T}}\Ep\left[B\right].
\end{align*} 
Therefore, we have 
\begin{align*}
\left|\Ep\left[\sum_{\ell = 0}^{\mu_T-1}g_{J,\ell}^{\ast}\right] - \Ep\left[\sum_{\ell = 0}^{\mu_T-1}g_{J,\ell}\right]\right| &\leq 2\Ep\left[\sum_{\ell = 0}^{\mu_T-1}E[g_{J,\ell}^{\ast}]\right]^{1/2}\left({F^2\Ep[B^2] \over \mu_T}\right)^{1/2} + 4F^2\sqrt{{a_T \log \mathcal{N}_T \over \mu_T}}\Ep\left[B\right].
\end{align*}
Hence, using (C.4) in \cite{Sc20}, we have
\begin{align}\label{indep-block-bound2}
\Ep\left[\sum_{\ell = 0}^{\mu_T-1}g_{J,\ell}^{\ast}\right] &\leq (1+\varepsilon)\Ep\left[\sum_{\ell = 0}^{\mu_T-1}g_{J,\ell}\right] + 4F^2(1+\varepsilon)\sqrt{{a_T \log \mathcal{N}_T \over \mu_T}}\Ep\left[B\right] + {(1+\varepsilon)^2 \over \varepsilon}{F^2\Ep[B^2] \over \mu_T}.
\end{align}
Now we show 
\begin{align}\label{T-expect-bound}
\Ep[B] &\leq 3\sqrt{a_T\mu_T\log \mathcal{N}_T},\ \Ep[B^2] \leq 9a_T\mu_T\log \mathcal{N}_T. 
\end{align}
Define $\gamma := {1 + \sqrt{37} \over 3}$ and $\alpha:= \gamma\sqrt{a_T\mu_T \log \mathcal{N}_T}$. Note that $\gamma$ solves the equation $3\gamma^2 - 2\gamma - 12=0$. For all $j,\ell$, we have by construction
\begin{align*}
\left|{g_{j,\ell}^{\ast} - g_{j,\ell} \over 2Fr_j}\right| &\leq {4F^2a_T \over 2F\sqrt{4F^2a_T \log \mathcal{N}_T \over \mu_T}}  = \sqrt{a_T\mu_T \over \log \mathcal{N}_T} 
\end{align*}
and 
\begin{align*}
\Var\left({g_{j,\ell}^{\ast} - g_{j,\ell} \over 2Fr_j}\right) &\leq {2\Ep[(g_{j,\ell}^{\ast})^2] \over 4F^2r_j^2} \leq {8F^2a_T\Ep[g_{j,\ell}^{\ast}] \over 4F^2{1 \over \mu_T}\sum_{\ell' = 0}^{\mu_T-1}\Ep[g_{j,\ell'}^{\ast}]} = {2a_T\mu_T\Ep[g^*_{j,\ell}] \over \sum_{\ell' = 0}^{\mu_T-1}\Ep[g_{j,\ell'}^{\ast}]}.
\end{align*}
Then 
\[
\sum_{\ell = 0}^{\mu_T-1}\Var\left({g_{j,\ell}^{\ast} - g_{j,\ell} \over 2Fr_j}\right) \leq 2a_T\mu_T. 
\]
Using Bernstein's inequality (cf. Lemma 2.2.9 in \cite{vaWe96}), we have for all $x \geq \alpha$, 
\begin{align*}
\Prob\left(\left|{\sum_{\ell = 0}^{\mu_T-1}(g_{j,\ell}^{\ast} - g_{j,\ell}) \over 2Fr_j}\right| \geq x\right) &\leq 2\exp\left(-{{1 \over 2}x^2 \over 2a_T\mu_T + {1 \over 3}\sqrt{{a_T\mu_T \over \log \mathcal{N}_T}}x}\right)\\
&\leq 2\exp\left(-{{1 \over 2}\gamma\sqrt{a_T\mu_T \log \mathcal{N}_T} \over 2a_T\mu_T + {1 \over 3}\sqrt{{a_T\mu_T \over \log \mathcal{N}_T}}\gamma\sqrt{a_T\mu_T \log \mathcal{N}_T}}x\right)\\
&= 2\exp\left(-{3\gamma^2 \over 12 + 2\gamma}{\log \mathcal{N}_T \over \alpha}x\right) = 2\exp\left(-{\log \mathcal{N}_T \over \alpha}x\right).
\end{align*}
Thus, 
\begin{align*}
\Prob\left(B \geq x\right) &\leq \sum_{j=1}^{\mathcal{N}_T}\Prob\left(\left|{\sum_{\ell = 0}^{\mu_T-1}(g_{j,\ell}^{\ast} - g_{j,\ell}) \over 2Fr_j}\right| \geq x\right) \leq 2\mathcal{N}_T\exp\left(-{\log \mathcal{N}_T \over \alpha}x\right).
\end{align*}
Hence we have
\begin{align*}
E[B] &= \int_{0}^{\infty}\Prob(B \geq x)dx \leq \alpha + \int_{\alpha}^{\infty}P(B \geq x)dx\\
&\leq \alpha + 2\mathcal{N}_T{\alpha \over \log \mathcal{N}_T}\exp\left(-{\log \mathcal{N}_T \over \alpha}\alpha\right) = \left(1 + {2 \over \log \mathcal{N}_T}\right)\gamma \sqrt{a_T\mu_T \log \mathcal{N}_T}\\
&\leq 3\sqrt{a_T\mu_T \log \mathcal{N}_T}
\end{align*}
and 
\begin{align*}
E[B^2] &= 2\int_{0}^{\infty}x\Prob(B \geq x)dx \leq 2\int_{0}^{\alpha}xdx + 2\int_{\alpha}^{\infty}x\Prob(B \geq x)dx\\
&\leq 2\alpha^2 + 4\mathcal{N}_T\int_{\alpha}^{\infty}x\exp\left(-{\log \mathcal{N}_T \over \alpha}x\right)dx\\
&\leq \alpha^2 + 4\mathcal{N}_T\left({\alpha^2 \over \log \mathcal{N}_T} + {\alpha^2 \over \log^2 \mathcal{N}_T}\right)\exp\left(-{\log \mathcal{N}_T \over \alpha}\alpha\right)\\
&\leq 1.44\gamma^2a_T\mu_T \log \mathcal{N}_T \leq 9a_T\mu_T \log \mathcal{N}_T. 
\end{align*}

Combining (\ref{indep-block-bound2}) and (\ref{T-expect-bound}), we have
\begin{align*}
\Ep\left[\sum_{\ell = 0}^{\mu_T-1}g_{J,\ell}^{\ast}\right] &\leq (1+\varepsilon)\Ep\left[\sum_{\ell = 0}^{\mu_T-1}g_{J,\ell}\right] + 12(1+\varepsilon)F^2a_T \log \mathcal{N}_T + {9(1+\varepsilon)^2 \over \varepsilon}F^2a_T\log \mathcal{N}_T\\
&\leq (1+\varepsilon)\Ep\left[\sum_{\ell = 0}^{\mu_T-1}g_{J,\ell}\right] + {21(1+\epsilon)^2 \over \varepsilon}F^2a_T\log \mathcal{N}_T,
\end{align*}
where the last inequality follows from $1 \leq (1+\varepsilon)/\varepsilon$. 

\underline{Step 3} (Conclusion) From (\ref{expect-block2}) and (\ref{indep-block-sum-bound}), we have
\begin{align}
\Ep\left[\sum_{\ell = 0}^{\mu_T-1}\tilde{g}_{J,\ell}^{\ast}\right] &\leq \Ep\left[\sum_{\ell = 0}^{\mu_T-1}g_{J,\ell}^{\ast}\right] + 4F^2a_T\mu_T\beta(a_T) \nonumber \\
&\leq (1+\varepsilon)\Ep\left[\sum_{\ell = 0}^{\mu_T-1}g_{J,\ell}\right] + {21(1+\epsilon)^2 \over \varepsilon}F^2a_T\log \mathcal{N}_T + 4F^2a_T\mu_T\beta(a_T) \nonumber \\
&\leq (1+\varepsilon)\Ep\left[\sum_{\ell = 0}^{\mu_T-1}\tilde{g}_{J,\ell}\right] + {21(1+\epsilon)^2 \over \varepsilon}F^2a_T\log \mathcal{N}_T + 4(2+\varepsilon)F^2a_T\mu_T\beta(a_T). \label{pred-error-bound1}
\end{align}
Additionally, define
\begin{align*}
\tilde{h}_{J,\ell} &:= \sum_{t \in I_{2,\ell}}(f_{J}(X_t) - f_0(X_t))^2,\ \tilde{h}_{J,\ell}^{\ast} := \sum_{t \in I_{2,\ell}}(f_{J}(X_t^{\ast}) - f_0(X_t^{\ast}))^2.
\end{align*}
Then a similar argument to derive (\ref{pred-error-bound1}) yields 
\begin{align}\label{pred-error-bound2}
\Ep\left[\sum_{\ell = 0}^{\mu_T-1}\tilde{h}_{J,\ell}^{\ast}\right] &\leq (1+\varepsilon)\Ep\left[\sum_{\ell = 0}^{\mu_T-1}\tilde{h}_{J,\ell}\right] + {21(1+\epsilon)^2 \over \varepsilon}F^2a_T\log \mathcal{N}_T + 4(2+\varepsilon)F^2a_T\mu_T\beta(a_T).
\end{align}
Now, note that
\begin{align*}
R(f_{J}, f_0) &= {1 \over T}\Ep\left[\sum_{\ell = 0}^{\mu_T-1}\tilde{g}_{J,\ell}^{\ast} + \sum_{\ell = 0}^{\mu_T-1}\tilde{h}_{J,\ell}^{\ast} + \sum_{t = 2a_T\mu_T+1}^{T}(f_{J}(X_t^{\ast}) - f_0(X_t^{\ast}))^2\right],\\
\widehat{R}(f_{J}, f_0) &= {1 \over T}\Ep\left[\sum_{\ell = 0}^{\mu_T-1}\tilde{g}_{J,\ell} + \sum_{\ell = 0}^{\mu_T-1}\tilde{h}_{J,\ell} + \sum_{t = 2a_T\mu_T+1}^{T}(f_{J}(X_t) - f_0(X_t))^2\right].
\end{align*}
Together with (\ref{pred-error-bound1}) and (\ref{pred-error-bound2}), we have
\begin{align*}
R(f_{J}, f_0) &\leq {1 \over T}\left\{(1+\varepsilon)\Ep\left[\sum_{\ell = 0}^{\mu_T-1}\tilde{g}_{J,\ell} + \sum_{\ell = 0}^{\mu_T-1}\tilde{h}_{J,\ell}\right] + \Ep\left[\sum_{t = 2a_T\mu_T+1}^{T}(f_{J}(X_t^{\ast}) - f_0(X_t^{\ast}))^2\right]\right\}\\
&\quad + {2 \over T}\left\{{21(1+\epsilon)^2 \over \varepsilon}F^2a_T\log \mathcal{N}_T + 4(2+\varepsilon)F^2a_T\mu_T\beta(a_T)\right\}\\
&\leq (1+\varepsilon)\widehat{R}(f_J,f_0) + {21(1+\varepsilon)^2 \over \varepsilon}{F^2 \log \mathcal{N}_T \over \mu_T} + 4(2+\varepsilon)F^2\beta(a_T)\\
&\quad + {1 \over T}\Ep\left[\sum_{t = 2a_T\mu_T+1}^{T}\left\{(f_{J}(X_t^{\ast}) - f_0(X_t^{\ast}))^2 - (f_{J}(X_t) - f_0(X_t))^2\right\}\right].
\end{align*}
Since 
\begin{align}
\left|(f_J(X_t) - f_0(X_t))^2 - (f_J(X_t^{\ast}) - f_0(X_t^{\ast}))^2\right| &\leq 4F^2,\nonumber \\
\left|(\hat{f}_T(x) - f_0(x))^2 - (f_J(x) - f_0(x))^2\right| &= |\hat{f}_T(x) - f_J(x)||\hat{f}_T(x)+f_J(x) - 2f_0(x)| \nonumber \\
&\leq 4F\delta, \label{fhat-f-bound}
\end{align}
we have
\begin{align*}
R(\hat{f}_T, f_0) &\leq 4F\delta + (1+\varepsilon)\left\{\widehat{R}(\hat{f}_T,f_0) + 4F\delta\right\} + {21(1+\varepsilon)^2 \over \varepsilon}{F^2 \log \mathcal{N}_T \over \mu_T}\\
&\quad + 4(2+\varepsilon)F^2\beta(a_T) + {4F^2 \cdot 2a_T \over T}\\
&\leq (1+\varepsilon)\widehat{R}(\hat{f}_T,f_0) + {21(1+\varepsilon)^2 \over \varepsilon}{F^2 \log \mathcal{N}_T \over \mu_T} + {4F^2 \over \mu_T} + 4(2+\varepsilon)F^2\beta(a_T) + 4(2+\varepsilon)F\delta.
\end{align*}
\end{proof}

\begin{lemma}\label{lem: emp-error-bound}
Let $\{(Y_t, X_t)\}_{t=1}^{T}$ be a time series satisfying (2.1), and set $f_0:= m\mathbf{1}_{[0,1]^{d}}$. Also, let $\delta > 0$ and assume $\mathcal{N}_T := N(\delta, \mathcal{F}, \|\cdot\|_{\infty})<\infty$. Suppose that there is a number $F \geq 1$ such that $\|f\|_{\infty} \leq F$ for all $f \in \mathcal{F} \cup \{f_0\}$. Suppose also that $\supp(f) \subset [0, 1]^{d}$ for all $f \in \mathcal{F}$. Then, under Assumption \ref{Ass: model}, for all $\varepsilon \in (0, 1)$ there exists a constant $C_{ \varepsilon}$ depending only on $(C_\eta, \varepsilon,K)$ such that
\begin{align*}
\widehat{R}(\hat{f}_{T},f_0) &\leq {1 \over 1-\varepsilon}\Psi_T^{\mathcal{F}}(\hat{f}_{T}) + {1 \over 1-\varepsilon}\inf_{f \in \mathcal{F}}R(f,f_0) + C_{\varepsilon}F^2\gamma_{\delta,T},
\end{align*}
where 
\[
\gamma_{\delta,T} := \delta +  {(\log T)(\log \mathcal{N}_T) \over T} + {1 \over T}. 
\]
\end{lemma}

\begin{proof}
Let $\{f_{1},\dots, f_{\mathcal{N}_T}\}$ be a $\delta$-covering of $\mathcal{F}$ with respect to $\|\cdot\|_{\infty}$ and define a random variable $J$ taking values in $\{1,\dots, \mathcal{N}_T\}$ such that $\|\hat{f}_{T} - f_{J}\|_{\infty} \leq \delta$.  

\underline{Step 1} In this step, we will show that for any $\bar{f} \in \mathcal{F}$, 
\begin{align}\label{est-error-eq}
\widehat{R}(\hat{f}_T, f_0) &= \Psi_{T}(\hat{f}_T, \bar{f}) + \widehat{R}(\bar{f}, f_0) + 2\Ep\left[{1 \over T}\sum_{t=1}^{T}(\hat{f}_T(X_t) - f_0(X_t))\eta(X_t)v_t\right], 
\end{align}
where $\Psi_T(\hat{f}_T, \bar{f}) = \Ep\left[Q_T(\hat{f}_T) - Q_T(\bar{f})\right]$. As $Y_t = m(X_t) + \eta(X_t)v_t$, we have
\begin{align*}
Y_t^2 - f_0^2(X_t) &= (Y_t - \hat{f}_T(X_t))^2 - (\hat{f}_T(X_t) - f_0(X_t))^2 + 2\hat{f}_T(X_t)(Y_t - f_0(X_t))\\
&= (Y_t - \hat{f}_T(X_t))^2 - (\hat{f}_T(X_t) - f_0(X_t))^2 + 2\hat{f}_T(X_t)\eta(X_t)v_t.
\end{align*}
For the second equation, we have used the fact $\supp(\hat{f}_T) \subset [0,1]^d$. Likewise, 
\begin{align*}
Y_t^2 - f_0^2(X_t) &= (Y_t - \bar{f}(X_t))^2 - (\bar{f}(X_t) - f_0(X_t))^2 + 2\bar{f}(X_t)\eta(X_t)v_t.
\end{align*}
Since 
\[
\Ep\left[(f_0(X_t) - \bar{f}(X_t))\eta(X_t)v_t\right] = \Ep\left[(f_0(X_t) - \bar{f}(X_t))\eta(X_t)\Ep\left[v_t|\mathcal{G}_{t-1}\right]\right]=0, 
\]
we have 
\begin{align*}
\widehat{R}(\hat{f}_T, f_0) &= \Psi_T(\hat{f}_T, \bar{f}) + \widehat{R}(\bar{f},f_0) + 2\Ep\left[{1 \over T}\sum_{t=1}^{T}(\hat{f}_T(X_t) - f_0(X_t))\eta(X_t)v_t\right]\\
&\quad + 2\Ep\left[{1 \over T}\sum_{t=1}^{T}(f_0(X_t)-\bar{f}(X_t))\eta(X_t)v_t\right]\\
&= \Psi_T(\hat{f}_T, \bar{f}) + \widehat{R}(\bar{f},f_0) + 2\Ep\left[{1 \over T}\sum_{t=1}^{T}(\hat{f}_T(X_t) - f_0(X_t))\eta(X_t)v_t\right].
\end{align*}

\underline{Step 2} In this step, we will show
\begin{align}\label{est-error-bound1}
&\Ep\left[\left|{1 \over T}\sum_{t=1}^{T}(f_J(X_t) - f_0(X_t))\eta(X_t)v_t\right|\right] \nonumber \\ 
&\quad \leq 4C_\eta K\left({(\log T)(\log \mathcal{N}_T + {1 \over 2}\log 2) \over T}\right)^{1/2}\left(\widehat{R}(\hat{f}_T,f_0) + F\delta + {4F^2 \over T}\right)^{1/2}. 
\end{align}
For every $j=1,\dots,\mathcal N_T$, define 
\begin{align*}
A_j &:= \sum_{t=1}^{T}(f_j(X_t) - f_0(X_t))\eta(X_t)v_t,\\
B_j&:= \left(\sum_{t=1}^{T}(f_j(X_t) - f_0(X_t))^2\eta^2(X_t)(v_t^2+1)\right)^{1/2}
\end{align*}
and
\ba{
\xi_{j} &:= {A_j \over 2\sqrt{B_j^2 + \Ep[B_J^2]}},
}
where $\xi_{j}:=0$ if the denominator equals $0$. 
By the Cauchy-Schwarz inequality, 
\begin{align}
\Ep\left[\left|{1 \over T}\sum_{t=1}^{T}(f_J(X_t) - f_0(X_t))\eta(X_t)v_t\right|\right] &\leq {2 \over T}\Ep\left[\left|\xi_J\right|\left(B_J^2 + \Ep[B_J^2]\right)^{1/2}\right] \nonumber \\
&\leq {2 \over T}\Ep\left[\xi_J^2\right]^{1/2}\left(2\Ep[B_J^2]\right)^{1/2}. \label{est-error-bound11}
\end{align}
If $\E[B_J^2]>0$, from Assumption 3.1, Lemmas \ref{lem: mar-bound} and \ref{lem: exp-bound} with $y=\sqrt{\Ep[B_J^2]}$, we have
\[
\Ep\left[{\sqrt{\Ep[B_J^2]} \over \sqrt{B_j^2 + \Ep[B_J^2]}}\exp\left(2\xi_j^2\right)\right]  \leq 1.
\]
Hence
\ba{
\Ep\left[{\sqrt{\Ep[B_J^2]} \over \sqrt{B_J^2 + \Ep[B_J^2]}}\exp\left(2\xi_J^2\right)\right]
\leq\Ep\left[\max_{1\leq j\leq\mathcal N_T}{\sqrt{\Ep[B_J^2]} \over \sqrt{B_j^2 + \Ep[B_J^2]}}\exp\left(2\xi_j^2\right)\right]
\leq\mathcal N_T. 
}
Therefore, by the Cauchy-Schwarz inequality, we obtain
\ba{
\E[\exp(\xi_J^2)]
\leq\sqrt{\Ep\left[{\sqrt{\Ep[B_J^2]} \over \sqrt{B_J^2 + \Ep[B_J^2]}}\exp\left(2\xi_J^2\right)\right]\Ep\left[{\sqrt{B_J^2 + \Ep[B_J^2]}\over\sqrt{\Ep[B_J^2]}}\right]}
\leq\sqrt{\mathcal N_T\Ep\left[{\sqrt{B_J^2 + \Ep[B_J^2]}\over\sqrt{\Ep[B_J^2]}}\right]}.
}
Since
\ba{
\Ep\left[{\sqrt{B_J^2 + \Ep[B_J^2]}\over\sqrt{\Ep[B_J^2]}}\right]
\leq\sqrt{\Ep\left[{B_J^2 + \Ep[B_J^2]\over \Ep[B_J^2]}\right]}
\leq\sqrt 2,
}
we conclude $\Ep[\exp(\xi_J^2)] \leq 2^{1/4}\sqrt\mathcal N_T$. This inequality also holds when $\Ep[B_J^2]=0$ because we have $\xi_J=0$ in such a case. 
Then by Jensen's inequality, 
\begin{align}\label{xi-max-bound}
\Ep\left[\xi_J^2\right]
\leq\log\Ep\left[\exp(\xi_J^2)\right]
 \leq \frac{1}{2}\log \mathcal{N}_T + {1 \over 4}\log 2. 
\end{align}
Using Assumption \ref{Ass: model}, (\ref{fhat-f-bound}) and $\supp(f) \subset [0,1]^{d}$ for all $f \in \mathcal{F}$, 
\begin{align}
\Ep[B_J^2] &\leq C_\eta^2\Ep\left[\sum_{t=1}^{T}(f_J(X_t) - f_0(X_t))^2(v_t^2+1)\right] \nonumber \\
&\leq  C_\eta^2\Ep\left[\sum_{t=1}^{T}(\hat{f}_T(X_t) - f_0(X_t))^2(v_t^2+1)\right] + 4C_\eta^2F\delta\Ep\left[\sum_{t=1}^{T}(v_t^2+1)\right] \nonumber \\
&=  C_\eta^2\Ep\left[\sum_{t=1}^{T}(\hat{f}_T(X_t) - f_0(X_t))^2v_t^2\right] 
+ C_\eta^2\widehat{R}(\hat{f}_T, f_0)+ 8C_\eta^2FT\delta. \label{est-error-bound12}
\end{align}
Decompose
\begin{align}
&\Ep\left[\sum_{t=1}^{T}(\hat{f}_T(X_t) - f_0(X_t))^2v_t^2\right] \nonumber \\
&\leq \Ep\left[\sum_{t=1}^{T}(\hat{f}_T(X_t) - f_0(X_t))^2v_t^2\mathbf{1}_{\{|v_t| \leq K\sqrt{2\log T}\}}\right] + \Ep\left[\sum_{t=1}^{T}(\hat{f}_T(X_t) - f_0(X_t))^2v_t^2\mathbf{1}_{\{|v_t| > K\sqrt{2\log T}\}}\right] \nonumber \\
&\leq 2K^2T(\log T)\widehat{R}(\hat{f}_T, f_0) + 4F^2T\Ep\left[2K^2 \cdot {v_t^2 \over 2K^2}\mathbf{1}_{\{|v_t| > K\sqrt{2\log T}\}}\right] \nonumber \\
&\leq 2K^2T(\log T)\widehat{R}(\hat{f}_T, f_0) + 8F^2K^2T\Ep\left[{v_t^2 \over 2K^2}\mathbf{1}_{\{|v_t| > K\sqrt{2\log T}\}}\right]. \label{est-error-bound13}
\end{align}
Since $v_t$ are sub-Gaussian, we have
\begin{align}
\Ep\left[{v_t^2 \over 2K^2}\mathbf{1}_{\{|v_t| > K\sqrt{2\log T}\}}\right] &\leq \Ep\left[\exp\left({v_t^2 \over 2K^2}\right)\mathbf{1}_{\{|v_t| > K\sqrt{2\log T}\}}\right] \nonumber \\
&\leq \Ep\left[\exp\left({v_t^2 \over K^2}\right)\exp\left(-{v_t^2 \over 2K^2}\right)\mathbf{1}_{\{|v_t| > K\sqrt{2\log T}\}}\right] \nonumber \\
&\leq \Ep\left[\exp\left({v_t^2 \over K^2}\right)\exp\left(-{2K^2\log T \over 2K^2}\right)\right] \leq {2 \over T}. \label{est-error-bound14}
\end{align}
Then by (\ref{est-error-bound12})-(\ref{est-error-bound14}) and \eqref{K-est}, 
\begin{align}\label{BJ-bound1}
\Ep[B_J^2] &\leq 4C_\eta^2K^2T(\log T)\widehat{R}(\hat{f}_T, f_0) + 16C_\eta^2F^2K^2 + 4C_\eta^2FK^2T\delta.
\end{align}
Combining (\ref{est-error-bound11}), (\ref{xi-max-bound}), and (\ref{BJ-bound1}), we have (\ref{est-error-bound1}). 

\underline{Step 3} In this step, we complete the proof. By (\ref{est-error-bound1}) and the AM-GM inequality, 
\begin{align*}
\Ep\left[\left|{1 \over T}\sum_{t=1}^{T}(f_J(X_t) - f_0(X_t))\eta(X_t)v_t\right|\right] &\leq {\varepsilon \over 2} \hat{R}(\hat{f}_T, f_0) + \gamma_{\varepsilon}. 
\end{align*}
where 
\[
\gamma_\varepsilon = {16C_\eta^2K^2(\log T)(2\log \mathcal{N}_T + \log 2) \over \varepsilon T} + {F\delta \varepsilon\over2} + {2F^2\varepsilon \over T}.
\] 
Combining this with (\ref{est-error-eq}), we have for any $\bar{f} \in \mathcal{F}$,
\begin{align*}
\hat{R}(\hat{f}_T, f_0) &= \Psi_T(\hat{f}_T, \bar{f}) + \hat{R}(\bar{f}, f_0) + 2\Ep\left[{1 \over T}\sum_{t=1}^{T}(\hat{f}_T(X_t) - f_J(X_t))\eta(X_t)v_t\right]\\
&\quad + 2\Ep\left[{1 \over T}\sum_{t=1}^{T}(f_J(X_t) - f_0(X_t))\eta(X_t)v_t\right]\\
&\leq \Psi_T^{\mathcal{F}}(\hat{f}_{T}) + \hat{R}(\bar{f}, f_0) + {2C_\eta \delta \over T}\sum_{t=1}^{T}\Ep[|v_t|] + \varepsilon \hat{R}(\hat{f}_T, f_0) + 2\gamma_{\varepsilon}\\
&\leq  \Psi_T^{\mathcal{F}}(\hat{f}_{T}) + \hat{R}(\bar{f}, f_0) + 2C_\eta \delta + \varepsilon \hat{R}(\hat{f}_T, f_0) + 2\gamma_{\varepsilon}.
\end{align*}
Since $\hat{R}(\bar{f}, f_0) = R(\bar{f},f_0)$, we have 
\begin{align*}
(1-\varepsilon)\hat{R}(\hat{f}_T, f_0) &\leq \Psi_T^{\mathcal{F}}(\hat{f}_{T}) + R(\bar{f}, f_0) + \gamma'_{\varepsilon},
\end{align*}
where $\gamma'_{\varepsilon} = 2C_\eta \delta + 2\gamma_{\varepsilon}$. Taking the infimum over $\bar{f} \in \mathcal{F}$, we conclude
\begin{align*}
\hat{R}(\hat{f}_T, f_0) &\leq {1 \over 1-\varepsilon}\Psi_T^{\mathcal{F}}(\hat{f}_{T}) + {1 \over 1-\varepsilon}\inf_{f \in \mathcal{F}}R(\bar{f}, f_0) + {1 \over 1-\varepsilon}\gamma'_{\varepsilon}.
\end{align*}
Noting $F \geq 1$, we obtain the desired result.

\end{proof}

\section{Proofs for Section \ref{sec:ar-model}}\label{Appendix: minimax}

Throughout this section, we write $\|\cdot\|_2=\|\cdot\|_{L^2([0,1]^d)}$ for short. 

\subsection{Proof of Lemma \ref{lemma:beta}}\label{sec:lemma-beta}

The proof is based on Theorem 1.3 in \cite{HaMa11}. 
We begin by introducing some general notation. 
The total variation measure of a signed measure $\mu$ is denoted by $|\mu|$. 
For $x\in\mathbb R^d$, $\delta_x$ denotes the Dirac measure at $x$. 
Given a Markov kernel $\mathcal P$ on $\mathbb R^d$ and a probability measure $\mu$ on $\mathbb R^d$, we define the probability measure $\mu\mathcal P$ on $\mathbb R^d$ by $(\mu\mathcal P)(\cdot)=\int_{\mathbb R^d}\mathcal P(x,\cdot)\mu(dx)$. 
Moreover, we define Markov kernels $\mathcal P^n$, $n=1,2,\dots$, inductively as follows. For $n=1$, we set $\mathcal P^1:=\mathcal P$. For $n\geq2$, we define $\mathcal P^n(x,A):=(\mathcal P^{n-1}(x,\cdot)\mathcal P)(A)$ for $x\in\mathbb R^d$ and a Borel set $A$ in $\mathbb R^d$. 

Next, we rewrite model \eqref{ar-model} to a Markov chain. Let $X_t=(Y_{t-1},\dots,Y_{t-d})'$ and $\bar v_t=(v_t,0,\dots,0)'$ for $t=1,\dots,T$. Define the function $\bar m:\mathbb R^d\to\mathbb R^d$ as
\[
\bar m(x)=(m(x),x_1,\dots,x_{d-1})',\qquad x\in\mathbb R^d.
\]
Then the process $X=\{X_t\}_{t=1}^T$ satisfies 
\ben{\label{ar1-model}
\left\{\begin{array}{l}
X_{t+1}=\bar m(X_t)+\bar v_t,\qquad t=1,\dots,T,\\
X_1\sim\nu.
\end{array}\right.
}
Hence $X$ is a Markov chain. Let $\mathcal P_m$ be the transition kernel associated with $X$. We are going to apply Theorem 1.3 in \cite{HaMa11} to $\mathcal P_m^d$. 

First we check Assumption 1 in \cite{HaMa11} (geometric drift condition). 
Let $b_1=1$. Take positive numbers $b_2,\dots,b_d$ satisfying the following conditions:
\ben{\label{b-cond}
\sum_{j=i}^dc_j<b_i<b_{i-1}-c_{i-1},\qquad i=2,\dots,d.
}
Thanks to the condition $\sum_{i=1}^dc_i<1$, we can indeed take such numbers by induction. 
Then, we define the function $V:\mathbb R^d\to[0,\infty)$ as
\[
V(x)=\sum_{i=1}^db_i|x_i|,\qquad x\in\mathbb R^d.
\]
Denote by $g$ the standard normal density. We have for any $x\in\mathbb R^d$
\ba{
\int_{\mathbb R^d}V(y)\mathcal P_m(x,dy)&=\int_{-\infty}^\infty |y|g(y-m(x))dy+\sum_{i=2}^db_i|x_{i-1}|\\
&\leq m(x)+1+\sum_{i=1}^{d-1}b_{i+1}|x_{i}|\\
&\leq c_0+1+\sum_{i=1}^{d}(c_i+b_{i+1})|x_i|
\leq\gamma V(x)+c_0+1,
}
where $b_{d+1}:=0$ and
\[
\gamma:=\max_{i=1,\dots,d}\frac{c_i+b_{i+1}}{b_i}.
\]
Since $\gamma<1$ by \eqref{b-cond}, we obtain
\begin{equation}\label{eq:lyapunov}
\int_{\mathbb R^d}V(y)\mathcal P_m^d(x,dy)\leq \gamma^dV(x)+(c_0+1)\frac{1-\gamma^d}{1-\gamma}.
\end{equation}
Hence $\mathcal P_m^d$ satisfies Assumption 1 in \cite{HaMa11}. 

Next we check Assumption 2 in \cite{HaMa11} (minorization condition). Set 
\[
R:=\frac{3(c_0+1)}{1-\gamma}\qquad\text{and}\qquad
\mathcal C:=\{x\in\mathbb R^d:V(x)\leq R\}.
\]
Note that $\mathcal C$ is compact. 
A straightforward computation shows that $\mathcal P_m^d$ has the transition density given by
\ba{
p_m(x,y)=\prod_{i=1}^dg(y_i-m(y_{i+1},\dots,y_d,x_1,\dots,x_i)),\qquad x,y\in\mathbb R^d.
}
Then, for any $x,y\in\mathbb R^d$,
\ba{
p_m(x,y)&\geq\frac{1}{(2\pi)^{d/2}}\exp\left(-\sum_{i=1}^d(y_i^2+m(y_{i+1},\dots,y_d,x_1,\dots,x_i)^2)\right)\\
&\geq\frac{1}{(2\pi)^{d/2}}\exp\left(-\sum_{i=1}^dy_i^2-\sum_{i=1}^d\left(c_0+\sum_{j=i+1}^dc_{j-i}|y_j|+\sum_{j=1}^ic_{d-i+j}|x_j|\right)^2\right).
}
Using the Cauchy-Schwarz inequality and $\sum_{i=1}^dc_i<1$, we obtain
\ba{
&\left(c_0+\sum_{j=i+1}^dc_{j-i}|y_j|+\sum_{j=1}^ic_{d-i+j}|x_j|\right)^2\\
&=\left(\sqrt{c_0}\sqrt{c_0}+\sum_{j=i+1}^d\sqrt{c_{j-i}}\sqrt{c_{j-i}}|y_j|+\sum_{j=1}^i\sqrt{c_{d-i+j}}\sqrt{c_{d-i+j}}|x_j|\right)^2\\
&\leq\left(c_0+\sum_{j=i+1}^dc_{j-i}+\sum_{j=1}^ic_{d-i+j}\right)\left(c_0+\sum_{j=i+1}^dc_{j-i}y_j^2+\sum_{j=1}^ic_{d-i+j}x_j^2\right)\\
&\leq(c_0+1)\left(c_0+\sum_{j=i+1}^dc_{j-i}y_j^2+\sum_{j=1}^ic_{d-i+j}x_j^2\right).
}
Hence
\ba{
p_m(x,y)
&\geq\frac{1}{(2\pi)^{d/2}}\exp\left(-\sum_{i=1}^dy_i^2-(c_0+1)\sum_{i=1}^d\left(c_0+\sum_{j=i+1}^dc_{j-i}y_j^2+\sum_{j=1}^ic_{d-i+j}x_j^2\right)\right)\\
&=\frac{1}{(2\pi)^{d/2}}\exp\left(-\sum_{i=1}^dy_i^2-dc_0(c_0+1)-\sum_{j=2}^d\sum_{i=1}^{j-1}c_{j-i}y_j^2-\sum_{j=1}^d\sum_{i=j}^dc_{d-i+j}x_j^2\right)\\
&\geq\frac{1}{(2\pi)^{d/2}}\exp\left(-2|y|^2-dc_0(c_0+1)-|x|^2\right).
}
Therefore, setting
\[
\alpha:=\frac{1}{4^d}\inf_{x\in\mathcal C}\exp(-dc_0(c_0+1)-|x|^2),
\]
we obtain
\begin{equation}\label{eq:minorization}
\inf_{x\in\mathcal C}p_m(x,y)\geq\alpha\varphi(y)\qquad\text{for any }y\in\mathbb R^d,
\end{equation}
where $\varphi$ is the density of the $d$-dimensional normal distribution with mean 0 and covariance matrix $4^{-1}I_d$. This implies that $\mathcal P_m^d$ satisfies Assumption 2 in \cite{HaMa11}. 

Consequently, we have by Theorem 1.3 in \cite{HaMa11}
\begin{equation}\label{eq:contraction}
\rho_\beta(\mu_1\mathcal P_m^d,\mu_2\mathcal P_m^d)\leq\bar\alpha\rho_\beta(\mu_1,\mu_2)
\end{equation}
for any probability measures $\mu_1$ and $\mu_2$ on $\mathbb R^d$, where $\beta>0$ and $\bar\alpha\in(0,1)$ depend only on $\mathbf c$ and $d$, and
\[
\rho_\beta(\mu_1,\mu_2):=\int_{\mathbb R^d}(1+\beta V(x))|\mu_1-\mu_2|(dx).
\]
Now, applying \eqref{eq:contraction} repeatedly, we obtain
\begin{equation}\label{eq:geom}
\rho_\beta(\mu_1\mathcal P_m^{dn},\mu_2\mathcal P_m^{dn})\leq\bar\alpha^n\rho_\beta(\mu_1,\mu_2)\qquad\text{for }n=1,2,\dots.
\end{equation}
Therefore, for any integer $n\geq1$,
\begin{align*}
\beta_X(dn)
&=\sup_{t\geq1}\int_{\mathbb R^d}\|\delta_x\mcl P_m^{dn}-\eta \mcl P_m^{t+dn}\|\nu \mcl P_m^t(dx)\\
&\leq\sup_{t\geq1}\int_{\mathbb R^d}\rho_\beta(\delta_x\mcl P_m^{dn},(\nu P_m^t)P_m^{dn})\nu P_m^t(dx)\\
&\leq\bar\alpha^n\sup_{t\geq1}\int_{\mathbb R^d}\rho_\beta(\delta_x,\nu \mcl P_m^t)\nu \mcl P_m^t(dx)\\
&\leq2\bar\alpha^n\sup_{t\geq0}\left(1+\beta \int_{\mathbb R^d}V(x)\nu \mcl P_m^t(dx)\right)\\
&\leq2\bar\alpha^n\left(1+\beta \int_{\mathbb R^d}V(x)\nu (dx)+\frac{(c_0+1)}{1-\gamma}\right),
\end{align*}
where the first equality follows from \cite[Proposition 1]{Da73} and the last inequality follows from \eqref{eq:lyapunov}, respectively. Finally, note that $\beta_Y(t)\leq\beta_X(t)\leq\beta_X(d\lfloor t/d\rfloor)$ for any $t\geq1$. Thus we complete the proof of \eqref{eq:minimax-beta}.\qed

\subsection{Proof of Proposition \ref{prop:pdnn}}

First, one can easily check that, whenever $t>d$, $Y_t$ has density bounded by 1. Hence
\ben{\label{eq:l2-bound}
R(f,f_0)\leq\frac{4F^2d}{T}+\|f-f_0\|_2^2
}
for all measurable $f:\mathbb R^d\to\mathbb R$ with $\|f\|_\infty\leq F$. 
Next, since $J_T(f)\leq\lambda_T|\theta(f)|_0$, we have
\ba{
\inf_{f \in \mathcal{F}_{\sigma}(L_T,N_T,B_T,F)}\left(R(f,f_0) + J_T(f)\right)
&\leq\inf_{f \in \mathcal{F}_{\sigma}(L_T,N_T,B_T,F,S_{T,\eps_T})}\left(R(f,f_0) + \lambda_T|\theta(f)|_0\right)\\
&\leq\frac{4F^2d}{T}+\frac{C_0}{T^{1/(\kappa+1)}} + \lambda_TS_{T}\\
&\leq\frac{4F^2d+C_0+C_S\iota_\lambda(T)\log^{2+\nu_0+r} T}{T^{1/(\kappa+1)}},
}
where the second inequality follows from \eqref{eq:l2-bound}, $\|f-f_0\|_2^2\leq C_0T^{-1/(\kappa+1)}$, and the definition of $\mathcal{F}_{\sigma}(L_T,N_T,B_T,F,S_{T})$. 
Combining this with Theorem 3.2 and Lemma \ref{lemma:beta} gives the desired result.\qed

\subsection{Proof of Theorem \ref{thm:minimax}}

We begin by reducing the problem to establishing a lower bound on the minimax $L_2$-estimation error. 
\begin{lemma}\label{lem:reduce-to-l2}
Let $\{a_T\}_{T\geq1}$ be a sequence of positive numbers such that $a_T=O(T)$ as $T\to\infty$. Then, there is a constant $\rho>0$ such that
\begin{equation}\label{reduce-to-l2}
\liminf_{T\to\infty}a_T\inf_{\hat f_T}\sup_{m\in\mathcal M}R(\hat{f}_T,f_0)
\geq \rho\liminf_{T\to\infty}a_T\inf_{\hat f_T}\sup_{m\in\mathcal M}\Ep[\|\hat f_T-f_0\|_{2}^2]
\end{equation}
for any $\mcl M\subset\mcl M_0(\mathbf c)$.
\end{lemma}

\begin{proof}
Throughout the proof, we will use the same notation as in Section \ref{sec:lemma-beta}. 
First, by the proof of Lemma \ref{lemma:beta} and \cite[Theorem 3.2]{HaMa11}, $\mcl P_m^d$ has the invariant distribution $\Pi_m$ for all $m\in\mathcal M_0(\mathbf c)$. 
Next, fix an estimator $\hat{f}_T$ arbitrarily. Set $\bar c_0:=c_0+1$ and define 
\[
\tilde f_T:=\left[\{(-\bar c_0)\vee\hat f_T\}\wedge\bar c_0\right]\mathbf1_{[0,1]^d}.
\]
Since $\|f_0\|_\infty\leq \sum_{i=0}^dc_i<\bar c_0$ and $\supp (f_0) \subset[0,1]^d$, we have $|\tilde f_T-f_0|\leq|\hat f_T-f_0|$. Hence
\begin{align*}
R(\hat{f}_T,f_0)&\geq\Ep\left[{1 \over T}\sum_{t=1}^{T}(\tilde{f}_{T}(X_t^{\ast}) - f_0(X_t^{\ast}))^2\right]\\
&= {1 \over T}\sum_{t=1}^{T}\Ep\left[\int_{\mathbb R^d}\left\{\int_{\mathbb R^d}(\tilde{f}_{T}(y) - f_0(y))^2\mcl P_m^t(x,dy)\right\}\nu(dx)\right].
\end{align*}
For any integer $1\leq t_0\leq T$, we have 
\begin{align*}
&\left|{1 \over T}\sum_{t=t_0}^{T}\Ep\left[\int_{\mathbb R^d}\left\{\int_{\mathbb R^d}(\tilde{f}_{T}(y) - f_0(y))^2\mcl P_m^t(x,dy)\right\}\nu(dx)-\int_{\mathbb R^d}|\tilde f_T(y)-f_0(y)|^2\Pi_m(dy)\right]\right|\\
&\leq4\bar c_0^2{1 \over T}\sum_{t=t_0}^{T}\int_{\mathbb R^d}\|\delta_x\mcl P_m^{t}-\Pi_m\|\nu(dx)
\leq4\bar c_0^2{1 \over T}\sum_{t=t_0}^{T}\bar\alpha^{\lfloor t/d\rfloor}\int_{\mathbb R^d}\rho_\beta(\delta_x,\Pi_m)\nu(dx),
\end{align*}
where the last inequality follows from \eqref{eq:geom}. We have
\[
{1 \over T}\sum_{t=t_0}^{T}\bar\alpha^{\lfloor t/d\rfloor}
\leq\frac{1}{T\bar\alpha}\sum_{t=t_0}^{T}\bar\alpha^{t/d}
\leq\frac{\bar\alpha^{t_0/d}}{T\bar\alpha(1-\bar\alpha^{1/d})}
\]
and 
\[
\int_{\mathbb R^d}\rho_\beta(\delta_x,\Pi_m)\nu(dx)
\leq2+\beta\int_{\mathbb R^d}V(x)\nu(dx)+\beta\int_{\mathbb R^d}V(x)\Pi_m(dx).
\]
One can easily derive the following estimate from \eqref{eq:lyapunov} (cf.~\cite[Proposition 4.24]{Ha18}):
\begin{equation}\label{eq:V-mom}
\int_{\mathbb R^d}V(x)\Pi_m(dx)\leq\frac{c_0+1}{1-\gamma}. 
\end{equation}
Combining these estimates, we obtain
\begin{align*}
\left|{1 \over T}\sum_{t=t_0}^{T}\Ep\left[\int_{\mathbb R^d}\left\{\int_{\mathbb R^d}(\tilde{f}_{T}(y) - f_0(y))^2\mcl P_m^t(x,dy)\right\}\nu(dx)-\int_{\mathbb R^d}|\tilde f_T(y)-f_0(y)|^2\Pi_m(dy)\right]\right|
&\leq C_1{\bar\alpha^{t_0/d} \over T},
\end{align*}
where $C_1>0$ depends only on $\mathbf c,d$ and $\nu$. Consequently,
\begin{align*}
R(\hat{f}_T,f_0)
&\geq {1 \over T}\sum_{t=t_0}^{T}\Ep\left[\int_{\mathbb R^d}\left\{\int_{\mathbb R^d}(\tilde{f}_{T}(y) - f_0(y))^2\mcl P_m^t(x,dy)\right\}\nu(dx)\right]\\
&\geq\frac{T-t_0+1}{T}\Ep\left[\int_{\mathbb R^d}|\tilde f_T(y)-f_0(y)|^2\Pi_m(dy)\right]-C_1{\bar\alpha^{t_0/d} \over T}.
\end{align*}
Hence 
\begin{align*}
\sup_{m\in\mathcal M}R(\hat{f}_T,f_0)
&\geq\frac{T-t_0+1}{T}\sup_{m\in\mathcal M}\Ep\left[\int_{\mathbb R^d}|\tilde f_T(y)-f_0(y)|^2\Pi_m(dy)\right]-C_1{\bar\alpha^{t_0/d} \over T}\\
&\geq\frac{T-t_0+1}{T}\inf_{\hat f_T}\sup_{m\in\mathcal M}\Ep\left[\int_{\mathbb R^d}|\hat f_T(y)-f_0(y)|^2\Pi_m(dy)\right]-C_1{\bar\alpha^{t_0/d} \over T},
\end{align*}
where the last infimum is taken over all estimators $\hat f_T$ (possibly different from the one fixed at the beginning of the proof), and the last inequality holds because $\tilde f_T$ itself is an estimator. 
Now, choosing $t_0=\lfloor \sqrt T\rfloor$ and noting $a_T=O(T)$, we obtain
\ben{\label{risk-to-l1}
\liminf_{T\to\infty}a_T\inf_{\hat f_T}\sup_{m\in\mathcal M}R(\hat{f}_T,f_0)
\geq\liminf_{T\to\infty}a_T\inf_{\hat f_T}\sup_{m\in\mathcal M}\Ep\left[\int_{\mathbb R^d}|\hat f_T(y)-f_0(y)|^2\Pi_m(dy)\right].
}
Now, using the definition of $\Pi_m$, we can easily check that $\Pi_m$ has the density given by
\[
\pi_m(y)=\int_{\mathbb R^d}p_m(x,y)\Pi_m(dx),\qquad y\in\mathbb R^d.
\]
We have by \eqref{eq:minorization}
\[
\inf_{y\in[0,1]^d}\pi_m(y)\geq\alpha\inf_{y\in[0,1]^d}\varphi(y)\Pi_m(\mathcal C).
\]
By Markov's inequality and \eqref{eq:V-mom}, we obtain
\begin{align*}
1-\Pi_m(\mathcal C)
=\Pi_m(V>R)
\leq\frac{1}{R}\int_{\mathbb R^d}V(x)\Pi_m(dx)
\leq\frac{1}{3}.
\end{align*}
Hence we conclude
\[
\inf_{y\in[0,1]^d}\pi_m(y)\geq\frac{\alpha}{3}\inf_{y\in[0,1]^d}\varphi(y).
\]
Consequently, there is a constant $\rho>0$ depending only on $\mathbf c,d$ and $\nu$ such that
\begin{equation}\label{pi-bound}
\Ep\left[\int_{[0,1]^d}|\hat f_T(y)-f_0(y)|^2\Pi_m(dy)\right]\geq \rho\Ep\left[\int_{[0,1]^d}|\hat f_T(y)-f_0(y)|^2dy\right]
\end{equation}
for any estimator $\hat f_T$. 
Combining \eqref{risk-to-l1} with \eqref{pi-bound} gives the desired result.
\end{proof}

\begin{proof}[Proof of Theorem \ref{thm:minimax}]
We write $\mcl M_A=\mathcal M\big(\mathbf c, q, \mathbf d, \mathbf t, \bol\beta, A\big)$ for short. 
For each $m\in \mathcal M_A$ and $T\in\mathbb N$, we denote by $P_{m,T}$ the law of the random vector $\mathbf Y_T:=(Y_{-d+1},\dots,Y_T)'$ when $Y_t$ are defined by \eqref{ar-model}. Moreover, we denote by $\Ep_{m,T}[\cdot]$ the expectation under $P_{m,T}$. 

Now, note that any estimator based on the observation $\{Y_t\}_{t=1}^T$ is also an estimator based on $\mathbf Y_T$. 
Therefore, according to Theorem 2.7 in \cite{Ts09} and Lemma \ref{lem:reduce-to-l2}, it suffices to show that there is a constant $A>0$ having the following property: For sufficiently large $T\in\mathbb N$, there are an integer $M\geq1$ and functions $m_0,m_1,\dots,m_M\in\mathcal{M}_A$ such that
\begin{equation}\label{tsy-eq1}
\|m_{j}-m_{k}\|_2^2\geq \kappa\phi_T\quad\text{for all }0\leq j<k\leq M
\end{equation}
and
\begin{equation}\label{tsy-eq2}
P_{j}\ll P_{0}\quad\text{for all }j=1,\dots,M
\end{equation}
and
\begin{equation}\label{tsy-eq3}
\frac{1}{M}\sum_{j=1}^M\Ep_{m_j,T}\left[\log\frac{dP_{j}}{dP_{0}}\right]\leq \frac{1}{9}\log M,
\end{equation}
where $\kappa>0$ is a constant independent of $T$ and $P_j:=P_{m_j,T}$ for $j=0,1,\dots,M$. 

By the proof of \cite[Theorem 3]{Sc20}, there is a constant $A>0$ having the following property: For any $T\in\mathbb N$, there are an integer $M\geq1$ and functions $f_{(0)},\dots,f_{(M)}\in\mathcal G\big(q, \mathbf d, \mathbf t, \bol\beta, A\big)$ satisfying the following condition:
\begin{enumerate}[label=($\star$)]

\item\label{f-l2} For all $0\leq j<k\leq M$, 
\begin{equation}\label{f-upper}
T\|f_{(j)}-f_{(k)}\|_2^2\leq\frac{\log M}{9}
\end{equation}
and
\begin{equation}\label{f-lower}
\|f_{(j)}-f_{(k)}\|_2^2\geq\kappa\phi_T,
\end{equation}
where $\kappa>0$ is a constant depending only on $\mathbf t$ and $\bol\beta$.

\end{enumerate}
For each $j=1,\dots,M$, we define the function $m_j:\mathbb R^d\to\mathbb R$ as 
\[
m_j(x)=\begin{cases}
f_{(j)}(x) & \text{if }x\in[0,1]^d,\\
0 & \text{otherwise}.
\end{cases}
\]
It is evident that $m_0,m_1,\dots,m_M\in\mathcal{M}_A$ when $c_0\geq A$.  
In the following we show that these $m_j$ satisfy \eqref{tsy-eq1}--\eqref{tsy-eq3}.

First, \eqref{tsy-eq1} immediately follows from \eqref{f-lower}. 
Next, it is straightforward to check \eqref{tsy-eq2} and
\ba{
\frac{dP_j}{dP_0}(\mathbf Y_T)=\prod_{t=1}^T\frac{g(Y_t-m_j(Y_{t-1},\dots,Y_{t-d}))}{g(Y_t-m_0(Y_{t-1},\dots,Y_{t-d}))}
}
for every $j=1,\dots,d$, where $g$ is the standard normal density.  
Hence, with $X_t=(Y_{t-1},\dots,Y_{t-d})'$,
\ba{
\Ep_{m_j,T}\left[\log\frac{dP_j}{dP_0}(\mathbf Y_T)\right]
&=\frac{1}{2}\sum_{t=1}^T\Ep_{m_j,T}\left[m_0(X_t)^2-m_j(X_t)^2+2Y_t(m_j(X_t)-m_0(X_t))\right]\\
&=\frac{1}{2}\sum_{t=1}^T\Ep_{m_j,T}\left[(m_j(X_t)-m_0(X_t))^2\right].
}
When $t>d$, conditional on $X_{t-d}$, $X_t$ has the density given by
\[
\prod_{i=1}^dg(y_i-m(y_{i+1},\dots,y_d,Y_{t-d-1},\dots,Y_{t-d-i})),\qquad y\in\mathbb R^d,
\]
which is bounded by 1. Thus
\ba{
\Ep_{m_j,T}\left[\log\frac{dP_j}{dP_0}(\mathbf Y_T)\right]
\leq 2A^2d+\frac{T}{2}\int_{\mathbb R^d}(m_j(x)-m_0(x))^2dx
\leq2A^2d+\frac{\log M}{18},
}
where the last inequality follows from \eqref{f-upper}. Also, by \eqref{f-upper} and \eqref{f-lower}, $\kappa T\phi_T\leq(\log M)/9$. Since $T\phi_T\to\infty$ as $T\to\infty$, we have $2A^2d\leq(\log M)/18$ for sufficiently large $T$. For such $T$, we have \eqref{tsy-eq3}. This completes the proof. 
\end{proof}

\subsection{Proof of Theorem \ref{thm:rate-comp}}

Let $\kappa=\max_{i=0,\dots,q}t_i/(2\beta_i^*)$. 
By the proof of Theorem 1 in \cite{Sc20}, there exist constants $C_0,C_S>0$ such that 
\[
\sup_{m\in\mathcal{M}(\mathbf c, q, \mathbf d, \mathbf t, \bol\beta, A)}\inf_{f\in\mathcal{F}_{\sigma}(L_T,N_T,B_T,F,S_{T})}\|f-m\|_{\infty}^2\leq C_0\phi_T=C_0T^{-1/(\kappa+1)}
\]
with $S_{T}:=C_ST^{\kappa/(\kappa+1)}\log T$. So the desired result follows by Proposition \ref{prop:pdnn}.\qed

\subsection{Proof of Theorem \ref{thm:minimax-l0}}

For each $m\in \mathcal M_0(\mathbf c)$ and $T\in\mathbb N$, we denote by $P_{m,T}$ the law of the random vector $\mathbf Y_T:=\!(\!Y_{-d+1},\dots,Y_T\!)'$ when $Y_t$ are defined by \eqref{ar-model}. Moreover, we denote by $\Ep_{m,T}[\cdot]$ the expectation under $P_{m,T}$. 

Now, note that any estimator based on the observation $\{Y_t\}_{t=1}^T$ is also an estimator based on $\mathbf Y_T$. 
Therefore, by Lemma \ref{lem:reduce-to-l2}, it suffices to prove
\[
\liminf_{T\to\infty}T\inf_{\hat f_T}\sup_{m\in\mcl{M}_\Phi^0(\mathbf c,n_s,C)}\E[\|\hat f_T-f_0\|_2^2]>0.
\]

For each $T=1,2,\dots,$ define $m^\pm_T:=\pm\frac{1}{2\sqrt T}\varphi$ and write $P_{\pm}=P_{m^\pm_T,T}$. 
Note that $m^\pm_T\in\mcl{M}_\Phi^0(\mathbf c,n_s,C)$ by assumption. 
Also, by construction,
\[
\|m^+_T-m^-_T\|_2=\frac{1}{\sqrt T}\|\varphi\|_2=\frac{1}{\sqrt T}.
\]
Moreover, as in the proof of Theorem \ref{thm:minimax}, we can show that $P_{+}\ll P_{-}$ and
\ba{
\Ep_{m^+_T,T}\left[\log\frac{dP_+}{dP_-}(\mathbf Y_T)\right]
&\leq 2c_0^2d+\frac{T}{2}\int_{\mathbb R^d}(m^+_T(x)-m^-_T(x))^2dx
=2c_0^2d+\frac{1}{2},
}
where we used the assumptions $\supp(\varphi)\subset[0,1]^d$ and $\|\varphi\|_\infty\leq c_0$. 
Consequently, by Eq.(2.9) and Theorem 2.2 in \cite{Ts09},
\[
\liminf_{T\to\infty}\inf_{\hat f_T}\sup_{m\in\mcl{M}_\Phi^0(\mathbf c,n_s,C)}P_{m,T}\left(\|\hat f_T-f_0\|_2\geq{1\over 2\sqrt T}\right)>0.
\]
Since
\ba{
&\liminf_{T\to\infty}T\inf_{\hat f_T}\sup_{m\in\mcl{M}_\Phi^0(\mathbf c,n_s,C)}\E[\|\hat f_T-f_0\|_2^2]\\
&\geq\frac{1}{4}\liminf_{T\to\infty}\inf_{\hat f_T}\sup_{m\in\mcl{M}_\Phi^0(\mathbf c,n_s,C)}P_{m,T}\left(\|\hat f_T-f_0\|_2\geq{1\over 2\sqrt T}\right),
}
we complete the proof.\qed

\subsection{Proof of Theorem \ref{thm:rate-l0}}

We are going to apply Proposition \ref{prop:pdnn}. 
Fix $m\in\mcl{M}_\Phi^0(\mathbf c,n_s,C)$ arbitrarily. By definition, $m$ is of the form
\[
m(x)=\sum_{i=1}^{n_s}\theta_i\varphi_i(A_ix-b_i),
\]
where $A_i\in\mathbb R^{d\times d},b_i\in\mathbb R^d,\theta_i\in\mathbb R$ and $\varphi_i\in\Phi$ with $|\det A_i|^{-1}\vee|A_i|_\infty\vee|b_i|_\infty\vee|\theta_i|\leq C$ for $i=1,\dots,n_s$. Since $\Phi\subset\mathrm{AP}_{\relu,d}(C_1,C_2,D,r)$ by assumption, for every $i$, there exist parameters $L_i,N_i,B_i,S_i>0$ such that $L_i\vee N_i\vee S_i\leq C_1(\log_2T)^{r}$ and $B_i\leq C_2T$ hold and there exists an $f_i\in\mathcal{F}_{\relu}(L_i,N_i,B_i)$ such that $|\theta(f_i)|_0\leq S_i$ and $\|f_i-\varphi_i\|_{L^2([-D,D]^d)}^2\leq1/T.$ 
Define
\[
f(x)=\sum_{i=1}^{n_s}\theta_i f_i(A_ix-b_i),\qquad x\in\mathbb R^d.
\]
Then
\ba{
\|f-m\|_{L^2([0,1]^d)}
&\leq C\sum_{i=1}^{n_s}\sqrt{\int_{[0,1]^d}|f_i(A_ix-b_i)-\varphi_i(A_ix-b_i)|^2dx}\\
&\leq C\sum_{i=1}^{n_s}\sqrt{\int_{[-(d+1)C,(d+1)C]^d}|f_i(y)-\varphi_i(y)|^2|\det A_i|^{-1}dx}\\
&\leq C^{3/2}n_sT^{-1/2},
}
where we used the assumption $D\geq(d+1)C$ for the last inequality. 
Also, note that $\|m\|_\infty\leq \sum_{i=0}^dc_i\leq F$ because $m\in\mcl M_0(\mathbf c)$ and $\sum_{i=1}^dc_i<1$. Hence, with $\tilde f=(-F)\vee(f\wedge F)$, we have $|\tilde f-m|\leq|f-m|$. Thus $\|\tilde f-m\|_{L^2([0,1]^d)}^2\leq C^{3}n_s^2/T$. 
Therefore, the proof is completed once we show that there exists a constant $C_S>0$ such that $\tilde f\in\mathcal{F}_{\relu}(L_T,N_T,B_T,F,C_S\log^{r} T)$ for sufficiently large $T$. 

By Lemmas II.3--II.4 and A.8 in \cite{EPGB21}, there exists a constant $C'_S>0$ such that $f\in\mathcal{F}_{\relu}(L_T,N_T,B_T)$ and $|\theta(f)|_0\leq C'_S\log^{r} T$ for sufficiently large $T$. 
Also, note that $x\wedge F=-\relu(F-x)+F$ and $x\vee(-F)=\relu(x+F)-F$ for all $x\in\mathbb R$. Thus we have $\tilde f\in\mathcal{F}_{\relu}(L_T+4,N_T,B_T,F,4C'_S\log^{r} T+20)$ for sufficiently large $T$ by Lemma II.3 in \cite{EPGB21}. 
\qed

\section{Technical tools}\label{Appendix: technical tools}

Here we collect technical tools we used in the proofs. Let $\mathcal{A}$ and $\mathcal{B}$ be two $\sigma$-fields of a probability space $(\Omega, \mathcal{T}, \Prob)$. The $\beta$-mixing coefficient between $\mathcal{A}$ and $\mathcal{B}$ is defined by
\[
\beta(\mathcal{A}, \mathcal{B}) = {1 \over 2}\sup\left\{\sum_{i \in I}\sum_{j \in J}\left|\Prob(A_i \cap B_j) - \Prob(A_i)\Prob(B_j)\right|\right\},
\]
where the maximum is taken over all finite partitions $\{A_i\}_{i \in I} \subset \mathcal{A}$ and $\{B_j\}_{j \in J} \subset \mathcal{B}$ of $\Omega$.  

\begin{lemma}[Lemma 5.1 in \cite{Ri13}]\label{lem: block}
Let $\mathcal{A}$ be a $\sigma$-field in a probability space $(\Omega, \mathcal{T}, \Prob)$ and $X$ be a random variable with values
in some Polish space. Let $U$ be a random variable with uniform distribution over $[0, 1]$, independent of the $\sigma$-field generated by $X$ and $\mathcal{A}$. Then there exists a random variable $X^{\ast}$, with the same law as $X$, independent of $X$, such that $\Prob(X \neq X^{\ast}) = \beta(\mathcal{A}, \sigma(X))$  where $\sigma(X)$ denote the $\sigma$-field generated by $X$. Furthermore $X^{\ast}$ is measurable with respect to the $\sigma$-field generated by $\mathcal{A}$ and $(X, U)$. 
\end{lemma}

\begin{lemma}[Lemma 1.4 in \cite{deKlLa04}]\label{lem: sym-bound}
Let $\{d_i\}$ be a sequence of variables adapted to an increasing sequence of $\sigma$-fields $\{\mathcal{F}_i\}$. Assume that the $d_i$'s are conditionally symmetric (i.e. $\mathcal{L}(d_i|\mathcal{F}_{i-1}) = \mathcal{L}(-d_i|\mathcal{F}_{i-1})$, where $\mathcal{L}(d_i|\mathcal{F}_{i-1})$ is the conditional law of $d_i$ given $\mathcal{F}_{i-1}$). Then $\exp\left(\lambda \sum_{i=1}^{n}d_i - \lambda^2\sum_{i=1}^{n}d_i^2/2\right)$, $n \geq 1$, is a supermartingale with mean $\leq 1$, for all $\lambda \in \mathbb{R}$. 
\end{lemma}

\begin{lemma}\label{lem: mar-bound}
Let $\{d_i\}$ be a martingale difference sequence with respect to a filtration $\{\mathcal{F}_i\}$. Assume $\Ep[d_i^2]<\infty$ for all $i$. Then 
\[
\E\left[\exp\left(\lambda \sum_{i=1}^{n}d_i - \frac{\lambda^2}{2}\left(\sum_{i=1}^{n}d_i^2+\sum_{i=1}^{n}\E[d_i^2\mid\mathcal F_{i-1}]\right)\right)\right]\leq1
\]
for all $n\geq1$ and $\lambda\in\mathbb R$. 
\end{lemma}

\begin{proof}
Define a process $M=\{M_t\}_{t\in[0,\infty)}$ as $M_t=\sum_{i=1}^{\lfloor t\rfloor}\lambda d_i$ for $t\geq0$. It is straightforward to check that $M$ is an $\{\mathcal{F}_{\lfloor t\rfloor}\}$-martingale and its continuous martingale part is identically equal to 0. Moreover, the compensator of the process $\{\sum_{i=1}^{\lfloor t\rfloor}\{(-\lambda d_i)\vee0\}^2\}_{t\geq0}$ is $\{\sum_{i=1}^{\lfloor t\rfloor}\Ep[\{(-\lambda d_i)\vee0\}^2\mid\mathcal F_{i-1}]\}_{t\geq0}$ by Eq.(3.40) of \cite[Ch.~I]{JS03}. 
Therefore, by Proposition 4.2.1 in \cite{BJY86}, the process
\[
\left\{\exp\left(M_t-\frac{1}{2}\left(\sum_{i=1}^{\lfloor t\rfloor}\{(\lambda d_i)\vee0\}^2+\sum_{i=1}^{\lfloor t\rfloor}\E[\{(-\lambda d_i)\vee0\}^2\mid\mathcal F_{i-1}]\right)\right)\right\}_{t\in[0,\infty)} 
\]
is an $\{\mathcal{F}_{\lfloor t\rfloor}\}$-supermartingale. Hence
\ba{
1\geq\E\left[\exp\left(\lambda \sum_{i=1}^{n}d_i-\frac{1}{2}\left(\sum_{i=1}^{n}\{(\lambda d_i)\vee0\}^2+\sum_{i=1}^{n}\E[\{(-\lambda d_i)\vee0\}^2\mid\mathcal F_{i-1}]\right)\right)\right].
}
Since $\{(\lambda d_i)\vee0\}^2\leq\lambda^2d_i^2$ and $\{(-\lambda d_i)\vee0\}^2\leq\lambda^2d_i^2$, the desired result follows from the monotonicity of the exponential function. 
\end{proof}

\begin{lemma}[Theorem 1.2 in \cite{deKlLa04}]\label{lem: exp-bound}
Let $B \geq 0$ and $A$ be two random variables satisfying $\Ep\left[\exp\left(\lambda A - {\lambda^2 \over 2}B^2\right)\right] \leq 1$ for all $\lambda \in \mathbb{R}$. Then for all $y >0$, 
\begin{align*}
\Ep\left[{y \over \sqrt{B^2 + y^2}}\exp\left({A^2 \over 2(B^2 + y^2)}\right)\right] & \leq 1.
\end{align*}
\end{lemma}

\begin{lemma}[Proposition 8 in \cite{OhKi22}]\label{lem: DNN-cov-number-bound}
Let $L \in \mathbb{N}$, $N \in \mathbb{N}$, $B \geq 1$, $F >0$, and $S>0$. Then for any $\delta\in(0,1)$, 
\begin{align*}
\log N\left(\delta, \mathcal{F}_{\sigma}(L,N,B,F,S), \| \cdot \|_{\infty}\right) &\leq 2S(L+1)\log \left((L+1)(N+1)B\delta^{-1}\right). 
\end{align*}
\end{lemma}

\begin{lemma}\label{lem: DNN-cov-number-bound2}
Let $L \in \mathbb{N}$, $N \in \mathbb{N}$, $B \geq 1$, and $F >0$.  
Let 
\begin{align*}
\check{\mathcal{F}}_{\sigma, T}(L,N,B,F,S) &:= \left\{f \in \mathcal{F}_{\sigma}(L,N,B,F): J_T(f) \leq \lambda_TS\right\}.
\end{align*}
Then for any $\delta\in( \tau_T(L+1)((N+1)B)^{L+1},1)$, 
\begin{align*}
\log N\left(\delta, \check{\mathcal{F}}_{\sigma, T}(L,N,B,F,S), \| \cdot \|_{\infty}\right) &\leq 2S(L+1)\log \left({(L+1)(N+1)B \over \delta -\tau_T(L+1)((N+1)B)^{L+1}}\right). 
\end{align*}
\end{lemma}

\begin{proof}
   By the conditions imposed on the function $\pi_{\lambda_T,\tau_T}$, we have $\lambda_T|\theta(f^{(\tau_T)})|_0=J_T(f^{(\tau_T)})\leq J_T(f)$ for any $f \in \mathcal{F}_{\sigma}(L,N,B,F)$, where $f^{(\tau_T)}$ is the DNN such that $\theta_j(f^{(\tau_T)})=\theta_j(f)1_{\{|\theta_j(f)|>\tau_T\}}$ for all $j$. Noting this fact, we can prove the claim in the same way as Proposition 10 in \cite{OhKi22}.
\end{proof}

\bibliography{Ref-DNN-NLTS}
\bibliographystyle{chicago}

\end{document}